\documentclass[11pt]{amsart}
\usepackage{amsmath,amssymb, graphicx, amscd,latexsym,here}
\makeatletter
\newtheorem{Theorem}{Theorem}
\newtheorem{Lemma}[Theorem]{Lemma}
\newtheorem{Corollary}[Theorem]{Corollary}
\newtheorem{Proposition}[Theorem]{Proposition}

\newtheorem{Definition}[Theorem]{Definition}
\newtheorem{Remark}[Theorem]{Remark}

\newtheorem{Example}[Theorem]{Example}

\newtheorem{Observation}[Theorem]{Observation}
\newtheorem{Assertion}[Theorem]{Assertion}

\newcommand{\eps}{\varepsilon}
\newcommand{\ome}{\omega}

\newcommand\la{\lambda}

\newcommand\al{\alpha}

\newcommand\si{\sigma}
\newcommand\be{\beta}
\newcommand\Si{\Sigma}
\newcommand\ga{\gamma}
\newcommand\Ga{\Gamma}
\newcommand\de{\delta}
\newcommand\De{\Delta}

\newcommand\cN{\mathcal  N}

\newcommand\cP{\mathcal  P}

\newcommand\cR{\mathcal  R}

\newcommand\cU{\mathcal  U}

\newcommand\cV{\mathcal  V}

\newcommand\cF{\mathcal F}
\newcommand\BC{ {\mathbb C}}
\newcommand\BN{ {\mathbb  N}}

\newcommand\BZ{{\mathbb  Z}}
\newcommand\BR{ {\mathbb  R}}
\newcommand\BP{ {\mathbb  P}}

\newcommand\bfu{\mbox {\bf  u}}

\newcommand\bfv{\mbox {\bf  v}}

\newcommand\bfx{\mbox {\bf  x}}

\newcommand\bfw{\mbox {\bf  w}}
\newcommand\bfm{\mbox {\bf  m}}
\newcommand\bfn{\mbox {\bf  n}}

\newcommand\bfz{\mbox {\bf  z}}
\newcommand\bfy{\mbox {\bf  y}}
\newcommand\bfa{\mbox {\bf  a}}

\newcommand\bfb{\mbox {\bf  b}}

\newcommand\nl{\newline}
\newcommand\what{\widehat}
\newcommand\tl{\tilde}
\newcommand\wtl{\widetilde}

\newcommand\ord{{\rm{ord}\/}}
\newcommand\Int{\rm{Int}\/}

\newcommand\Ker{{\rm{Ker}\/}}
\newcommand\lkn{{\rm{lkn}\/}}

\newcommand\id{{\rm{id}}}

\newcommand\modulo{\rm{modulo}\/}

\newcommand\rank{\rm{rank}\/}
\newcommand\codim{\rm{codim}\/}

\newcommand\Cone{\rm{Cone}\/}

\newcommand\rdeg{{\rm{rdeg}\/}}
\newcommand\pdeg{{\rm{pdeg}\/}}

\newcommand\inv{^{-1}}

\def\mapright#1{\smash{\mathop{\longrightarrow}\limits^{{#1}}}}

\def\mapdown#1{\Big\downarrow\rlap{$\vcenter{\hbox{$#1$}}$}}

\def\mapup#1{\Big\uparrow\rlap{$\vcenter{\hbox{$#1$}}$}}

\def\inv{^{-1}}

\begin{document}
\title[Non-degenerate mixed functions
]
{Non-degenerate mixed functions
}

\author
[M. Oka ]
{Mutsuo Oka }
\address{\vtop{
\hbox{Department of Mathematics}
\hbox{Tokyo  University of Science}
\hbox{26 Wakamiya-cho, Shinjuku-ku}
\hbox{Tokyo 162-8601}
\hbox{\rm{E-mail}: {\rm oka@rs.kagu.tus.ac.jp}}
}}


\begin{abstract}
Mixed functions are analytic functions in variables
$z_1,\dots, z_n$ and their conjugates $\bar z_1,\dots, \bar z_n$.
We introduce the notion of Newton non-degeneracy for mixed  functions
 and develop a basic tool for the study of mixed hypersurface singularities.
We  show the
 existence of a canonical resolution of the singularity,
and the existence of the  Milnor fibration under the strong
 non-degeneracy condition. 
\end{abstract}
\maketitle

\maketitle

\section{Introduction}
Let $f(\bfz)$ be a holomorphic function of $n$-variables $z_1,\dots,
z_n$ such that $f({\bf 0})=0$. As is well-known,
J. Milnor proved that there exists a positive number $\eps_0$ such that
the argument mapping
$f/|f|:S_\eps^{2n-1}\setminus K_\eps \to S^1$  is a locally trivial
fibration
for any positive $\eps$ with $\eps\le \eps_0$ where
$K_\eps=f\inv(0)\cap S_\eps^{2n-1}$ (\cite{Milnor}). 
In the same book, he proposed to study the links coming from a pair of real-valued real analytic
functions $g(\bfx,\bfy),h(\bfx,\bfy)$ where $\bfz=\bfx+\bfy i$.
Namely putting $f(\bfx,\bfy):=g(\bfx,\bfy)+i\, h(\bfx,\bfy):\BR^{2n}\to
\BC$,
he proposed to study the  condition for 
$f/|f|: S_\eps^{2n-1}\setminus K_\eps\to S^1$ to be a fibration. 
This  is an interesting problem. 
In fact, if one can find such a pair of analytic functions 
$g,h$,  it may give an interesting  link variety $K_\eps$ whose
complement $S_\eps^{2n-1}\setminus K_\eps$ is fibered over $S^1$ where $K_\eps$
cannot come from any complex analytic links.
 The difficulty is that for an arbitrary choice of $g,h$, it is usually
not a fibration.
A breakthrough is given by the work of Ruas, Seade and Verjovsky
\cite{R-S-V}.
After this work, 
many examples of pairs $\{g,h\}$ which  give real Milnor fibrations
have been investigated.
However
in  most papers,  certain restricted types of
functions are mainly considered
(\cite{G-L-M,G-L-M2,SeadeBook,Pichon-Seade,F-P-W,Pichon1,Bodin-Pichon}).

The purpose  of this paper is to propose a wide class of pairs $\{g,h\}$  such that the corresponding 
mapping $f=g+i\,h$ defines a  Milnor fibration.
We consider a complex valued  analytic function  $f$ expanded in a convergent
 power series of variables  $\bfz=(z_1,\dots, z_n)$ and
   $\bar\bfz=(\bar z_1,\dots, \bar z_n)$
$$f(\bfz,\bar \bfz)=\sum_{\nu,\mu}c_{\nu,\mu}\,\bfz^\nu\bar\bfz^\mu$$
where 
$ \bfz^\nu=z_1^{\nu_1}\cdots z_n^{\nu_n}$ for $ \nu=(\nu_1,\dots,\nu_n)$
(respectively
$\bar\bfz^\mu=\bar z_1^{\mu_1}\cdots\bar z_n^{\mu_n}$  for 
$ \mu=(\mu_1,\dots,\mu_n)$)
as usual.
   Here $\bar z_j$ is the complex conjugate of $z_j$. 
 We call $f(\bfz,\bar \bfz)$ {\em a mixed analytic function} (or {\em a mixed polynomial, if
 $f(\bfz,\bar \bfz)$ is a polynomial})
of
$z_1,\dots, z_n$. We are interested in the topology of  the hypersurface
$V=\{\bfz\in \BC^n\,|\, f(\bfz,\bar\bfz)=0\}$,
 which we call {\em a mixed hypersurface}.
Here we use the terminology {\it hypersurface} in order to point out the
 similarity 
with complex analytic hypersurfaces. We will see later that 
 $\codim_{\BR}\,V=2$ 
if $V$ is non-degenerate (Theorem \ref{isolatedness}).
We denote the set of mixed functions of variables $\bfz,\bar\bfz$
by $\BC\{\bfz,\bar\bfz\}$.
This approach is equivalent to the original one.
In fact, writing  $\bfz=\bfx+ i\, \bfy $  with $z_j=x_j+ i\,y_j $
 $j=1,\dots,n$,  and
using real variables $\bfx=(x_1,\dots, x_n)$
and $\bfy=(y_1,\dots, y_n)$, and dividing $f(\bfz,\bar\bfz)$
 in the real and the imaginary parts
so that
$f(\bfx,\bfy)=g(\bfx,\bfy)+ i\,h(\bfx,\bfy)$ where $g:=\Re\, f,\,h:=\Im\,f$,
we can  see
that $V$ is defined by two real-valued 
analytic functions $g(\bfx,\bfy),\,h(\bfx,\bfy)$  of
 $2n$-variables $x_1,y_1,\dots, x_n,y_n$.
Conversely, for a given
real analytic variety $W=\{g(\bfx,\bfy)=h(\bfx,\bfy)=0\}$ which is defined by two real-valued analytic
functions $g,\,h$,
we can consider $W$ as a mixed hypersurface by introducing a mixed function
$f(\bfz,\bar\bfz)=0$ where
\[
 f(\bfz,\bar\bfz):=g(\frac{\bfz+\bar\bfz}2,\frac{\bfz-\bar\bfz}{2\, i})
\,+\,i\,h(\frac{\bfz+\bar\bfz}2,\frac{\bfz-\bar\bfz}{2 \, i}).
\]
The advantage of our view point is that we can use rich techniques of
complex hypersurface singularities.
For complex hypersurfaces defined by  holomorphic functions, the notion of the
non-degeneracy in the sense of the  Newton boundary
 plays an important role for the resolution of
singularities
and the determination of the Milnor fibration
(\cite{Kouchnirenko,Varchenko,ResHyp87,Principal90,Okabook}).
%
We will introduce the notion of {\em non-degeneracy}
 for mixed functions or mixed polynomials
and prove basic properties in \S 2 and \S 3.

In \S 4, we will give a canonical resolution  of
 mixed hypersurface singularities.
First we take an admissible toric modification
$\what \pi: X\to \BC^n$. This does not resolve the singularities but it
turns out that we only need a real modification or a polar modification
after the toric modification
to complete the resolution (Theorem \ref{real-resolution}).

In \S 5, we consider the Milnor fibration of a given mixed function $f(\bfz,\bar\bfz)$.
 It turns out that {\em  the non-degeneracy} is not enough for the existence of
the Milnor fibration of $f$. We need {\em the strong non-degeneracy} of 
 $f(\bfz,\bar\bfz)$  which guarantees the existence of the  Milnor fibration
 (Theorem
 \ref{MilnorFibering1},Theorem \ref{Milnor2}).
We show that the Milnor fibrations of  the first type and of the second type,
\[
 f/|f|: S_\eps\setminus K_\eps\to S^1\quad
\text{and}
\quad
f:\partial E(r,\de)^*\to S_\de^1,
\]
are equivalent (Theorem \ref{equivalence}).
 We also show that for a polar weighted homogeneous polynomial, the global
 fibration
is equivalent to the above two fibrations (Theorem 33).

In \S 6, 
 we will see that the mixed singularities are much more complicated
than the complex singularities
and that the topological equivalence class  is not a  combinatorial invariant
even
in the easiest case of plane curves.

In \S 7, we discuss  Milnor fibrations for non-isolated mixed
singularities under the super strong non-degeneracy condition
(Theorem \ref{NIM}).

In \S 8, we give an A'Campo type formula for the zeta function of the
Milnor fibration in the case of mixed curves (Theorem \ref{MixAC}).

This paper is a continuation of the previous one \cite{OkaPolar}
and we use the same notations. This paper consists of the following
sections.
We hope this paper provides a systematical method to study mixed singularities.

{\small
\begin{center} Contents
\end{center}
\begin{enumerate}
\item[Section 1.] Introduction
\item[Section 2.] Newton boundary and  non-degeneracy of mixed
		 functions
\item[Section 3.] Isolatedness of the singularities
\item[Section 4.] Resolution of the singularities
\item[Section 5.] Milnor fibration
\item[Section 6.] Curves defined by mixed functions
\item[Section 7.] Milnor fibration for mixed polynomials with
  non-isolated singularities
\item[Section 8.] Resolution of a polar type  and the  zeta function
\end{enumerate}
\noindent Below are
notations we use frequently in this paper:
\begin{eqnarray*}\begin{split}
& S_r^{2n-1},\,S_r=\{\bfz=(z_1,\dots, z_n)\in \BC^n\,|\, \|\bfz\|=r\},\,
\text{(sphere of the radius $r$)}\\
&  \|\bfz\|=\sqrt{|z_1|^2+\cdots+|z_n|^2}\\
&B_r^{2n},\,B_r=\{\bfz=(z_1,\dots, z_n)\in \BC^n\,|\, \|\bfz\|\le r\}\,\,
\text{(ball of the radius $r$)}\\
&\BC^I=\{\bfz=(z_1,\dots, z_n)\,|\, z_j=0,\,j\notin I\},
\,
B_r^I=\{\bfz\in \BC^I\,|\, \|\bfz\|\le r\}\\
&\BC^{*I}=\{\bfz=(z_1,\dots, z_n)\,|\, z_j=0 \iff j\notin I\}\\
&\BC^{*n}=\BC^{*I},\,B^{*n}=B^{*I}\, \,\text{with}\,\,I=\{1,\dots,n\}\\
&\BR^{+n}=\{(x_1,\dots,x_n)\in \BR^n\,|\, x_j\ge 0,\,j=1,\dots,n\}\\
&(\bfz,\bfw)=z_1\bar w_1+\cdots+z_n\bar w_n:\,\,\text{hermitian
 inner product}\\
&\Re(\bfz,\bfw)=\Re(z_1\bar w_1+\cdots+z_n\bar w_n):\,\,
\text{real Euclidean inner product}\\
&D(\de):=\{\eta\in \BC\,|\,|\eta|\le \de\},\,\,
D(\de)^*:=\{\eta\in \BC\,|\,0<|\eta|\le \de\}\\
&S_\de^1:=\{\eta\in \BC\,|\,|\eta|=\de\}.
\end{split}
\end{eqnarray*}}
\section{Newton boundary and non-degeneracy of mixed functions}
\subsection{Polar weighted homogeneous polynomials}
\subsubsection{Radial degree and polar degree}
Let $M=\bfz^{\nu}\bar\bfz^\mu$ be a mixed monomial
where $\nu=(\nu_1,\dots, \nu_n)$,  $\mu=(\mu_1,\dots,\mu_n)$ and let
$P={}^t(p_1,\dots, p_n)$ be a weight vector.
We define {\em the radial degree of $M$}, $\rdeg_P\, M$
and {\em the polar degree of $M$}, $\pdeg_P\,M$ with respect to $P$
by 
\[
\rdeg_P\,M= \sum_{j=1}^np_j(\nu_j+\mu_j),\quad
\pdeg_P\,M=\sum_{j=1}^n p_j(\nu_j-\mu_j).
\]
\subsubsection{Weighted homogeneous polynomials}
Recall that a polynomial $h(\bfz)$  is called
{\em a  weighted homogeneous polynomial with weights 
$P={}^t(p_1,\dots, p_n)$ }  if $p_1,\dots, p_n$ are integers and 
 there exists a positive integer $d$ so that 
\[
 f(t^{p_1}z_1,\dots, t^{p_n}z_n)=t^d f(\bfz),\,\, t\in \BC.
\]
The integer
$d$ is called the degree of $f$ with respect to the weight vector $P$.

A mixed polynomial $f(\bfz,\bar\bfz)=\sum_{i=1}^\ell c_i \,\bfz^{\nu_i}\bar\bfz^{\mu_i}$
is called {\em  a  radially weighted homogeneous polynomial }
if 
there exist integers $q_1,\dots, q_n\ge 0$ and $d_r>0$
such that 
 it satisfies the equality:
\[
 f(t^{q_1}z_1,\dots, t^{q_n}z_n,t^{q_1}\bar z_1,\dots, t^{q_n}\bar z_n)=
t^{d_r}f(\bfz,\bar\bfz),\quad t\in \BR^{*}.
\]
Putting $Q={}^t(q_1,\dots, q_n)$, this is equivalent to
$\rdeg_Q\,\bfz^{\nu_i}\bar\bfz^{\mu_i}=d_r$
for $i=1,\dots, \ell$ with $c_i\ne 0$.
Write $f=g+i\, h$ so that $g,\,h$ are polynomials with real coefficients of $2n$-variables
$(x_1,y_1,\dots, x_n,y_n)$.
If $f$ is a  radially weighted homogeneous polynomial of type
$(q_1,\dots,q_n;d_r)$,  $g(\bfx,\bfy)$ and $h(\bfx,\bfy)$ are 
weighted homogeneous polynomials of type
$(q_1,q_1,\dots,q_n,q_n;d_r)$ (i.e., $\deg x_j=\deg\,y_j=q_j$).

A polynomial $f(\bfz,\bar \bfz)$ is called {\em a polar weighted
homogeneous polynomial} if there exists 
a weight vector $(p_1,\dots, p_n)$ and a non-zero integer $d_p$ such that
\begin{eqnarray*}
& f(\la^{p_1}z_1,\dots, \la^{p_n}z_n,\bar\la^{p_1}\bar
 z_1,\dots,\bar\la^{p_n}\bar z_n)=\la^{d_p}f(\bfz,\bar\bfz),\,\la\in \BC^*,\,|\la|=1
\end{eqnarray*}
 where 
 $\gcd(p_1,\dots, p_n)=1$. Usually we assume that $d_p>0$.
This is equivalent to
\[
 \pdeg\,_P\,\bfz^{\nu_i}\bar\bfz^{\mu_i}=d_p,\quad i=1,\dots, \ell.
\]
 Here the weight $p_i$ can be zero or a negative integer.
The weight vector  $(p_1,\dots, p_n)$ 
is  called {\em  the polar weights}  and
$d_p$ is called   the {\em polar degree} respectively.
This notion was first introduced by  Ruas-Seade-Verjovsky {\cite{R-S-V}
and Cisneros-Molina \cite{Molina}.
 In \cite{OkaPolar}, 
we have assumed that a polar weighted homogeneous polynomial
is  also a radially
weighted homogeneous polynomial.
Although  it is not necessary to be assumed, we will only consider such
polynomials
in this paper.

Recall that the radial weights and polar weights define $\BR^*$-action
and $S^1$-action on $\BC^n$ respectively by
\begin{eqnarray*}
& t\circ \bfz=(t^{q_1}z_1,\dots, t^{q_n}z_n),\quad 
t\circ \bar\bfz=(t^{q_1}\bar z_1,\dots, t^{q_n}\bar z_n),\quad
t\in \BR^*
\\
& \la\circ \bfz=(\la^{p_1}z_1,\dots, \la^{p_n}z_n),\quad \la\circ \bar
 \bfz=\overline{\la\circ \bfz},\quad \la\in
 S^1\subset \BC
 \end{eqnarray*}
In other words, this  is an $\BR^*\times S^1$ action on $\BC^n$.
\begin{Lemma}\label{transversality}
Let $f(\bfz,\bar\bfz)$ be a radially weighted homogeneous 
polynomial,
 $V=\{\bfz\in \BC^{n}\,|\, f(\bfz,\bar\bfz)=0\}$ and
 $V^*=V\cap \BC^{*n}$.
Assume  that $V\setminus\{O\}$ (respectively
 $V^*$) is smooth and $\codim_{\BR}V=2$. If the radial weight vector is strictly positive,
namely $q_j> 0$ for any $j=1,\dots, n$,
 the sphere
$S_r$
intersects transversely with $V\setminus\{O\}$ (resp. with  $V^*$ ) for
 any $r>0$.
\end{Lemma}
We are mainly  considering the case that $V\setminus\{O\}$ has no
mixed singularity in the sense of \S \ref{mixed singular}.
\begin{proof}
This is essentially the same with Proposition 4 in \cite{OkaPolar}.
In Proposition 4, we have assumed that $f(\bfz,\bar\bfz)$ is  polar
weighted homogeneous
 but we did not use this assumption in the proof. The radial
action is enough as we will see below.
 Assume that three vectors $dg, \,dh,\, d\phi$
are linearly dependent at $\bfz_0=(\bfx_0,\bfy_0)\in V^*$, where
 $f(\bfz,\bar\bfz)=g(\bfx,\bfy)+ih(\bfx,\bfy)$ and
 $\phi(\bfx,\bfy)=\sum_{j=1}^n (x_j^2+y_j^2)$.
As  $V\setminus\{O\}$ (resp. $V^*$) is non-singular, we can find real numbers $\al,\be$ so that 
$d\phi(\bfx_0,\bfy_0)=\al\,dg(\bfx_0,\bfy_0)+\be\,dh(\bfx_0,\bfy_0)$.
Here $d\phi,\,dg,\,dh$ are the  respective gradient vectors of the functions $\phi,g,h$. For example, 
$dg(\bfx,\bfy)=(\frac{\partial g}{\partial x_1},\frac{\partial g}{\partial y_1}\dots,\frac{\partial g}{\partial x_n},
\frac{\partial g}{\partial y_n})$.
Let $\ell(t)=(t\circ \bfx_0,t\circ \bfy_0),\,t\in \BR^+$ be
 the orbit of $\bfz_0$ by the radial action.
Let $\bfv$ be the tangent vector of the orbit. Then
we have :
\begin{eqnarray*}\begin{split}
&\ell(t)=(t^{q_1}x_{01},t^{q_1}y_{01},\cdots,t^{q_n}x_{0n},t^{q_n}y_{0n})\\
&\frac{d}{dt}\phi(\ell(t))|_{t=1}=\Re(d\phi(\bfx_0,\bfy_0),\bfv)=2\sum_{i=1}^n
 q_i(x_{0i}^2+y_{0i}^2)>0.
\end{split}\end{eqnarray*}
On the other hand, we also have  the equality:
\begin{eqnarray*}\begin{split}
 \frac{d}{dt}\,\phi(\ell(t))|_{t=1}&=\al\,
		  \Re(dg(\bfx_0,\bfy_0),\bfv)+\be\,
\Re (dh(\bfx_0,\bfy_0),\bfv)\\
&=\al\,\frac{dg(\ell(t))}{dt}|_{t=1}+\be\, \frac{dh(\ell(t))}{dt}|_{t=1}=0.
\end{split}
\end{eqnarray*}
This is an obvious  contradiction to the above inequality.
\end{proof}

\subsection{Newton boundary of a mixed function}\label{Newton boundary}
Suppose that we are given a mixed analytic function
$f(\bfz,\bar \bfz)=\sum_{\nu,\mu}c_{\nu,\mu}\,\bfz^\nu\bar\bfz^\mu$.
We always assume that $c_{0,0}=0$ so that $O\in f\inv(0)$.
We call the variety $V=f\inv(0)$ {\em the mixed hypersurface}.
The {\em  radial Newton polygon} $\Ga_+(f;\bfz,\bar\bfz)$ (at the origin) of  a mixed
function $f(\bfz,\bar \bfz)$ is defined  by the
convex hull of
\[
 \bigcup _{c_{\nu,\mu}\ne 0}(\nu+\mu)+\BR^{+n}.
\]
Hereafter we  call  $\Ga_+(f;\bfz,\bar\bfz)$ simply the Newton polygon of $f(\bfz,\bar\bfz)$.
{\em The Newton boundary} $\Ga(f;\bfz,\bar\bfz)$ is 
defined by the union of compact faces
of $\Ga_+(f)$.
Observe that $\Ga(f)$ is nothing but the ordinary Newton boundary if
$f$
is a complex analytic function.
For a given positive integer vector $P=(p_1,\dots, p_n)$, we associate a linear
function
$\ell_P$ on $\Ga(f)$ defined by 
$\ell_P(\nu)=\sum_{j=1}^n p_j\nu_j$ for $\nu\in \Ga(f)$
 and let $\De(P,f)=\De(P)$ be the face where $\ell_P$ takes
its minimal value. In other words, $P$ gives radial weights for variables
$z_1,\dots, z_n$ by $\rdeg_P\,z_j=\rdeg_P\bar z_j=p_j$
and $\rdeg_P\,\bfz ^{\nu}\bar\bfz^{\mu}=\sum_{j=1}^n p_j(\nu_j+\mu_j)$.
To distinguish the points on the Newton boundary and weight vectors, we denote
by $N$ the set of integer weight vectors
and  denote a vector $P\in N$ by a column vectors. We denote by $ N^+,\,N^{++}$ the
subset of positive or   strictly positive weight vectors respectively.
Thus $P={}^t(p_1,\dots, p_n)\in N^{++}$
(respectively $P\in N^+$) if and only if $p_i>0$ (resp. $p_i\ge 0$)
for any $i=1,\dots, n$.
 We denote the minimal value of $\ell_P$ by $d(P;f)$ or simply
$d(P)$. 
Note that 
\[
 d(P;f)=\min\,\{\rdeg_P\, \bfz^{\nu}\bar\bfz^{\mu}\,|\,c_{\nu,\mu}\ne 0\}.
\]
For a positive weight $P$, we define
{\em the face function $f_P(\bfz,\bar\bfz)$} by
\[
 f_P(\bfz,\bar\bfz)=\sum_{\nu+\mu\in \De(P)} c_{\nu,\mu}\, \bfz^{\nu}\bar\bfz^\mu.
\]
\begin{Example}\label{Ex2}{\rm
Consider a mixed function $f:=z_1^3\bar z_1^2+z_1^2z_2^2+z_2^3\bar z_2$.
The Newton boundary $\Ga(f;\bfz,\bar\bfz)$
has two faces $\De_1,\De_2$ which have weight vectors
$ P:={}^t(2,3)$ and  $ Q:={}^t(1,1)$ respectively.
The corresponding invariants are }
\[
 \begin{split}
&f_P(\bfz,\bar\bfz)=z_1^3\bar z_1^2+z_1^2z_2^2,\,\,d(P;f)=10\\
&f_Q(\bfz,\bar\bfz)=z_1^2z_2^2+z_2^3\bar z_2,\,\, d(Q;f)=4.
\end{split}
\]
\begin{figure}[htb,here]
\setlength{\unitlength}{1bp}
\begin{picture}(600,150)(-100,0)
\put(0,130){\special{epsfile=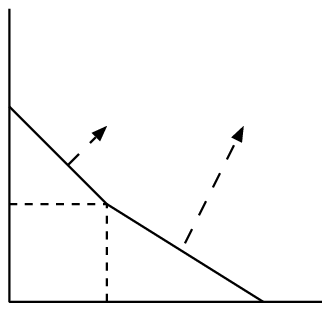 vscale=1 hscale=1}}
\put(30,100){$Q$}
\put (60,100){$P$}
\put (60,60){$\De_1$}
\put (10,95){$\De_2$}
\put (-8,70) {$2$}
\put (-8,105){$4$}
\put (25,30) {$2$}
\put(75,30){$5$}
\end{picture}
\vspace{-1.0cm}
\caption{$\Ga(f)$}
\label{Example2}
\end{figure}
\end{Example}
It is sometimes  important to consider the 
convex hull of vertices  $\what\De(P)$ in $\BR^{n}\times \BR^{n}$
which is defined by
\[
 \what \De(P)=\text{convex hull of}\left\{
(\nu,\mu)\in \BR^{n}\times \BR^{n}
\,|\,c_{\nu,\mu}\ne 0,\,\nu+\mu\in \De(P)
\right\}
\]
Let $S:\BR^{n}\times \BR^{n}\to \BR^{n}$ be the map defined by 
$(\nu,\mu)\mapsto \nu+\mu$. Then $\De(P)=S(\what \De(P))$ by the
definition.
We call $\what \De(P)$ {\em the mixed face of $\Ga(f)$ } and  
$\De(P)$ {\em the radial face of $\Ga(f)$ with respect
to $P$} respectively, when the distinction is necessary.
\subsection{Non-degenerate functions}
Suppose that  $f(\bfz,\bar\bfz)$ is a given 
  mixed function $f(\bfz,\bar\bfz)$.
For  $P\in N^{++}$,
the face function $f_P(\bfz,\bar\bfz)$ is a radially weighted
homogeneous polynomial of type $(p_1,\dots,p_n;d)$ with $d=d(P;f)$.

\begin{Definition}{\rm
Let $P$ be a strictly positive weight vector. We say that 
$f(\bfz,\bar\bfz)$ is 
{\em non-degenerate } 
 for  $P$,
if  
the fiber $f_P\inv(0)\cap \BC^{*n}$ contains no  critical point of the mapping
$f_P:\BC^{*n}\to \BC$.
In particular, $f_P\inv(0)\cap \BC^{*n}$ is  a smooth real codimension
$2$ manifold or an empty set.
We say that $f(\bfz,\bar\bfz)$ is {\em strongly non-degenerate}
for $P$
if  
the mapping   $f_P:\BC^{*n}\to \BC$ has no critical points.
If $\dim\, \De(P)\ge 1$, we further assume that $f_P:\BC^{*n}\to \BC $
 is surjective
onto $\BC$.


A mixed function $f(\bfz,\bar\bfz)$ is called {\em  non-degenerate}
(respectively {\em strongly non-degenerate})
if $f$ is 
non-degenerate (resp. {\em strongly non-degenerate})
for any strictly positive weight vector $P$.

Consider the function
 $f(\bfz,\bar\bfz)=z_1\bar z_1+\cdots+z_n\bar z_n$. Then $V=f\inv(0)$
is a single point $\{O\}$. By the above definition, $f$ is a
 non-degenerate  mixed function. To avoid such an  unpleasant situation,
we say that a mixed function  $g(\bfz,\bar\bfz)$ is {\em a true
 non-degenerate function} if it satisfies further {\em the
non-emptiness  condition}:

$(NE)$ : For any  $P\in N^{++}$ with $\dim\,\De(P,g)\ge 1$,
the fiber $g_P\inv(0)\cap \BC^{*n}$ is non-empty.
}\begin{Remark}{\rm 
Assume that $f(\bfz)$ is a holomorphic function. Then $f_P(\bfz)$
is a weighted homogeneous polynomial and we have the Euler equality:
\[
d(P;f) f_P(\bfz)=\sum_{i=1}^n p_iz_i \frac{\partial f_P}{\partial z_i}(\bfz).
\]
Thus $f_P:\BC^{*n}\to \BC$ has no critical point over $\BC^*$.
Thus $f$ is non-degenerate for $P$ implies $f$ is strongly
 non-degenerate for $P$. This is also the case if $f_P(\bfz,\bar\bfz)$
is a polar weighted homogeneous polynomial.}
\end{Remark}

\begin{Example}{\rm 
{\rm I}. Consider 
the mixed function $f:=z_1^3\bar z_1^2+z_1^2z_2^2+z_2^3\bar z_2$
which we have  considered in Example \ref{Ex2}.
Then $f$ is strongly non-degenerate for each of the  weight vectors
$P={}^t(2,3),\,Q={}^t(1,1)$.

\vspace{.2cm}\noindent
{\rm II}. Consider a mixed function
\[
 g(\bfz,\bar\bfz)=z_1\bar z_1+\dots+z_r\bar z_r-(z_{r+1}\bar
 z_{r+1}+\cdots
+z_n\bar z_n),\quad 1\le r \le n-1.
\]
Then $V=g\inv(0)$ is  a smooth real codimension one variety and thus 
it is degenerate for $P={}^t(1,1,\dots, 1)$.

\vspace{.2cm}\noindent
{\rm III.} Consider  a mixed function
\[
 f(\bfz,\bar\bfz)=z_1^2+a z_1\bar z_2+\bar z_2^2,\quad a\in \BC.
\]
Then $f$ is non-degenerate if and only if $a\ne \pm 2$.

\vspace{.2cm}\noindent
{\rm IV.} Finally we give an example of a mixed function
which is non-degenerate but not
 strongly non-degenerate.
Consider  a mixed function
\begin{eqnarray*}
\begin{split}
f(\bfz,\bar\bfz)
&=1/4\,{{  z_1}}^{2}-1/4\,{{  \bar z_1}}^{2}+{  z_1}\,{  \bar z_1}- \left( 1
+i \right)  \left( {  z_1}+{  z_2} \right)  \left( {  \bar z_1}+{  
\bar z_2 } \right) \\
&=g(x_1,x_2,y_1,y_2)+ih(x_1,x_2,y_1,y_2)\\
\text{where}\,\,\,&g(x_1,x_2,y_1,y_2)={{  x_1}}^{2}+{{  y_1}}^{2}- \left( {  x_1}+{  x_2} \right) ^{2}-
 \left( {  y_1}+{  y_2} \right) ^{2}\\
&h(x_1,x_2,y_1,y_2)={  x_1}\,{  y_1}- \left( {  x_1}+{  x_2} \right) ^{2}- \left( {
  y_1}+{  y_2} \right) ^{2}
\end{split}
\end{eqnarray*}
$f$ is a radially  homogeneous polynomial of degree 2 but it is not polar weighted
homogeneous.
One can check that $f$ is  non-degenerate for the weight vector 
$P={}^t(1,1)$ but it is not strongly non-degenerate.
In fact, it has two families of critical points
\[
 t\mapsto (x_1,x_2,y_1,y_2)=(t,-t,\pm t, \mp t),\,\,0<t
\]
with the critical values $(2\pm i)t^2$.} 
\end{Example}
\end{Definition}
\begin{Proposition}
Let $g(\bfz,\bar\bfz)$ be a radially weighted homogeneous polynomial
and let $M:=\bfz^{\bfa}\bar\bfz^{\bfb}$ be a mixed monomial
and put $h:=Mg(\bfz,\bar\bfz)$.
Then $0$ is a regular value of $g:\BC^{*n}\to \BC$ 
 if and only if $0$ is a regular value of
$h:\BC^{*n}\to \BC$.\end{Proposition}
The assertion is immediate from the definition because 
$g\inv (0)\cap \BC^{*n}=h\inv (0)\cap\BC^{*n} $ and the tangential map
$dh_{\bfw}:T_{\bfw}\BC^{*n}\to T_0\BC$ is equal to $Mdg_{\bfw}$ for any
$\bfw\in g\inv(0)\cap \BC^{*n}$.

Recall that  for a subset $I\subset \{1,\dots,n\}$, we use the notations
$\BC^I=\{\bfz\in \BC^n\,|\,z_j=0,\, j\notin I\}$ and $f^I=f|_{\BC^I}$.
\begin{Proposition}\label{restriction} Assume that $f(\bfz,\bar\bfz)$ is a 
non-degenerate (respectively strongly non-degenerate)
mixed function. Assume that  $f^I$ is not constantly zero for some
$I\subset \{1,2,\dots,n\}$.
Then $f^I$ is a
 non-degenerate (resp. strongly non-degenerate) function  as a function of
 variables $\{z_i,\bar z_i,\,|\, i\in I\}$.
\end{Proposition}
\begin{proof}
The proof is exactly parallel to that of Proposition 1.5, 
\cite{Okabook}.
Take a compact face $\De$ of $\Ga(f^I)$. There is a strictly positive
 weight
vector $P={}^t(p_i)_{i\in I}\in N^I$ such that $\De=\De(P,f^I)$.
We  consider   a strictly positive weight vector
$Q={}^t(q_1,\dots, q_n)$ such that $q_i=p_i$ for $i\in I$ and 
$q_i=\nu$ for $i\notin I$.
It is easy to see that $f_Q(\bfz,\bar\bfz)=f^I_{P}(\bfz_I,\bar\bfz_I)$
if $\nu$ is  sufficiently large.
Here $f_P^I=(f^I)_P$. Now by the assumption, $0$ is not a critical value of
$f_Q:\BC^{*n}\to \BC$ 
(respectively $f_Q:\BC^{*n}\to \BC$
has no critical points). As $f_Q$ contains only 
variables $z_i,\,i\in I$, 
$0$ is not a critical value of $f^I_{P}:\BC^{*I}\to \BC$
(resp.  $f^I_{P}:\BC^{*I}\to \BC$  has no critical points).
\end{proof}

For a complex valued mixed function $f(\bfz,\bar\bfz)$, we use the notation (\cite{OkaPolar}):
\begin{eqnarray*}
&d f(\bfz,\bar\bfz)=(\frac{\partial f}{\partial z_1},\dots,\frac{\partial
 f}{\partial z_n})\in \BC^n,\quad
\bar d f(\bfz,\bar\bfz)=(\frac{\partial f }{\partial\bar z_1},\dots,\frac{\partial
 f}{\partial\bar z_n})\in \BC^n
\end{eqnarray*}
We use  freely the following convenient criterion for
a given point to be a critical point
 as a function to $\BC$ in this
paper.
\begin{Proposition} {\rm (Proposition 1, \cite{OkaPolar})}
\label{critical point}
The following two conditions are equivalent. Let $\bfw\in \BC^n$.
\begin{enumerate}
\item  $\bfw$ is a critical point of $f:\BC^n\to \BC$.
\item There exists a complex 
number $\alpha$ with $|\al|=1$ such that 
$ \overline{df(\bfw,\bar\bfw)}=\al\, \bar d f(\bfw,\bar\bfw)$.
\end{enumerate}
\end{Proposition}
\noindent
Hereafter we use the simplified  notation
$\overline{df}(\bfw,\bar\bfw)$ for $ \overline{df(\bfw,\bar\bfw)}$.
\begin{Example}{\rm
Let us consider the following mixed polynomials
\[
 f_1=z_1\bar z_1-z_2^2,\quad f_2=z_1\bar z_1-z_2\bar z_2,
\quad f_3=z_1^2\bar z_1-z_2^2\bar z_2\\
\]
and the corresponding mixed varieties
$V_i=f_i\inv(0),\,i=1,2,3$. Each of them has an isolated singularity at
the origin.
In fact, as real varieties, they are described as follows.
\begin{eqnarray*}\begin{split}
V_1&=\{(\bfx,\bfy)\,|\, x_1^2+y_2^2=x_2^2-y_2^2,\,x_2y_2=0\}\\
&=\{(\bfx,\bfy)\,|\, x_1^2+y_2^2=x_2^2,\,y_2=0\}\\
V_2&=\{(\bfx,\bfy)\,|\, x_1^2+y_2^2=x_2^2+y_2^2\},\,\dim_{\BR} V_2=3\\
V_3&=\{(z_1,z_2)\,|\,z_1=r_1\exp(i\,\theta_1 ),\,z_2=r_2\exp(i\,\theta_2 ),\, r_1=r_2,\,\theta_1=\theta_2\}\\
&=\{(z_1,z_2)\,|\,z_1-z_2=0\}.
\end{split}
\end{eqnarray*} 
$V_3$ is a special case of polynomials which has been considered 
in \cite{R-S-V}.
$f_1,\,f_3$ are non-degenerate but $f_2$ is a degenerate mixed function
as it is not surjective (onto $\BC$) and 
$\dim_{\BR} V_2=3$.
Note also that $\overline{df_2}=\bar df_2=(z_1,-z_2)$.
 $f_1$ is not  a polar weighted
homogeneous polynomial ( as the monomial $z_1\bar z_1$
can not have a positive degree)  while $f_3$ is a polar
weighted polynomial
of type $(1,1;1).$}\end{Example}
\subsection{Some useful functions}
Let $J$ be a subset of $\{1,\dots,n\}$ and
consider {\em the J-conjugation   map}
$\iota_J:\BC^n\to \BC^n$ defined by:
\[
 \iota_J:\,(z_1,\dots, z_n)\mapsto (w_1,\dots, w_n),\,
w_j=\begin{cases} z_j&\,j\notin J\\
\bar z_j&\, j\in J.
\end{cases}
\] 
Of course, we define $\iota_J(\bar z_j)=\overline{\iota_J(z_j)}$.

Let $f(\bfz,\bar\bfz)$ be a mixed function. 
We call that  $f(\bfz,\bar\bfz)$ is 
{\em  J-conjugate  holomorphic}  if  $f$ is an analytic function 
of the variables
 $\{z_j\,|\,j\notin J\}$ 
and $\{\bar z_k\,|\, k\in J\}$, or equivalently
$f\circ \iota_J(\bfz)$ is a holomorphic function.

A mixed polynomial $f(\bfz,\bar\bfz)$ is called  a {\em J-conjugate
weighted homogeneous polynomial}
if   $f\circ\iota_J(\bfz)$ is a weighted homogeneous
polynomial. Let $P=(p_1,\dots, p_n)$  be the weight vector of 
$f\circ\iota_J(\bfz)$ and let $d$ be the degree.
We say that  $f(\bfz,\bar\bfz)$ is {\em a J-conjugate
weighted homogeneous polynomial of the weight type}
 $(p_1,\dots, p_n;d)$.
The following is obvious by the definition.
\begin{Proposition}
Assume that  $f(\bfz,\bar\bfz)$ is a J-conjugate weighted homogeneous
polynomial of the  weight type $(p_1,\dots, p_n;d)$.
Then $f(\bfz,\bar\bfz)$ is a polar weighted polynomial with
the polar weight type $(\iota_J P;d)$
where 
\[
\iota_J P=(p_1',\dots, p_n'),\quad p_j'=\begin{cases}-p_j\quad &j\in J\\ p_j\quad &j\notin J
\end{cases}
\]
Furthermore
$f(\bfz,\bar\bfz)$ is also a radially weighted homogeneous polynomial
 of the radial weight type $(p_1,\dots, p_n;d)$.
\end{Proposition}

Let $M=\bfz^\nu\bar\bfz^\mu$ be a mixed monomial and let 
$g(\bfz,\bar\bfz)=M\cdot f(\bfz,\bar\bfz)$
where $f(\bfz,\bar\bfz)$  is a $J$-conjugate weighted homogeneous
polynomial.
 We say
$g(\bfz,\bar\bfz)$ is  {\em a pseudo J-conjugate weighted homogeneous
polynomial} if $\pdeg_{P'} g\ne 0$
where $P'=\iota_J P$
is the polar weight vector of $f(\bfz,\bar\bfz)$.
Note that $g\circ\iota_J(\bfz)$ need not to be holomorphic.
Further, if  $J=\emptyset$, we say that $g$  is {\em a pseudo 
 weighted homogeneous
polynomial}. Then $g$ takes the form $M\,f(\bfz)$ where $f$  a weighted
homogeneous
polynomial and $M$ is  a mixed monomial.
\begin{Proposition}\label{non-degeneracy}
Assume that  $f(\bfz,\bar\bfz)$ is a $J$-conjugate  weighted homogeneous
polynomial of  the weight type  $(p_1,\dots, p_n;d)$.
Let  $M=\bfz^\nu\bar\bfz^\mu$ be a monomial and assume that
$g(\bfz,\bar\bfz)=Mf(\bfz,\bar\bfz)$ is a pseudo $J$-conjugate weighted
 homogeneous polynomial, namely $\pdeg_{P'} M+d\ne 0$.
Then $g:\BC^{*n}\to \BC$ has no critical points if and only if
$f:\BC^{*n}\to \BC$ has no critical points.
\end{Proposition}
\begin{proof}
As $g(\bfz,\bar\bfz)$ is a polar weighted polynomial, the only possible
 singular fiber is $g\inv (0)$. Thus the assertion is immediate 
as $g\inv(0)=f\inv(0)$ in $\BC^{*n}$.
\end{proof}
\begin{Example}{\rm
Let $f(\bfz,\bar\bfz)=z_1^{2}+\dots+z_{n-1}^{2}+\bar z_{n}^3$. 
Then $f$ is  a $J$-conjugate weighted homogeneous polynomial
of  the weight type $(3,\dots,3,2;6)$
with $J=\{n\}$.
A mixed polynomial $g(\bfz,\bar\bfz)=\bfz^\nu\bar\bfz^{\mu}f(\bfz,\bar\bfz)$
is a pseudo J-conjugate weighted homogeneous polynomial if
$$3\sum_{j=1}^{n-1}(\nu_i-\mu_i)-2(\nu_n-\mu_n)+6\ne 0.$$
}\end{Example}
\begin{Definition}{\rm 
Let $f(\bfz,\bar\bfz)$ be a mixed function.
We say that $f$ is a {\em  Newton pseudo conjugate  weighted homogeneous
polynomial}
if for any  $P\in N^{++}$,  there exists a subset
 $J(P)\subset \{1,\dots,n\}$ such that 
the face function
$f_P(\bfz,\bar\bfz)$ is a $J(P)$-pseudo conjugate weighted homogeneous 
polynomial.
Here $J(P)$ 
can differ for 
each $P$.
For a Newton pseudo conjugate weighted homogeneous
function, the non-degeneracy condition 
is easily checked by Proposition \ref{non-degeneracy}.}
\end{Definition}
\begin{Example}{\rm 
I. Let $f(\bfz,\bar\bfz)=z_1^5+z_1^2\bar z_2^2+z_2^m\bar z_2^2$ with 
 $m\ge 2$. Then the Newton boundary has two faces and the
corresponding weights are
$P=(2,3)$ and $Q=(m,2)$. The face functions are
\[
 f_P(\bfz,\bar\bfz)=z_1^2(z_1^{3}+\bar z_2^2),\,
f_Q(\bfz,\bar\bfz)=\bar z_2^2(z_1^2+z_2^m)
\]
and
$f$ is a Newton  pseudo conjugate 
weighted homogeneous polynomial if $m\ne 2$.
Note that for $m=2$, the polar degree of $f_Q(\bfz,\bar\bfz)$ is $0$.
See also the next example.
We  give a class of functions which can not be  non-degenerate.

\vspace{.3cm}\noindent
II. Consider the radially weighted homogeneous polynomial
\[
 f(\bfz,\bar\bfz)=\sum_{i=1}^n c_j z_j^{a_j}\bar
 z_j^{a_j},\quad c_1,\dots, c_n\in \BC^*.
\] where $a_1,\dots, a_n$ are positive integers.
This is very special as $ z_j^{a_j}\bar z_j^{a_j}=|z_j|^{2a_j}\ge 0$.
Let $\Omega:=\{\sum_{j=1}^n \al_j c_j\,|\, \al_j>0\}$ be the open cone
of the complex numbers $\BC$
 generated by $c_1,\dots, c_n$.}
\end{Example}
\begin{Proposition}\label{cone}Let $f(\bfz,\bar\bfz)$ be as above.
The image of $f:\BC^{*n}\to \BC$ is $\Omega$ and 
$f$ is  a submersion on $\Omega$.
\end{Proposition}
\begin{proof} As $ z_j^{a_j}\bar z_j^{a_j}>0$ for $z_j\ne 0$, 
$f(\BC^{*n})\subset \Omega$. For an $\eta\in \Omega$, write $\eta$ as
$\eta=\sum_{j=1}^n \al_jc_j$ with $\al_j>0$. Take $w_j$ so that 
$|w_j|^{2a_j}=\al_j$. Then $\bfw=(w_1,\dots, w_n)\in f\inv(\eta)\cap
 \BC^{*n}$.
Thus the image of $f$ is onto $\Omega$.
We identify $ T_{f(\bfw)}\BC$ with $\BC$ by $\al\frac{\partial}{\partial x}+\be \frac{\partial}{\partial y}$ $\leftrightarrow$ $\al+i\,\be $.
Here the coordinates of $\BC$ are $x+ i\,y$.
Then it is easy to see that 
the tangent vector of the $j$-th radial line $r_j(t,\bfw)$ defined by
 $t\mapsto (w_1,\dots, tw_j,\dots,w_n)$
is mapped by $d_{\bfw}f$ to $2a_j|w_j|^{2a_j}c_j$. This
 implies that $f:\BC^{*n}\to \BC$ is a submersion onto $\Omega$.
\end{proof}
\begin{Corollary} Let $f(\bfz,\bar\bfz)=\sum_{i=1}^n c_j z_j^{a_j}\bar
 z_j^{a_j}=\sum_{i=1}^n c_j|z_j|^{2a_j}$  as in Proposition \ref{cone}.
 \begin{enumerate}
\item  If $0\in \Omega$, $V=f\inv(0)\subset \BC^{*n}$ is smooth and
 non-empty.

\item
 $f(\bfz,\bar\bfz)$ is not a true non-degenerate mixed function.
\end{enumerate}
\end{Corollary}
\begin{proof}
The first assertion is immediate from Proposition \ref{cone}.
For the second assertion,
take  any  two dimensional subspace $\BC^I$ of $\BC^n$ with $I=\{i,j\}$.
the open cone $\Omega(c_i,c_j)$ generated by $c_i,\,c_j$ cannot be
the whole 
 $\BC$.
Considering the weight vector $S$ so that $\deg_S z_k=N$, $k\ne i,j$
and $\deg_S z_k=1$ for $k=i,j$, we see that 
$f_S(\bfz,\bar\bfz)=c_i|z|^{2a_i}+c_j|z_j|^{2a_j}$,
as long as $N$ is sufficiently large.
If $\dim_{\BR}\Omega_{c_i,c_j}=2$, it is easy to see that 
$0\notin \Omega_{c_i,c_j}$. Thus 
$(f^{I})\inv(0)\cap \BC^{*I}=\emptyset$.
If $\dim_{\BR}\Omega_{c_i,c_j}=1$, either $0\notin \Omega_{c_i,c_j}$
or $0\in \Omega_{c_i,c_j}$. If $0\notin \Omega_{c_i,c_j}$,
$(f^{I})\inv(0)\cap \BC^{*I}=\emptyset$ as above.
If $0\in \Omega_{c_i,c_j}$, $\arg c_i+\arg c_j=0$ and the real dimension
 of $(f^I)\inv(0)\cap
\BC^{*2}$ is 3 and any point of $(f^I)\inv(0)$
is a critical point. Thus in any case $f^I$ is 
not true non-degenerate.
\end{proof}

\begin{Example}
{\rm Consider 
\[
 \begin{split}
&g(\bfz,\bar\bfz)=\sum_{j=1}^n|z_j|^{2a_j},\quad
h(\bfz,\bar\bfz)=\sum_{j=1}^m
|z_j|^{2a_j}-\sum_{j=m+1}^n|z_j|^{2a_j}
\end{split}
\]
with $1<m<n$.
Then the image of $g$ and $h$ are  the strictly positive  half real line $\{x\ge
0\}$
and  the whole real line $\BR$ respectively
and $g\inv(0)\cap \BC^{*n}=\emptyset$
and $\dim_{\BR}h\inv(0)=2n-1$.
}\end{Example}
\subsection{Pull-back of a polar weighted homogeneous
   polynomial}
Let $\si=(p_{ij})=(P_1,\dots, P_n)$ be a unimodular matrix where 
$P_j={}^t(p_{1j},\dots, p_{nj})$ is the $j$-th
column vector. Consider the toric morphism
\[\begin{split}
& \psi_\si:\BC^n\to \BC^n,\quad
\bfw\mapsto \bfz=(z_{1},\dots, z_{n})\\
&z_{j}=w_1^{p_{j1}}\cdots w_n^{p_{jn}},\quad j=1,\dots,n.
\end{split}
\]
See \S 4.1 for more details. Let $f(\bfz,\bar\bfz)=\sum_{j=1}^m
c_{\mu,\nu}\bfz^{\mu_j}{\bar\bfz}^{\nu_j}$
be a polar weighted homogeneous polynomial of type
$(p_1,\dots, p_n;d_p)$ and let 
$(q_1,\dots, q_n;d_r)$ be the radial weights.
Then they satisfy the equality:
\[
 \sum_{j=1}^n (\mu_j-\nu_j)p_j=d_p,\,\,\,
\sum_{j=1}^n(\nu_j+\nu_j)q_j=d_r,\,j=1,\dots, m
\] where $P={}^t(p_1,\dots, p_n)$ and $Q={}^t(q_1,\dots, q_n)$.
Consider the pull-back
\[
 \psi_\si^*(f)(\bfw,\bar\bfw)=\sum_{j=1}^mc_{\mu,\nu}
\psi_\si^*(\bfz^{\mu_j}{\bar\bfz}^{\nu_j})
=\sum_{j=1}^mc_{\mu,\nu}\bfw^{\mu_j'}{\bar\bfw}^{\nu_j'}
\]
where
$\mu_j'=\mu_j\si,\quad \nu_j'=\nu_j\si$ and $\mu_j,\nu_j$ are considered
raw
vectors.
We define  
$ P':=\si\inv P$. 
Then we see that 
\[\begin{split}
& (\mu_j'+\nu_j')Q'=(\mu_j+\nu_j)\si\si\inv Q=d_r\\
& (\mu_j'-\nu_j')P'=(\mu_j-\nu_j)\si\si\inv P=d_p
\end{split}
\]
for any $j=1,\dots, m$. 
Thus
\begin{Lemma} Let $f(\bfz,\bar\bfz)$ be a polar weighted mixed polynomial of
the radial weight type 
$(q_1,\dots, q_n;d_r)$ and 
of the  polar weight type
$(p_1,\dots, p_n;d_p)$.
Then $g(\bfw,\bar\bfw):=\psi_\si^*f(\bfw,\bar\bfw)$ is also a polar
 weighted homogeneous polynomial.
The  radial weight type and  the  polar weight type are
$(q_1',\dots, q_n';d_r)$ and
 $(p_1',\dots, p_n';d_p)$ respectively
where 
\[
\left(\begin{matrix}q_1'\\
\vdots\\ q_n'
\end{matrix}\right)=\si\inv \left(
\begin{matrix}q_1\\
\vdots\\
 q_n\end{matrix}
\right),\quad
\left(\begin{matrix}p_1'\\
\vdots\\
p_n'
\end{matrix}\right)=\si\inv
\left(\begin{matrix}p_1\\
\vdots\\
 p_n
\end{matrix}
\right).
\]
Two fibrations are isomorphic by $\psi_\si$, using the following
 commutative diagram.
\[
 \begin{matrix}
\BC^{*n}&\mapright{f}&\BC^*\\
\mapup{\psi_\si}&&\mapup{\id}\\
\BC^{*n}&\mapright{g}&\BC^*
\end{matrix}
\]
(The commutativity implies
that  $\psi_\si$ is a fiber preserving diffeomorphism.)
\end{Lemma}
\section{Isolatedness of the singularities}
Let
$f(\bfz,\bar\bfz)=\sum_{\nu,\mu}\,c_{\nu,\mu}\,\bfz^\nu\bar\bfz^\mu$.
As we are mainly interested in the topology of a germ
of a mixed hypersurface  at the origin,
we always assume that $f$ does not have the constant term so that 
$O\in f\inv(0)$. Put $V=f\inv(0)\subset \BC^n$.
\subsection{Mixed singular points}\label{mixed singular}
We say that $\bfw\in V$ is a {\em mixed singular point}
if $\bfw$ is a critical point of the mapping $f:\BC^n\to \BC$.
We say that $V$ is {\em mixed non-singular} if it has no mixed singular
points.
If $V$ is mixed non-singular, $V$ is smooth variety of real codimension two.
Note that a singular point of $V$ (as a point of a real algebraic variety)
is a mixed singular point of $V$
but the converse is not necessarily true.
For example, every point of the sphere $S=\{z_1\bar z_1+\cdots+z_n\bar
z_n=1\}$
is a mixed singular point.
\subsection{Non-vanishing coordinate subspaces}
For a subset $J\subset \{1,2,\dots,n\}$, we consider the subspace
$\BC^J$ and the restriction $f^J:=f|_{\BC^J}$.
Consider the set
\[
 \cN\cV(f)=\{I\subset\{1,\dots,n\}\,|\, f^{I}\not \equiv 0\}.
\]
We call $\cN\cV(f)$ {\em the set of non-vanishing coordinate subspaces
 for $f$}.
Put 
\[
 V^{\sharp}=\bigcup_{I\in \cN\cV(f)}V\cap\BC^{*I}.
\]

\begin{Theorem}\label{isolatedness} Assume that 
$f(\bfz,\bar\bfz)$
is a true non-degenerate  mixed function. Then  
there exists a positive number $r_0$ such that the following properties
are satisfied.
\begin{enumerate}
\item{\rm (Isolatedness of the singularity)}
The mixed hypersurface $V^{\sharp}\cap B_{r_0}$
is mixed non-singular. In particular, $ \codim_{\BR}V^{\sharp}=2$.
\item {\rm (Transversality)} The sphere $S_r$
with $0<r\le r_0$ intersects $V^{\sharp}$ transversely.
\end{enumerate}
\end{Theorem}
\begin{proof}
We prove that the origin is an isolated mixed  singularity.
Or $V^\sharp\cap  B_{r_0}$ has no mixed singularity, if $r$ is
 sufficiently small.
Denote  the mixed singular locus of $V$ by  $\Sigma_m(V) $.
Assume the contrary. Using the Curve Selection Lemma (\cite{Milnor,Hamm1}),
we can find a real analytic curve
$\bfz(t)\in \BC^n,\,0\le t\le 1$ so that
$\bfz(t)\in \Si_m(V)\cap V^{\sharp}$ for $t\ne 0$
and $\bfz(0)=O$. Using Proposition  \ref{critical point}
we can find a real analytic  family $\la(t)$ in $S^1\subset \BC$
such that 
\begin{eqnarray}\label{Sing-cond1}
 \overline{df}(\bfz(t),\bar\bfz(t))=\la(t)\, \bar df(\bfz(t),\bar\bfz(t)).
\end{eqnarray}
Put $I=\{j\,|\, z_j(t)\not \equiv 0\}$. As $\bfz(t)\in V^{\sharp}$,
$I\in \cN\cV(f)$,
the restriction $f^I=f|_{\BC^I}$ is not constantly zero.
 We may assume that
$I=\{1,\dots,m\}$ and we consider $f^I$ and 
the Taylor expansion of  $\bfz(t)$:
\begin{eqnarray*}
 &\bfz_i(t)=b_i t^{a_i}+\text{(higher terms)},\,b_i\ne 0\quad
i=1,\dots, m\\
&\la(t)=\la_0+\la_1 t+\text{(higher terms)},\quad \la_0\in S^1\subset \BC.
\end{eqnarray*}
Put $A=(a_1,\dots, a_m)$ and  we consider the face function $f^I_A$
of $f^I(\bfz,\bar\bfz)$. Let $d=d(A;f^I)>0$ and $\bfb=(b_1,\dots,
 b_m)\in \BC^{*m}$.
Then we have
\begin{eqnarray*}
&\frac{\partial f}{\partial z_j}(\bfz(t),\bar\bfz(t))=\frac{\partial f^I_A}{\partial
 z_j}(\bfb,\bar\bfb)\,t^{d-a_j}+\text{(higher terms)},\quad j=1,\dots,m\\
 &\frac{\partial f}{\partial \bar z_j}(\bfz(t),\bar\bfz(t))=\frac{\partial f^I_A}{\partial
 \bar z_j}(\bfb,\bar\bfb)\,t^{d-a_j}+\text{(higher terms)}\quad j=1,\dots,m.\\
 \end{eqnarray*}
Observe that by the equality (\ref{Sing-cond1}),   we have the following
 equality:
\[
 \ord_t{\frac{\partial f^I}{\partial z_j}
 \,(\bfz(t),\bar\bfz(t))}=\ord_t \,
 \frac{\partial f^I}{\partial \bar z_j}(\bfz(t),\bar\bfz(t)),\,\,j=1,\dots,m.
\]
 Thus by (\ref{Sing-cond1}), we get the equality:
 \[
  \overline{df^I_A}(\bfb,\bar\bfb)=\la_0\, \bar df^I_A(\bfb,\bar\bfb).
 \]
On the other hand, the equality $f^I(\bfz(t))\equiv 0$ implies that 
$ f^I_A(\bfb,\bar\bfb)= 0$.
  This implies that $\bfb\in \BC^{*I}$ is a critical point of 
${f^I_A}:\BC^{*I}\to \BC$, which is a contradiction to the non-degeneracy of 
$f^I(\bfz,\bar\bfz)$.

The second  assertion
 is the result of  a standard argument ( Corollary 2.9, \cite{Milnor}).
\end{proof}

We say that $f$ is {\em $k$-convenient} if 
 $J\in \cN\cV(f)$ for any $J\subset \{1,\dots,n\}$ with $|J|=n-k$.
We say that $f$ is {\em convenient} if $f$ is $(n-1)$-convenient.
Note that $V^\sharp=V\setminus\{O\}$ if $f$ is convenient.
For a given $\ell$ with $0<\ell\le n$, we put 
$ W(\ell)=\{\bfz\in \BC^n\,|\, |I(\bfz)|\le \ell\}$
where $I(\bfz)=\{i|z_i=0\}$.
Thus $W(n-1)=\BC^{*n}$.
If $f$ is $\ell$-convenient, $V\cap W(\ell)\subset V^{\sharp}$.

\begin{Corollary}
Assume that $f(\bfz,\bar\bfz)$ is a  convenient true  non-degenerate
mixed polynomial. Then $V=f\inv(0)$ has an isolated
mixed singularity
at the origin.
\end{Corollary}
\begin{Remark} The assumption ``true'' is to make sure that 
$V^*=f\inv(0)\cap \BC^{*n}$ is non-empty.
\end{Remark}

\section{Resolution of the singularities}
We consider a mixed analytic function $f(\bfz,\bar\bfz)$ and
the corresponding mixed hypersurface $V=f\inv(0)$.
 We assume that $O\in V$
is an isolated mixed singularity, unless otherwise stated.

 If $f$ is complex analytic, a ``resolution of $f$''
is usually understood as a proper  holomorphic mapping
$\varphi: X\to \BC^n$  so that

(i) $E:=\varphi\inv(O)$ is a union of smooth 
(complex analytic) divisors
which  intersect transversely and $\varphi:X-E\to \BC^n-\{O\}$
is biholomorphic,

(ii) the divisor  $(\varphi^*f)$ is a union of smooth 
divisors intersecting  transversely and we can write
 $(\varphi^*f)=\widehat V\cup E$ where 
$\widehat V$ is the  strict transform of $V$ (= the closure of
$\varphi\inv(V-\{O\})$), 

(iii) for any point $P\in E_I^*\cap \what V$ with $I=\{i_1,\dots, i_s\}$, there exists an analytic coordinate chart
$(u_1,\dots, u_n)$ so that the pull-back of $f$
 is written as $U\times u_1^{m_1}\cdots u_j^{m_j}$
where $U$ is a unit in a neighborhood of $P$,
 $E_{i_k}=\{u_k=0\}$
($k=1,\dots,s-1$) and $\what V=\{u_s=0\}$.
Here $E_I^*:=\cap E_{i\in I}\setminus \cup_{j\notin I}E_j$.

For a mixed hypersurface, a resolution of this type
 does not exist in general.
The main reason is that  there is no complex structure in the tangent
 space of $V$.
Nevertheless we will show that 
 a suitable toric modification partially resolves such  
singularities. 
\subsection{Toric modification and resolution of complex analytic singularities }
For the reader's convenience, we recall some basic facts about 
 the  toric modifications at the origin.
 We use the notations and the terminologies of
 \cite{ResHyp87,Principal90,Okabook} and \S \ref{Newton boundary}.

\subsubsection{Toric modification}
Let $A =(a_{i,j}) \in GL(n, \BZ)$ with $\det A =\pm 1$. We  call 
such a matrix {\it a unimodular matrix}.
We associate to $A$ a birational morphism  
$$\psi_A : \BC^{*n} \to \BC^{*n}$$
which is defined by 
$\psi_A(\bfz)=(z_1^{a_{1,1}}\dotsb z_{n}^{a_{1,n}},\dots, z_1^{a_{n,1}
}\dotsb z_{n}^{a_{n,n}})$.  If the coefficients of $A$
are non-negative, $\psi_A$ can be defined on $\BC^n$.
Note that $\psi_A$ is a group homomorphism of the algebraic group
$\BC^{*n}$
and we have 
\[
 \psi_A\inv=\psi_{A\inv},\quad
\psi_A\circ \psi_B=\psi_{AB}.
\]

We consider the space of integer  weight vectors $N$ and we denote 
 weight vectors by column vectors. Here the coordinates
 $\bfz=(z_1,\dots, z_n)$ is fixed. The space of the weight vectors
with coefficients in $\BR$
is denoted by $N_{\BR}$. 

Now we consider the subspace of positive weight
 vectors
$N_{\BR}^+$.
Let $P_1,\dots$, $P_m$ be  vectors in $N_{\BR}^+$. 
The polyhedral cone
generated by $P_1,\dots, P_m$
 is defined by
\[
 \Cone (P_1,\dots, P_m):=\{  t_1 P_1+\dots+t_mP_m  \in N | t_i \in \BR, ~t_i \ge 0,~i=1,\dots,m\}.
\]
The interior of $\Cone(P_1,\dots, P_m)$ is called an open cone and it is
defined as
\[
 \Int \Cone (P_1,\dots, P_m):=\{  t_1 P_1+\dots+t_mP_m  \in N | t_i \in \BR, ~t_i > 0,~i=1,\dots,m\}.
\]
The cone $ \Cone (P_1,\dots, P_m)$ is called a {\em simplicial cone} if
$\{P_1,\dots, P_m\}$ are linearly independent.
We  consider only the case where $P_1,\dots, P_m$ are integer vectors.
We call  $P_1,\dots, P_m$ {\em the vertices} of the cone,
 if  $P_1,\dots, P_m$ are chosen to be
primitive
integer vectors, by multiplications of rational numbers if necessary.
It is called a {\em  regular simplicial cone} if $\{P_1,\dots, P_m\}$
can be a part of $\BZ$-basis of $N$.
For a regular simplicial cone 
$ \si=\Cone(P_1,\dots, P_n)$ of dimension $n$ with vertices $P_1,\dots,
P_n$,
we associate a unimodular matrix $A$ whose $j$-th column is $P_j$.
By an abuse of notation, we also denote  $A$ by $\si$.
Let $E_1,\dots, E_n$ be the standard basis of $N$. 
( $E_j={}^t(0,\dots,1,0,\dots, 0)$ where $1$ is at the  $j$-coordinate.)
Then
$\Cone(E_1,\dots, E_n)$ is a regular simplicial cone and 
it is nothing but $N_{\BR}^+$.

We consider a simplicial cone subdivision
$\Si^*$ of the cone
$\Cone(E_1,\dots, E_n)$  for which 
every cone is regular. Such a subdivision is  called a {\em regular fan}.
Suppose that $\Si^*$ is a regular fan.
 Let $\mathcal S$ be the set of $n$-dimensional
cones and let $\mathcal V^+$ be the set of strictly positive vertices.
For simplicity, we assume that the vertices of $\Si^*$ are the union of 
$\{E_1,\dots, E_n\}$ and $\mathcal V^+$.
For each $\si\in \mathcal S$, we consider a copy of a complex Euclidean space $\BC_\si^n$
with coordinates $\bfu_\si=(u_{\si1},\dots, u_{\si n})$ and the morphism
$\pi_\si:\BC_\si^n\to \BC^n$ defined by
$\pi_\si(\bfu_\si)=\psi_\si(\bfu_\si)$. Taking the disjoint sum
$\amalg_{\si\in \mathcal S}\BC_\si^n$, we glue
together
$\amalg_{\si\in \mathcal S}\BC_\si^n$
under the following equivalence relation:
\[
 \bfu_\si\sim \bfu_\tau \quad
\text{if}\quad \psi_{\tau\inv \si}\,\text{ is well-defined at}\,\, \bfu_\si
\,\,\text{and}\,\,\psi_{\tau\inv\si}(\bfu_\si)=\bfu_\tau.
\]
We denote  the quotient space
$\amalg_{\si\in \mathcal S}\BC_\si^n/\sim$ by  $X_{\Si^*}$.
 Then
$X_{\Si^*}$  is a complex manifold of dimension $n$
and the morphisms $\pi_\si:\BC_\si^n\to \BC^n,\,\si\in \mathcal S$
are compatible with the identification and thus they
 define a birational proper holomorphic mapping 
\[
 \what\pi:\, X_{\Si^*}\to \BC^n.
\]
The restriction $\what \pi$ to $X_{\Si^*}\setminus \what\pi\inv(0)$
is a biholomorphic onto $\BC^n\setminus \{O\}$.
We call $\what \pi: X_{\Si^*}\to \BC^n $ {\em the toric modification}
associated with the  regular fan $\Si^*$ {\cite{ResHyp87,Okabook}}.
The irreducible  exceptional divisors correspond bijectively to the
vertices
$P\in \mathcal V^+$ and we denote it by $\what E(P)$.
Then $\what\pi\inv(O)=\bigcup_{P\in \cV^+}\what E(P)$.

The easiest non-trivial case is when
$\mathcal V^+=\{P={}^t(1,\dots,1)\}$. In this case, $X_{\Si^*}$
is nothing but the ordinary blowing-up at the origin of $\BC^n$.

\subsubsection{Dual Newton diagram and admissible toric modifications}
Let $f(\bfz,\bar\bfz)=\sum_{\nu,\mu}c_{\nu,\mu}\bfz^\nu\bar\bfz^\mu$ be a germ of mixed function
in $n$ variables $z_1,\dots, z_n$.
We introduce an equivalence relation in $N_{\BR}^+$ by
\[
 P\sim Q,\,P,Q\in N_{\BR}^+\iff \De(P;f)=\De(Q;f).
\]
The set of equivalence classes  gives an open polyhedral cone subdivision
of $N_{\BR}^+$ and we denote it as $\Ga^*(f;\bfz)$ and we call it {\em the
dual Newton diagram}. 
Let $\Si^*$ be a regular fan which is a regular simplicial cone subdivision
of $\Ga^*(f)$. If $\Si^*$ is a  regular simplicial cone subdivision
of $\Ga^*(f)$,
the toric modification 
$\what \pi: X_{\Si^*}\to \BC^n$ is called {\em admissible } for 
$f(\bfz,\bar\bfz)$.
The basic fact for  non-degenerate holomorphic functions is:
\begin{Theorem}( \cite{ResHyp87,Principal90,Okabook}) Assume that $f(\bfz)$ be a non-degenerate convenient
analytic function with an isolated singularity at the origin.
Let $\what\pi: X_{\Si^*}\to \BC^n$ be an admissible toric
 modification. Then  it is a good resolution of the mapping $f:\BC^n\to
 \BC$
at the origin.
\end{Theorem}
This is a starting observation of the present paper.
 \subsection{Blowing up  examples}
We consider some examples.

\begin{Example}
{\rm A. Let  $C_1=\{(z_1,z_2)\in
\BC^2\,|\,z_1^2-z_2^2=0\}\}$, $V_1=\{(z_1,z_2)\in
\BC^2\,|\,f_1(\bfz,\bar\bfz)=0\}$ 
and $V_2=\{(z_1,z_2)\in\BC^2\,|\,f_2(\bfz,\bar\bfz)=0\}$
where  $f_1(\bfz,\bar\bfz)=\bar z_1^2-z_2^2$ and
 $f_2(\bfz,\bar\bfz)=z_1\bar z_1-z_2^2$. $C_1$ is 
a union of two smooth complex line,
$V_1$ is a union of two smooth
real planes, $\bar z_1\pm z_2=0$ and $V_2$ is an irreducible variety.  Consider
\begin{eqnarray*}
{\what \pi}_1:X_1\to \BC^2
\end{eqnarray*}
where ${\what \pi}_1:X_1\to \BC^2$  is the toric modification associated
 with the regular
fan
generated by vertices
\begin{eqnarray*}
 \Si_1^*=\left\{E_1=\left(\begin{matrix}1\\0\end{matrix}\right),P=\left(\begin{matrix}1\\1\end{matrix}\right), E_2=\left(\begin{matrix}0\\1\end{matrix}\right)
\right\}.
\end{eqnarray*}
Geometrically, ${\what \pi}_1$ is an ordinary blowing up.
Note  that for the complex curve $C_1$,
the two components are separated 
by a single blowing up $\what \pi_1$.
We will see what happens to the two other mixed curves $V_1,\,V_2$.
In the toric coordinate $\BC_\si^2$ with
$\si=\Cone(P,E_2)$ and the toric coordinates $(u_{1},u_{2})$,
the strict transform $\widehat V_1,\,\widehat V_2$ of $V_1,\,V_2$  are defined  in 
the torus $\BC_\si^{*2}$  as 
\[\begin{split}
&\what C_1\cap \BC_\si^{*2}=\{(u_1,u_2)\in  \BC_\si^{*2}\,|\, u_1^2-u_1^2u_2^2=u_1^2(1-u_2^2)=0\}\\
&\widehat V_1\cap \BC_\si^{*2}=\{(u_1,u_2)\in  \BC_\si^{*2}\,|\, \bar u_1^2-u_1^2u_2^2=0\},\\
& \widehat V_2\cap \BC_\si^{*2}=
\{ (u_1,u_2)\in  \BC_\si^{*2}\,|\, u_1(\bar u_1-u_1u_2^2)=0\}.
\end{split}
\]
The first expression shows that $\what C_1$ is already smooth and
 separated
into two peaces.
 Unlike the case of holomorphic functions, we observe that
$$\{(u_1,u_2)\in \BC_\si^2\,|\, \bar u_1^2-u_1^2u_2^2=0\}\supsetneq
\widehat V_1,
\,\{ (u_1,u_2)\in \BC_\si^2\,|\, \bar u_1-u_1u_2^2=0\}\supsetneq \widehat V_2
$$
as $\widehat E(P)=\{u_1=0\}\not\subset \widehat V_i,\,i=1,2$.
In both cases, we see that the 1-sphere  
$|u_2|=1$ appears as their intersection with the exceptional divisor 
$\widehat E(P)$.
 It is easy to see that for $\widehat V_1$, both irreducible components
$L_\pm^*=\{(u_1,u_2)\in \BC_\si^{*2}|\bar u_1 \pm u_1u_2=0\}$
satisfy the limit equality
$\widehat L_\eps^*\cap \widehat E(P)=\{(0,u_2)||u_2|=1\}$ with $\eps=\pm$.
Thus  ${\widehat L_+}\cap {\widehat L_-}$ is  the 1-sphere $|u_2|=1$ and the
ordinary blowing up does not separate the  two smooth components.
For $\widehat V_2$, we will see later that it has two  link components. See
\S 6 for the definition of the link components.
This illustrates the complexity of the limit set of the tangent lines in
the mixed varieties. 

\vspace{.3cm}\noindent
B. We consider an ordinary cusp (complex analytic) $C_2=\{z_2^2-z_1^3=0\}$
and a mixed curve  $V_3=\{z_2^2-z_1^2\bar z_1=0\}$ with the same Newton boundary and an
admissible toric blowing up ${\what \pi}:X_2\to \BC^2$
which is associated with the regular simplicial fan:
\begin{eqnarray*}
\Si_2^*=\left\{E_1, P=
\left(\begin{matrix}1\\1\end{matrix}\right),
Q=\left(\begin{matrix}2\\3\end{matrix}\right), R=
\left(\begin{matrix}1\\2\end{matrix}\right),E_2\right\}
\end{eqnarray*}
 Let
$(u_{1},u_{2})$ be the toric coordinate of $\BC_\si^2$
with $\si=(Q,R)=\left(\begin{matrix}2&1\\3&2\end{matrix}\right)$. Then
the pull back of the defining polynomials are defined in this coordinate chart as
\[\begin{split}
&\what C_2\cap \BC_\si^{*2}=\{(u_1,u_2)\in \BC_\si^{*2}\,|\, u_1^6u_2^3(u_2-1)=0\}\\
&\what V_3\cap \BC_\si^{*2}=\{(u_1,u_2)\in \BC_\si^{*2}\,|\,u_1^4u_2^2(u_1^2u_2^2-\bar u_1^2\bar u2)=0\}.
\end{split}
\]
Observe that 
$\what C_2$ is smooth and transverse to the exceptional divisor 
$\what E(Q)=\{u_1=0\}$.
The strict transform $\widehat V_3$ is defined by
$u_{1}^2\,u_{2}^2-\bar u_{1}^2\, \bar u_{2}=0$ in $\BC_{\si}^{*2}$.
We see  again that
 for $\wtl V_3$, a sphere 
$|u_{2}|=1$ appears as the intersection with the exceptional divisor.
We observe that  $\widehat V_3\cap \widehat E(Q)=\{(0,u_{2})||u_{2}|=1\}$.

\vspace{.3cm}
The above  examples show that the toric modification does not resolve
the singularities of  non-degenerate  mixed
 hypersurfaces.
To get a good resolution of a mixed hypersurface singularity,
 we need to compose a toric modification with a 
normal real blowing up or  a normal polar
modification which we introduce below. }
\end{Example}
\subsection{Normal real blowing up and normal polar blowing up of $\BC$}
Consider the complex plane with two coordinate systems
$z=x+  i \,y$ and $z=r\exp(i\,\theta )$. 
We can consider the following  two modifications.

\vspace{.2cm}
\noindent
(I) Let $\iota_{\BR}:\BC\setminus\{O\}\to \BC\times \BR\BP^1$ defined by
$z=x+ i\,y \mapsto (z,[x:y])$ and let $\cR\BC$ be the closure of the image of
$\iota_{\BR}$.
This is called the real blowing up. $\cR\BC$ is a real two dimensional
manifold which has two coordinate charts $(U_0,(\tilde x,t))$ and 
$(U_1,(s,\tilde y))$. These coordinates are defined by
$\tilde x=x, t=y/x$ and $\tilde y=y,s=x/y$.
The canonical projection $\ome_{\BR}:\cR\BC\to \BC$
is given as  $\ome_{\BR}(\tilde x,t)=\tilde x(1+ i\,t)$
and 
$\ome_{\BR}(s,\tilde y)=\tilde y(s+ i)$. Note that  $\ome_{\BR}\inv(O)=
\BR\BP^1$
and 
$\ome_{\BR}:\cR\BC\setminus \{O\}\times \BR\BP^1\to \BC\setminus\{O\}$
is diffeomorphism.

\vspace{.2cm}\noindent
(II) Consider the polar embedding
$\iota_p:\BC\setminus\{O\}\to \BR^+\times S^1$ which is defined by
$\iota_p(r\exp(\theta\,i ))=(r,\exp(\theta\,i ) )$. 
Here $\BR^+=\{x\in \BR|x\ge 0\}$.
Let $\cP\BC=\BR^+\times S^1$ and $\omega_p:\cP\BC\to \BC$ be the projection
defined by $\omega_p(r,\exp(\theta\,i))=r\exp(\theta\,i )$.
We can see easily that 
$\omega_p\inv(O)=\{0\}\times S^1$ and 
$\omega_p: \cP\BC\setminus \{0\}\times S^1\to \BC\setminus\{O\}$ is a
diffeomorphism.
Note that $\cP\BC$ is a manifold with boundary.

\subsubsection{Canonical factorization}\label{factorization}
There exists a canonical mapping $\psi: \cP X\to \cR\BC$ which is defined
 by
\[
 \psi(r,\exp(\theta\,i))=\begin{cases}
&(\tilde x,t)=(r\cos\theta,\tan\theta),\quad \theta\ne \pm \frac{\pi}2\\
&(s,\tilde y)=(\cot\theta,r\sin\theta),\quad \theta\ne 0,\pi
\end{cases}
\]
It is obvious that $\psi$ gives the commutative diagram
\[
 \begin{matrix}
\cP\BC&\mapright{\psi}& \cR\BC\\
\mapdown{\omega_{p}}&&\mapdown{\omega_{\BR}}\\
\BC&=&\BC
\end{matrix}
\]
Note that the restriction of $\psi$ over the exceptional sets
is a $2:1$ map:
\[
 \psi: \{O\}\times S^1\to \{O\}\times\BR\BP^1,\quad
\exp(\theta\,i )\mapsto [\cos(\theta):\sin(\theta)]
\]

\subsection{Resolution of a mixed function}
Let $f(\bfz,\bar \bfz)$ be a  mixed function and let
$V=f\inv(0)$ and we assume that $V$ has an isolated mixed singularity at the
origin and the real codimension of $V$ is  two.
 (Note that if  $V$ is non-degenerate, it has a real
codimension two by the definition of non-degeneracy and Theorem \ref{isolatedness}.)
Let $Y$ be a real analytic manifold of dimension $2n$
and let $\Phi:Y\to \BC^n$ be a proper  real analytic mapping.
We say that $\Phi:Y\to \BC^n$ is {\em a resolution of a real type}
(respectively {\em a resolution of a polar type}) of the mixed function
 $f$ 
if 
\begin{enumerate}
\item Let $E=\Phi\inv(O)$ and let $E=E_1\cup\cdots\cup E_r$ be
the irreducible components. Each $E_j$ is a real codimension one 
smooth subvariety.
\item $Y$ is a real analytic manifold of dimension $2n$.
 For a resolution of a real type,
$Y$  has no boundary 
while for a resolution of a polar type
$Y$ is   a real analytic manifold with boundary 
and $\partial Y=E$.
\item The restriction $\Phi:Y-E \to \BC^n\setminus\{O\}$ is a real
 analytic
      diffeomorphism.
\item Let $\widetilde  V$ be the strict transform of $V$
(=the closure of $\Phi\inv(V\setminus\{O\})$).
Then $\widetilde V$ is a smooth manifold of real codimension 2 in an open
      neighborhood of $E$.
\item For  $I=\{i_1,\dots, i_t\}$, put
 $E_I^*:=\bigcap_{k=1}^tE_{i_k}\setminus\bigcup_{j\notin I} E_j$. For
 $P\in E_I^*\cap \wtl V$,
 there exists a local real analytic coordinate system
$(U,(u_1,\dots, u_{2n}))$ centered at $P$ such that 
\[
 \Phi^*f(\bfu)=u_1^{m_1}\cdots u_t^{m_t}(u_{t+1}+\,i\,u_{t+2})
\]
so that 
$U\cap E_{i_j}=\{u_j=0\}$ for $j=1,\dots, t$ and 
$U\cap \widetilde V= \{u_{t+1}+\,i\,u_{t+2}=0\}$.
In the case of a resolution of a polar type, we assume also that 
$Y\cap U=\{u_1\ge 0,\dots, u_t\ge 0\}$.
\end{enumerate}
For example, assume that $t=1$ for simplicity. 
Then the condition (5) says the following.
\noindent
If we are considering 
a resolution of a real type, 
\[
 U\cong \BR^{2n}\,\text{or}\,\,B^{2n},\,E_{i_1}=\{u_1=0\},
\quad \Phi^*f(\bfu)=u_1^{m_1}(u_2+i\,u_3),
\]
\noindent
 if we are considering a resolution of polar type,
\[
 U\cong \BR^{2n}\cap \{u_1\ge 0\},\, E_{i_1}=\{u_1=0\},
\, \Phi^*f(\bfu)=u_1^{m_1}(u_2+i\,u_3).
\] See the next section for more details.

\subsubsection{Normal real blowing up}
Let $X$ be a complex manifold of dimension $n$ with a finite
number of  smooth
complex divisors
$E_1,\dots, E_\ell$  such that the union of divisors
   $E=\bigcup_{i=1}^\ell E_i$
has at most normal crossing singularities.
Then we can consider the  composite of real modifications  for the normal
complex 1-dimensional subspaces
along the  divisor $E_1,\dots, E_\ell$. Put it as
$\omega_{\BR}: \cR X\to X$
and we call it {\em the normal real blowing up} along $E$.
 It is immediate from the definition that 

\vspace{.3cm}
\begin{enumerate}
\item $\cR X$ is a differentiable manifold and 
$\omega_{\BR}: \cR X\,\backslash \,{\omega_{\BR}\inv(E)}\to Y\,\backslash\, E$ is a
   diffeomorphism.

\item Inverse image $\widetilde E_j:=\omega_{\BR}\inv(E_j)$ of $E_j$
is a real codimension 1 variety which is fibered over $E_j'$ with a fiber
$S^1$. Here $E_j'$ is the normal real blowing up of $E_j$
along $\bigcup_{i\ne j} E_i\cap E_j$.
Putting $E_I^*:=\bigcap _{i\in I}E_i\backslash \bigcup_{j\notin I}E_j$,
 $\widetilde E_I^*:=\omega_{\BR}\inv(E_I^*)$ is fibered over $E_I^*$ with fiber 
$(S^1)^k$ where  $k=|I|$.
\end{enumerate}
Take a point $P\in E_I^*$ and choose a local coordinate
$(W,(u_1,\dots, u_n))$ so that $I=\{1,\dots, m\}$ and 
$E_j=\{u_j=0\},\,j=1,\dots,m$.
Then $\omega_{\BR}\inv(W)$ is
isomorphic to $(\cR\BC)^m\times \BC^{n-m}$ covered by $2^m$ coordinates
$W_{\eps_1,\dots,\eps_m}=U_{\eps_1}\times\cdots\times U_{\eps_m}\times
      \BC^{n-m}$ where $\eps_j=0$ or $1$.
For example, $W_{1,0,\dots,0}$ has the coordinates 
(as a real analytic manifold)
$(s_1,\tilde y_1,\tilde x_2,t_2,\dots,\tilde x_m,t_m,u_{m+1},\dots, u_n)$
 so that the projection to the coordinate chart ${\bf u}\in W$ is  given by
\[
 u_1=\tilde y_1( s_1 + i ),\,u_2=\tilde x_2(1+i\,t_2  ),\dots, u_m=\tilde
 x_m
(1+i\,t_m   ).
\]
In this coordinate chart, the exceptional real divisor
$\widetilde E_j,\,j\le m$ is defined by $\wtl E_1=\{\tilde y_1=0\}$ 
and $\wtl E_j=\{\tilde x_j=0\}$ for $2\le j\le m$.
\subsubsection{Normal polar blowing up}
We can also consider the composite of the polar blowing ups along 
exceptional
divisors, which we denote as
$\omega_p: \cP X\to X$. In the same coordinate chart
$(W,\bfu),\,\bfu=(u_1,\dots, u_n)$ as in the previous discussion,
$\omega_p\inv(W)$ is written as
\[
 \omega_p\inv(W)=(\BR^+\times S^1)\times \cdots\times(\BR^+\times S^1)\times \BC^{n-m}
\]
 with coordinates
 $(r_1,\exp(i\,\theta_1 ),\dots,
r_m,\exp(i\,\theta_m ),u_{m+1},\dots, u_n)$
and the projection is given by
\[
 \begin{split}
& (r_1,\exp(i\,\theta_1 ),\dots,
r_m,\exp(i\,\theta_m ),u_{m+1},\dots, u_n)\mapsto (u_1,\dots,
u_n),\,\\
&u_j=r_j\exp(i\,\theta_j ),\,j=1,\dots, m.
\end{split}
\]
Note that $\cP X$ is a manifold with boundary and
$\omega_p\inv(E_j)$ is the boundary component which is  given by 
$\{r_j=0\}$.

\subsection{A resolution of a real type and a resolution of a  polar type }
Now we can state our main result for the resolution of non-degenerate
mixed singularities.
 Assume that 
$f(\bfz,\bar\bfz)=\sum_{\nu,\mu}\,c_{\nu,\mu}\,\bfz^\nu\bar\bfz^\mu$
is a non-degenerate convenient mixed function
and consider the mixed hypersurface $V=f\inv(0)$.
Let $\Ga(f)$ be the Newton boundary and let $\Ga^*(f)$ be the dual
 Newton diagram. Take a regular simplicial cone subdivision
in the sense of \cite{Okabook} and let 
${\what \pi}: X\to \BC^n$ be the associated toric modification.
Let $\cV^+$ be the set of 
strictly positive vertices of $\Si^*$
and let $\widehat E(P),P\in \cV^+$  be the exceptional divisors.
We may assume that the vertices which are not strictly positive are
the canonical bases
$\{E_1,\dots, E_n\}$ of $N$ by the convenience assumption
where $E_j={}^t(0,\dots, 1,\dots,0)$.
Put $\widehat E=\bigcup _{P\in \mathcal P} \widehat E(P)$.
Then we take the normal real blowing-ups  $\omega_{\BR}: \cR X\to X $
 along the exceptional divisors of $\widehat E$. Then we consider the composite
\[
 \Phi:={\what \pi}\circ \omega_{\BR}: \cR X\mapright{\omega_{\BR}} X\mapright{{\what \pi}}
 \BC^n,\quad \xi\mapsto {\what \pi}(\omega_{\BR}(\xi)).
\]
Put $\wtl E(P):=\omega_{\BR}\inv(\widehat E(P))$ with $P\in\cV^+$.
\begin{Theorem}\label{real-resolution} 
$\Phi: \cR X\to \BC^n$
gives a good resolution  of a real type of  $f$ at the origin
and the exceptional divisors are ${\widetilde E}(P)$ for $P\in \mathcal V^+$.
The multiplicity of $\wtl E(P)$ of the function $\Phi^*f$ along 
$\wtl E(P)$ is $d(P;f)$.
\end{Theorem}
Let $f(\bfz,\bar\bfz)=g(\bfx,\bfy)+ih(\bfx,\bfy)$ be the decomposition
of $f$ into the real and the imaginary part. Then the above assertion
for the multiplicity is equivalent to: the mutiplicities
of $\Phi^*g,\,\Phi^*h$ along $\wtl E(P)$  are the same and  equal to  $d(P;f)$.
\begin{proof}
We use the same notations as those in \cite{ResHyp87,Principal90, Okabook}. Let
 $ \wtl V,\,\widehat V$ be the
strict transforms of $V$ into $\cR X$ and $X$ respectively.
\[
 \begin{matrix}
\Phi:&\cR X&\mapright{\omega_{\BR}}&X&\mapright{{\what \pi}}&\BC^n\\
&\cup&                           &\cup&&\cup\\
& \wtl V&\mapright{\omega_{\BR}}&\widehat V&\mapright{{\what \pi}}& V&
\end{matrix}
\]
Take any point ${\tl \xi}\in  \wtl V\cap \Phi\inv(O)$ and consider 
$\what\xi=\Phi({\tl \xi})\in \widehat V$.
Assume that $\what\xi$ is in a toric coordinate chart $\BC_\si^n$
with $\si=(P_1,P_2,\dots,P_n)$ which  is a unimodular 
matrix. 
Assume  that  $\what\xi\in \widehat E(P_1,P_2,\dots,P_s)^*$ where 
$\widehat E(P_1,P_2,\dots,P_s)^* = \bigcap_{i=1}^s\widehat E(P_i)\setminus
 \bigcup_{j>s}\widehat E(P_j)$.
For simplicity, we assume that $s=1$ and  $\what\xi\in \widehat E(P_1)^*$,
leaving the other cases to the reader, as the argument is exactly
 the same.
We denote the coordinates
in this chart as $(u_{\si 1},\dots,u_{\si n})$ and
 $u_{\si,j}=x_{\si j}+ i\, y_{\si j} $. 
For simplicity, we write simply $u_j,\,x_j,\,y_j$ for $u_{\si j},\,x_{\si
 j},\,y_{\si j}$ respectively.
By the assumption, 
$\what\xi=(0,\xi_{2},\dots, \xi_n)$ with $\xi_j\ne 0,\,j\ge 2$
in the toric coordinate space $\BC_\si^n$.
We may assume that ${\tl \xi}\in (\BC_\si^n)_0$.
The coordinates of $(\BC_\si^n)_0$ are given by
$(\tilde x_{ 1},t_{ 1},u_{ 2},\dots,u_{ n})$.
The divisor $\widetilde E(P_1)$ is given by $\{\tilde x_{ 1}=0\}$ and
the projection $\omega_{\BR}|(\BC_\si^n)_0 \to \BC_\si^n$ is given as 
\begin{eqnarray*}
 (\tilde x_{ 1},t_{ 1},u_{ 2},\dots,u_{ n})
\mapsto (u_{ 1},\dots,
u_{ n}),\,\,
 u_{ 1}=\tilde x_{ 1}(1+i\,t_{ 1}).\qquad\qquad
\end{eqnarray*}
Let  $\De=\De(P_1)$.
Take an arbitrary monomial  $\bfz^{\nu}\bar\bfz^\mu$. 
Then we observe that
\begin{eqnarray*}\begin{split}
 &\pi_\si^*(\bfz^{\nu}\bar\bfz^\mu)=u_1^{P_1(\nu)}\cdots
 u_n^{P_n(\nu)}\times
\bar u_1^{P_1(\mu)}\cdots \bar u_n^{P_n(\mu)}\quad\text{and}\\
&\omega_{\BR}^* \pi_\si^*(\bfz^{\nu}\bar\bfz^\mu)=
\tilde x_{ 1}^{P_1(\nu+\mu)}(1+i\, t_{1} )^{P_1(\nu)}(1-i\, t_{ 1}  )^{P_1(\mu)}
\prod_{j=2}^n u_{ j}^{P_j(\nu)}\bar u_{ j}^{P_j(\mu)}.
\end{split}
\end{eqnarray*}
Here we recall that $P_1(\nu)=\sum_{j=1}^n p_{j1}\nu_j$.
By the definition of $d(P_1)$, for any monomial $\bfz^{\nu}\bar\bfz^\mu$
which appears in $f(\bfz,\bar\bfz)$, we have
\begin{eqnarray*}
\begin{split}
&P_1(\nu)+P_1(\mu)\ge d(P_1),\,\,\text{and}\\
&        \,P_1(\nu)+P_1(\mu)= d(P_1)
\iff \nu+\mu\in \De(P_1).
\end{split}
\end{eqnarray*}
Thus we can write the pull-back function as
\begin{eqnarray*}
 &\Phi^*f(\tilde x_{ 1},t_{ 1},\bfu_{\si}')=
\tilde x_{ 1}^{d(P_1)}
\times \what f_\si(\tilde x_{ 1},t_{ 1},\bfu_\si')\\
 &\Phi^* f_{\De}(\tilde x_{ 1},t_{ 1},\bfu_{\si}')=
\tilde x_{ 1}^{d(P_1)}
\times \what f_{\De,\si}(t_{ 1},\bfu_\si')\\
& \what f_\si(\tilde x_{ 1},t_{ 1},\bfu_\si')
\equiv \what f_{\De,\si}(t_{ 1},\bfu_\si')
\,\,\modulo\,\,\, (\tilde x_{1}).
\end{eqnarray*}
where $\bfu_\si'=(u_{2},\dots,u_{ n})$.
The important point here 
 is that $ \what f_{\De,\si}$ does not contain
the variable $\tilde x_{1}$.
In the above notation, the strict transform $ \wtl V$ is defined by
$\what f_\si(\tilde x_{1},t_{1},\bfu_\si')=0$ in $(\BC_\si^n)_0$.
Let ${\tl \xi}=(0,\tau_{1},\xi_2,\dots,\xi_n)$ in the coordinates
$(\tilde x_{1},t_{1},\bfu_\si')$.
 Using the expression
$f(\bfz,\bar \bfz)=g(\bfx,\bfy)+  i \, h(\bfx,\bfy)$
and $f_\De(\bfz,\bar \bfz)=g_\De(\bfx,\bfy)+  i \, h_\De(\bfx,\bfy)$,
we write these functions
$\what f_\si,\,\what f_{\De,\si}$ as the sum of real-valued functions: 
\begin{eqnarray*}\begin{split}
& \what f_\si(\tilde x_{1},t_{1},\bfx_\si',\bfy_{\si}')=
\what g_\si(\tilde x_{1},t_{1},\bfx_\si',\bfy_{\si}')+  i \,
\what h_\si(\tilde x_{1},t_{1},\bfx_\si',\bfy_{\si}')\\
& \what f_{\De,\si}(t_{1},\bfx_\si',\bfy_{\si}')=
\what g_{\De,\si}(t_{1},\bfx_\si',\bfy_{\si}')+  i \,
\what h_{\De,\si}(t_{1},\bfx_\si',\bfy_{\si}'),\\
&\qquad \text{where}\quad 
\bfx_\si'=(x_{ 2},\dots,x_{ n}),\,\,
\bfy_\si'=(y_{2},\dots, y_{ n}).
\end{split}
\end{eqnarray*}
The main assertion in Theorem \ref{real-resolution} is that
 the rank of the Jacobian matrix
of the  functions $\tilde x_{ 1},\what g_{\si},\what h_\si$:
\[
J:= \frac{\partial (\tilde x_{ 1},\what g_\si,\what h_\si)}
{\partial(\tilde x_{1},t_{1},\bfx_\si',\bfy_\si')}({\tl \xi})
=\left(
\begin{matrix}
1&0\\
\star&
 \frac{\partial (\what g_\si,\what h_\si)}
{\partial(t_{1},\bfx_\si',\bfy_\si')}({\tl \xi})\\
\end{matrix}
\right)
\]
is $3$, which is
 equivalent to
\[
\rank\, \left(\frac{\partial (\what g_\si,\what h_\si)}
{\partial(t_{1},\bfx_\si',\bfy_\si')}({\tl \xi})\right)=2.
\]
Note that $g_{\De,\si}({\tl \xi})=h_{\De,\si}({\tl \xi})=0$ and
\[
g_\si-g_{\De,\si}\equiv 0,\quad h_\si-h_{\De,\si}\equiv 0\,\,
\modulo \,\,(\tilde x_{1})
\]
 therefore
\begin{eqnarray}\label{eq1}
 \frac{\partial
 (g_\si,h_\si)}{\partial(t_{1},\bfx_\si',\bfy_\si')}
({\tl \xi})=
\frac{\partial (g_{\De,\si},h_{\De,\si})}{\partial(t_{1},\bfx_\si',\bfy_\si')}({\tl \xi}).
\end{eqnarray}
Now recall that $g_{\De,\si},h_{\De,\si}$ does not contain the  variable
$\tilde x_{ 1}$. Define a modified point 
${\tl \xi}'\in (\BC_\si^n)_0$ by 
 ${\tl \xi}'=(1,\tau_{ 1},\xi_2,\dots,\xi_n)$
and put ${\what\xi}'=\omega_{\BR}({\tl \xi}')\in \BC_\si^{*n}$ and 
$\bfw_0=\pi_\si({\what\xi}')\in \BC^{*n}$.
Put $\bfw_0=\bfx_0+ i\,\bfy_0  $.
(Recall that $\pi_\si:\BC_\si^n\to \BC^n$ is the projection of the toric
 modification in this chart.)
Then as $g_{\De,\si}({\what\xi}')=g_{\De,\si}({\tl \xi}')=0$, we have
\begin{eqnarray*}\begin{split}
 \rank\,\left(\frac{\partial
 (g_{\De,\si},h_{\De,\si})}{\partial(
t_{1},\bfx_\si',\bfy_\si')}({\tl \xi})\right)
&=\rank\,\left(\frac{\partial
 (g_{\De,\si},h_{\De,\si})}{\partial(
t_{1},\bfx_\si',\bfy_\si')}({\tl \xi}')\right)\\
&=
 \rank\, \left(\frac{\partial
 (g_{\De,\si},h_{\De,\si})}{\partial(
\tilde x_{1},t_{1},\bfx_\si',\bfy_\si')}({\tl \xi}')\right)
\end{split}
\end{eqnarray*}
Now  we consider the hypersurface
\begin{eqnarray*}\begin{split}
V_\De^*:=&\{\bfz\in \BC^{*n}\,|\, f_\De(\bfz)=0\}\\
=&\{\bfx+ i\,\bfy \in\BC^{*n} \,|\,
g_\De(\bfx,\bfy)=h_\De(\bfx,\bfy)=0\}
\end{split}
\end{eqnarray*}
where $z_j=x_j+ i\,y_j,\,j=1,\dots, n$ and 
$\bfx=(x_1,\dots, x_n),\,\bfy=(y_1,\dots, y_n)$.
Note that $\bfw_0\in V_\De^*$.
As $f_\De(\bfw_0)=0$ and 
$\Phi={\what \pi}\circ\omega_{\BR}:\Phi\inv(\BC^{*n})\to \BC^{*n}$
 is a diffeomorphism,
we see that 
\begin{eqnarray*}\begin{split}
 \rank\,\left(\frac{\partial
 (g_{\De,\si},h_{\De,\si})}{\partial(
\tilde x_{1},t_{1},\bfx_\si',\bfy_\si')}({\tl \xi}')\right)&=
\rank\,\left( \frac{\partial
 (\tilde x_{1}^{d(P_1)} g_{\De,\si}, \tilde
 x_{1}^{d(P_1)}h_{\De,\si})}
{\partial(\tilde x_{1},t_{1},\bfx_\si',\bfy_\si')}({\tl \xi}')\right)\\
&=
\rank\,\left(\frac{\partial
 (g_{\De},h_{\De})}{\partial(
\bfx,\bfy)}(\bfx_0,\bfy_0)\right)=2
\end{split}
\end{eqnarray*}
where $\bfw_0=\bfx_0+i\,\bfy_0$.
The first equality is the result of 
$g_{\De \si}({\tl\xi}')=h_{\De \si}({\tl\xi}')=0$.
The last equality follows from  the non-degeneracy condition
which assumes that
$f_\De: \BC^{*n}\to \BC$ has $0$ as a regular value.
\end{proof}

We can also use the normal polar
blowing-up
$\omega_p: \cP X\to X$ along  $\widehat E(P),\,P\in \cV^+$ and the composite
$\Phi_p: \cP X\to \BC^n$.
Put $\wtl E(P):=\Phi_p\inv(\widehat E(P))$, $P\in\cV^+$.
\begin{Theorem}\label{polar-resolution}Under the same assumption as in Theorem
 \ref{real-resolution}, $\Phi_p:\cP X\to X$ gives  a good resolution
of a polar type of $f(\bfz,\bar\bfz)$ where $\Phi_p$ is the composite
\[
 \Phi_p:\cP X\mapright{\omega_p}X\mapright{{\what \pi}} \BC^n.
\]
The multiplicity of $\wtl E(P)$ of the function $\Phi_p^*f$ along 
$\wtl E(P)$ is $d(P;f)$.
There is a canonical factorization
$\eta:\cP X\to \cR X$ so that $\omega_p=\omega_{\BR}\circ \eta$
and $\Phi_p=\Phi\circ\eta$.
 \end{Theorem}
\begin{proof} The proof is almost the same as that of Theorem
 \ref{real-resolution}. 
For an arbitrary monomial  $\bfz^{\nu}\bar\bfz^\mu$. 
Then we observe that
\begin{eqnarray*}\begin{split}
 &\pi_\si^*(\bfz^{\nu}\bar\bfz^\mu)=u_1^{P_1(\nu)}\cdots
 u_n^{P_n(\nu)}\times
\bar u_1^{P_1(\mu)}\cdots \bar u_n^{P_n(\mu)}\quad\text{and}\\
&\omega_{p}^* \pi_\si^*(\bfz^{\nu}\bar\bfz^\mu)=
r_{ 1}^{P_1(\nu+\mu)} \exp(P_1(\nu-\mu)\theta_1\,i)
\prod_{j=2}^n u_{ j}^{P_j(\nu)}\bar u_{ j}^{P_j(\mu)}.\\
\end{split}
\end{eqnarray*}
Thus we simply replace $(\tl x_1,t_1,\bfu_\si')$ by 
$(r_1,\theta_1,\bfu_\si')$ with $u_{\si1}=r_1\exp(i\,\theta_1)=\tl
 x_1(1+i\,t_1)$
in the previous calculation.
The factorization follows from \S \ref{factorization}.

 \end{proof}
\begin{Remark}{\rm 
The assertion of Theorem \ref{real-resolution}
and Theorem \ref{polar-resolution}  says that the strict transform
$\wtl V$ is a ``Cartier divisor'' in the sense that it is locally 
defined by a single 
complex-valued real analytic function in $\mathcal R X$, although $\widehat V$
 is not a Cartier divisor in $X$.
Note also that  the pull-back of $g$ and $h$ are real-valued
 functions
which have the same multiplicity $d(P)$ along $\wtl E(P),\,P\in \cV^+$.}

\end{Remark}
\begin{Example}\label{counter-example}{\rm 
We consider two modifications:
\begin{eqnarray*}
{\what \pi}_1:X_1\to \BC^2,\quad {\what \pi}_2:X_2\to \BC^2
\end{eqnarray*}
where ${\what \pi}_j:X_j\to \BC^2$  is the toric modification associated
 with the regular
fan $\Si_j^*$ ($j=1,2$) which  are defined by the vertices as follows.
\begin{eqnarray*}
 \Si_1^*=\left\{E_1=\left(\begin{matrix}1\\0\end{matrix}\right),P=\left(\begin{matrix}1\\1\end{matrix}\right), E_2=\left(\begin{matrix}0\\1\end{matrix}\right)
\right\},\\
\Si_2^*=\left\{E_1, P=
\left(\begin{matrix}1\\1\end{matrix}\right),
Q=\left(\begin{matrix}2\\3\end{matrix}\right), R=
\left(\begin{matrix}1\\2\end{matrix}\right),E_2\right\}
\end{eqnarray*}

\noindent
1. Let $V_1=f(\bfz,\bar\bfz)=\bar z_1^2-z_2^2=0$. This is a union of two smooth
real planes
$z_2\pm \bar z_1=0$. In the toric coordinate chart $\BC_\si^2$ with
$\si=\Cone(P,E_2)$,
the strict transform $ \wtl V_1$ of $V_1$ is defined  in $\BC_\si^{*2}$ by
\[
\widehat V_1:\quad \bar u_1^2-u_1^2u_2^2=0.
\]
We have seen that $\widehat V_1\cap \widehat E(P)=\{u_1=0\,|\, |u_2|=1\}$.
Now take  the normal real blowing up along $\widehat E(P)$,
$\omega_{\BR}: \cR X\to X$.
The strict transform is defined in $(\BC_\si^2)_\eps$ as
\begin{eqnarray*}
\widehat V_1&=\{(\tilde x_1,t_1,u_2)\in \BR^2\times \BC\,|\, (1-i\, t_1)^2-(1+i\, t_1)^2u_2^2=0\}\\
&=\{(\tilde y_1,s_1,u_2)\in \BR^2\times \BC\,|\, (s_1- i)^2-(s_1+ i)^2u_2^2=0\}
\end{eqnarray*}
Note these equations give two smooth components $L_\eps,\,\eps=\pm 1$
 which are disjoint:
\[
 \{(\tilde x_1,t_1,u_2)\in \BR^2\times \BC\,|\, (1-i\, t_1)\pm (1+i\, t_1)u_2=0\}.
\]
This expression shows that the strict transform is embedded in
the cylinder $|u_2|=1$. 
Let us see this in a normal polar modification
$\omega_p:\cP X\to X$.
Now $\cP X$ is locally diffeomorphic to the 
 product of $S^1\times \BR^+\times \BC $ and
the strict transform is now defined in a simple equation
\[
 \tilde V_1=\{(r_1,\exp(\theta\,i ),u_2)\,|\, u_2=\mp \exp(-2\,\theta\,i )\}
\]
and it has  two link components.
This shows that the strict transform is a product (it does not depend on
$r_1$)
and for a fixed $r_1$, they are parallel torus knots in 
$S^1\times S^1=S^1\times \{|u_2|=1\}$. Observe that the direction of
twisting is opposite in the first and the second $S^1$'s
 with respect to the canonical orientation
of $S^1$.

\noindent
2. Let us consider another mixed curve:
\[
 V_2: \{z_1\bar z_1-z_2^2=0\}
\] 
Equivalently $V_2$ is defined by
\[
 \{(x_1,y_1,x_2,y_2)\in \BR^4\,|\,x_1^2+y_1^2=x_2^2-y_2^2,\,x_2y_2=0\}.
\] This can be defined as
\begin{eqnarray*}
V_2= \{(x_1,y_1,x_2,y_2)\in \BR^4\,|\,y_2=0,\,x_2^2=x_1^2+y_1^2\}.
\end{eqnarray*}
This curve is real analytically (or real algebraically) irreducible at the origin (see
\cite{Benedetti-Risler} for the definition) but
we can see that $V_2\setminus\{O\}$ has two connected components
$z_2=|z_1|$ and $z_2=-|z_1|$. Thus  for the geometrical study of real
analytic varieties, especially for the study of real analytic curves, it is better 
to see
the connected components of $f\inv(0)\setminus \{O\}$.
We  apply the same toric modification ${\what \pi}_1$ and we consider
its 
 strict transform on the toric chart $\Cone (P,E_2)$
where we use  the same notation as in Example \ref{counter-example}.
\[
 \widehat V_2:\bar u_1-u_1u_2^2=0.
\]
Again we see that $\widehat V_2\cap \widehat E(P)=\{(0,u_2)\,|\,|u_2|=1\}$.
Take   the normal real blowing up along $\widehat E(P)$.
The strict transform is defined in $(\BC_\si^2)_\eps$ as
\begin{eqnarray*}
\wtl V_2&=\{(\tilde x_1,t_1,u_2)\in \BR^2\times \BC\,|\, (1-i\, t_1)-(1+i\, t_1)u_2^2=0\}\\
&=\{(\tilde y_1,s_1,u_2)\in \BR^2\times \BC\,|\, (s_1- i)-(s_1+ i)u_2^2=0\}
\end{eqnarray*}
which is non-singular. They have two real analytic  components:
\begin{eqnarray*}
&\{(\tilde x_1,t_1,u_2)\in \BR^2\times \BC\,|\,u_2\pm(1-i\, t_1)/\sqrt{1+t_1^2}=0\}\quad\text{or}\\
&\{(\tilde y_1,s_1,u_2)\in \BR^2\times \BC\,|\,
 u_2\pm (s_1- i)/ \sqrt{s_1^2+1}
=0\}
\end{eqnarray*}
Note that $\sqrt{1+t_1^2}$ is a real analytic function, although 
$\sqrt{x_1^2+y_1^2}$ is not an analytic function at $O$.
The above expression says that $\wtl V_2$ is a product
 \[
  \left \{(t_1,u_2)|\sqrt{1+t_1^2}\,u_2\pm(1-i\, t_1)=0\right\}\times \BR
  \]
  where the
second factor is the line with coordinate $\tilde x_1$.
Using the   resolution of a polar type, 
 $\wtl V_2$
is simply written as
\[
 \wtl V_2=\{(r_1,\theta_1,u_2)\in \BR^+\times S_1\times \BC\,|\,
u_2\pm \exp(-\,\theta_1\,i )=0\}.
\]
 Again we observe that it is a product of torus knots and $\BR^+$.

\vspace{.3cm}\noindent
3. Next we consider $V_3=\{z_2^2+z_1^2\bar z_1=0\}$. 
The Newton boundary is the same with that of the cusp singularity
$z_2^2+z_1^3=0$. Thus 
we use the toric modification ${\what \pi}_2: X_2\to \BC^2$. Let
$(u_1,u_2)$ be the toric coordinate
of the chart  $\si=(Q,R)=\left(\begin{matrix}2&1\\3&2\end{matrix}\right)$. Then
the pull back of $f$ is defined in this coordinate chart as
\[
\what f(u_{1},u_{2})= (u_{1}^3\,u_{2}^2)^2+(u_{1}^2\,u_{2})^2 (\bar u_{1}^2 \,\bar u_{2})=u_1^4u_2^2\,(u_1^2u_2^2+\bar u_{1}^2 \,\bar u_{2})
\]
Thus the strict transform can be written as $u_{1}^2\,u_{2}^2+\bar u_{1}^2\, \bar u_{2}=0$.
Thus again we see that
 $\widehat V_3\cap \widehat E(Q)=\{(0,u_2)\,|\,|u_2|=1\}$  and  $S^1$ appears as the limit of $\widehat V_3\cap\what E(Q)$ where
$\what E(Q)$ is the exceptional divisor corresponding to $Q$.
We take a normal polar modification
$\omega_p:\cP X_2\to X_2$ and consider this  in the coordinate chart $\omega_p\inv(\BC_\si^2)$
with coordinates
$(r_1,\exp(\theta_1\,i ),r_2,\exp(\theta_2\,i ))$
with $u_1=r_1\exp(\theta_1\,i),\,u_2=r_2\exp(\theta_2\,i )$. Then
 $\wtl V_3$ is defined by
\[
 \tilde V_3=\{r_2\exp(3\,\theta_2\,i )+\exp(-4\,\theta_1\,i )=0\}
\]
which implies that 
$r_2=1$ and $3\theta_2\equiv -4\theta_1$ mod $2\pi$.
We see that $\wtl V_3\cap \widetilde E(Q)$ is a torus knot
but the orientations  for $\theta_1$ and $\theta_2$
are reversed.

In the resolution of a  real type, the equation
 is apparently a little complicated.
In the chart $\BC_{\si,0,0}$,  $\wtl V_3\cap \widetilde E(Q)$ is given
by $g(t_1,s_2,\tilde x_2)=h(t_1,s_2,\tilde x_2)=0$ where
\begin{eqnarray*}
&g(t_1,s_2,\tilde x_2)={\tilde x_2}-{\tilde x_2}\,s_2^{2}-4\,{\tilde x_2}\,{t_1}\,{s_2}-{
\tilde x_2}\,t_1^{2}+{\tilde x_2}\,t_1^{2}s_2^{2}+1-2\,{
t_1}\,{s_2}-t_1^{2}\\
&h(t_1,s_2,\tilde x_2)=-2\,{\tilde x_2}\,{s_2}-2\,{\tilde x_2}\,{t_1}+
2\,{\tilde x_2}\,{t_1}\,
s_2^{2}+2\,{\tilde x_2}\,t_1^{2}{s_2}+{s_2}+2\,{t_1}-
t_1^{2}{s_2}.
\end{eqnarray*}
Taking the resultant of $g(t_1,s_2,\tilde x_2)$ and $h(t_1,s_2,\tilde x_2)$
in $t_1$, we see that $s_2^2\tilde x_2^2+s_2^2=1$ which corresponds to
$r_2=1$
in the polar resolution.}
\end{Example}
\subsubsection{Pseudo weighted homogeneous hypersurface}
Suppose that $f(\bfz,\bar\bfz)$ is a convenient non-degenerate
mixed function, let ${\what \pi}: X\to \BC^n$ be an admissible toric
modification and let $\omega_{\BR}:\cR X\to X$ be a real modification
along the exceptional divisors as in Theorem \ref{real-resolution}.
Suppose that for a  strictly positive weight $P$,
 $f_P(\bfz,\bfz)$ is a pseudo weighted homogeneous
polynomial.
Write it as
$f_P(\bfz,\bar\bfz)=Mh(\bfz)$ where  $M$ is a mixed monomial and
 $h(\bfz)$
is a  weighted homogeneous polynomial
with $h\inv(0)\cap \BC^{*n}$  being smooth.
Take a toric coordinate chart $\si=(P_1,\dots, P_n)$ with $P=P_1$.
Put $d_M=\rdeg_P\,M$ and $d_h=\rdeg_P \,h$. Then $\rdeg_P\,f=d_M+d_h$.
Then the strict transform $\widehat V$
in the  toric coordinates $\BC_\si^n$ is already non-singular.
Using
the same notation as in the proof of Theorem  \ref{real-resolution},
we have 
\begin{eqnarray*}
&{\what \pi}^*f_P(\bfu_\si,\bar\bfu_\si)={\what \pi}^*(M)\,{\what \pi}^*h(\bfu_\si)=
M'u_{1}^{d_h}\tl h(\bfu_\si')
\end{eqnarray*} where $M'$ is a mixed monomial
and  $h(\bfu')$ is a polynomial of 
$u_{\si2},\dots,u_{\si n}$. Let 
$ E(P_1)=\{\bfu'\in \BC_{\si}^{n-1}\,|\, \tilde h(\bfu_\si')=0\}$
be the exceptional divisor. Then 
$\widehat V$ is diffeomorphic to the product
$\BC\times E(P_1)$. Now we take the normal real modification.
The defining equation of the strict transform 
$\wtl V$ in $(\BC_{\si})_0$ is given as
$\what f_\si(\tilde x_{1},t_\si,\bfu_\si')=0$ where
\begin{eqnarray*}\begin{split}
&\what f_\si(\tilde x_{1},t_\si,\bfu_\si')\equiv 
\what f_{P,\si}(\tilde x_{1},t_\si,\bfu_\si')\,\,\modulo (\tilde x_1)\\
&\Phi^*f_P(\tilde x_{1},t_\si,\bfu_\si')\,=\,
\tilde x_{1}^{d(P_1)}\,\what f_{P, \si}(\tilde x_{1},t_\si,\bfu_\si')\\
&\what f_{P,\si}(\tilde
 x_{1},t_\si,\bfu_\si')\,=\,(1+t_{1}\,i)^a\,(1-t_{1}\,i)^b\,\tilde h
(u_{\si2},\dots,u_{\si n}).
\end{split}
\end{eqnarray*}
Thus we see that  $\wtl V$ is a
product  
$\cR\BC\times  E(P_1)$. The modification $\omega_p:\cP X\to X$
is simply the polar modification of the trivial factor $\BC$.

\section{Milnor fibration}
In this section, we study the Milnor fibration,
assuming  that  $f(\bfz,\bar\bfz)$ is
 a strongly
 non-degenerate convenient mixed function.
We have seen in Theorem \ref{isolatedness} that
there exists a positive number $r_0$ 
 such that $V=f\inv(0)$ is mixed non-singular except at the origin
in the ball $B_{r_0}^{2n}$
and the sphere $S_r^{2n-1}$ intersects transversely
with $V$ for
any $0<r\le r_0$.
The following is a key assertion for which we need the strong non-degeneracy.
\begin{Lemma}\label{MilnorFibering2}
Assume that  $f(\bfz,\bar\bfz)$ is
 a strongly
 non-degenerate convenient mixed function.
For any fixed positive number $r_1$ with  $r_1\le r_0$, there
 exists
 positive numbers
 $\de_0\ll  r_1$
such that for any $\eta\ne 0,\,|\eta|\le \de_0$ and $r$ with $r_1\le
 r\le r_0$,
{\rm (a)}
the fiber $V_\eta:=f\inv(\eta)$ has no mixed singularity inside the ball
 $B_{r_0}^{2n}$ and
{\rm (b)}  the intersection $V_\eta\cap S_{r}^{2n-1}$ is transverse
and smooth.
\end{Lemma}
\begin{proof}
As  the assertion (b) follows from the compactness argument,  we show the
 assertion (a) by contradiction.
We assume that 
(a) does not hold. Then using  the Curve Selection
 Lemma (\cite{Milnor,Hamm1}),  we can find an analytic  path
$\bfz(t),\,0\le t\le 1$ such that 
$\bfz(0)=O$ and 
$f(\bfz(t),\bar\bfz(t))\ne 0$ for $t\ne 0$ and $\bfz(t)$ is a critical point
 of the function $f:\BC^n\to \BC$.
The proof is similar to that of Theorem \ref{isolatedness}
as we will see below.
Using Proposition \ref{critical point},
we can find a real analytic  family $\la(t)$ in $S^1\subset \BC$
such that 
\begin{eqnarray}\label{Sing-cond}
 \overline{df}(\bfz(t),\bar\bfz(t))=\la(t)\, \bar df(\bfz(t),\bar\bfz(t)).
\end{eqnarray}
Put $I=\{j\,|\,z_j(t)\not \equiv 0\}$.
We may assume that 
 $I=\{1,\dots,m\}$ and we consider $f^I$. 
As $f(\bfz(t),\bar \bfz(t))=f^I(\bfz(t),\bar \bfz(t))\not \equiv 0$, we
 see that $f^I\ne 0$. 
Consider the Taylor expansions of $\bfz(t)$ and $\la(t)$:
\begin{eqnarray*}
 &\bfz_i(t)=b_i \,t^{a_i}+\text{(higher terms)},\,b_i\ne 0\quad
i=1,\dots, m\\
&\la(t)=\la_0+\la_1\, t+\text{(higher terms)},\quad \la_0\in S^1\subset \BC.
\end{eqnarray*}
Let  $A=(a_1,\dots, a_m)$, $\bfb=(b_1,\dots, b_m)$
  and  we consider the face function $f^I_A$
of $f^I(\bfz,\bar\bfz)$.
Then 
we have
\begin{eqnarray*}\begin{split}
&\frac{\partial f}{\partial z_j}(\bfz(t),\bar\bfz(t))=\frac{\partial f^I_A}{\partial
 z_j}(\bfb)\,t^{d-a_j}+\text{(higher terms)},\\
 &\frac{\partial f}{\partial \bar z_j}(\bfz(t),\bar\bfz(t))=\frac{\partial f^I_A}{\partial
 z_j}(\bar\bfb)\,t^{d-a_j}+\text{(higher terms)},\,\,\,d=d(A;f^I).
\end{split}
 \end{eqnarray*}
Observe that we have the following equality for any $j$
 by the equality (\ref{Sing-cond}):
\[
 \ord_t\,{\frac{\partial f^I}{\partial z_j}
 \,(\bfz(t),\bar\bfz(t))}=\ord_t \,
 \frac{\partial f^I}{\partial \bar z_j}(\bfz(t),\bar\bfz(t)).
\]
 Thus by (\ref{Sing-cond}), we get the equality:
  $\overline{df^I_A}(\bfb,\bar \bfb)=\la_0\, \bar df^I_A(\bfb,\bar\bfb)$
and $\bfb\in \BC^{*n}$.
  This implies that $\bfb$ is 
 a critical point of $f_A^I:\BC^{*I}\to \BC$, which is a contradiction to the
 strong non-degeneracy of 
$f^I(\bfz,\bar\bfz)$.
\end{proof}
\subsection{ Milnor fibration, the second description}
Put 
\[\begin{split}
&D(\de_0)^*=\{\eta\in \BC\,|\,0<|\eta|\le \de_0\},\,\,
S^1_{\de_0}=\partial D(\de_0)^*=\{\eta\in \BC\,|\,|\eta|=\de_0\}\\
&E(r,\de_0)^*=f\inv(D(\de_0)^*)\cap B_r^{2n},\,\,
\partial E(r,\de_0)^*=f\inv(S_{\de_0}^1)\cap B_r^{2n}.
\end{split}
\]
By Lemma \ref{MilnorFibering2} and the theorem of Ehresman
(\cite{Wolf1}), we obtain the following  description of   the Milnor fibration
 of the second type
(\cite{Hamm-Le1}).
\begin{Theorem}{\rm(The second description of the Milnor fibration)}\label{Milnor2}
Assume that $f(\bfz,\bar\bfz)$ is a convenient, strongly non-degenerate
mixed function. 
Take positive numbers $r_0,r_1$
and $\de_0$ such that $r\le r_0$ and $\de_0\ll r_1$ as in
 Lemma
\ref{MilnorFibering2}.
Then $f: E(r,\de_0)^*\to D(\de_0)^*$ and 
$f:\partial E(r,\de_0)^*\to S_{\de_0}^1$ are  locally trivial fibrations
and the topological isomorphism class does not depend on the choice of 
$\de_0$ and $r$.
\end{Theorem}
\subsection{ Milnor fibration, the first description}
We consider now the  original Milnor fibration on the sphere, which is
defined as follows: 
\[
 \varphi:S_r^{2n-1}\setminus K_r\to S^1,
\quad \bfz\mapsto \varphi(\bfz)=f(\bfz,\bar\bfz)/|f(\bfz,\bar\bfz)|
\]
where $K_r=V\cap S_r^{2n-1}$. The fibrations of this type for mixed functions 
and related topics have
been studied by 
many authors (\cite{R-S-V,Seade, G-L-M,SeadeBook,
Pichon-Seade,Bodin-Pichon}).
But most of the works treat rather special classes of functions.
The mapping $\varphi$ can be identified with 
$\varphi(\bfz)=-\Re (i \log f(\bfz))$, taking the argument $\theta$
as a local
coordinate of the circle $S^1$.
We use the  basis
$\{\frac{\partial}{\partial z_j},\,\frac{\partial}{\partial \bar z_j}\,|\,
j=1,\dots, n\}$ of the tangent space
$T_{\bfz}\BC^n\otimes \BC$.
For a mixed  function $g(\bfz,\bar\bfz)$,
we use two complex ``gradient vectors''
defined by
\[
 dg=(\frac{\partial g}{\partial z_1},\dots,
 \frac{\partial g}{\partial z_n}),
\quad
\bar dg=(\frac{\partial g}{\partial \bar z_1},\dots,
 \frac{\partial g}{\partial \bar z_n}).
\]
Take a  smooth path
$\bfz(t),\,-1\le t\le 1$ with $\bfz(0)=\bfw\in \BC^n\setminus V$
and put $\bfv=\frac {d\bfz} {dt}(0)\in T_{\bfw}\BC^n$.
Then
we have 
\begin{eqnarray*}
\begin{split}
-{\frac{ d}{dt}}&\left( \Re (i\, \log f(\bfz(t),\bar\bfz(t))\right)_{t=0}\\
&=-\Re\left(\sum_{i=1}^n
 i\,\left\{\frac{\partial f}{\partial z_j}(\bfw,\bar\bfw )\frac {d z_j}{dt}(0)+
\frac{\partial f}{\partial \bar z_j}(\bfw,\bar\bfw )\frac {d\bar z_j}{dt}(0)
\right\}/{f(\bfw,\bar\bfw )}\right )\\
&=\Re(\bfv,i\,\overline{ d\log f}(\bfw,\bar\bfw ))+
\Re(\bar \bfv,i\,\overline{\bar d\log f}(\bfw,\bar\bfw ))
\\
&=\Re(\bfv,i\,\overline{ d\log f}(\bfw,\bar\bfw ))+
\Re( \bfv,-i\,{\bar d\log f}(\bfw,\bar\bfw ))\\
&=\Re(\bfv, i\,( \overline{d\log f}-{\bar d\log f})(\bfw,\bar\bfw )).
\end{split}\end{eqnarray*}
Namely we have 
\begin{eqnarray}\label{argument-derivative}
-{\frac{ d}{dt}}\left( \Re (i\, \log
		  f(\bfz(t),\bar\bfz(t))\right)_{t=0}=
\Re(\bfv, i\,( \overline{d\log f}-{\bar d\log f})(\bfw,\bar\bfw )).
\end{eqnarray}
Thus by the same argument as in Milnor \cite{Milnor}, we get
\begin{Lemma}\label{ExistenceRadius0} A point
$\bfz\in S_r^{2n-1}\setminus K_r$ is a critical point of 
$\varphi$ if and only if the  two complex vectors
$i\,(\overline{d\log f}(\bfz,\bar\bfz)-{\bar d \log f}(\bfz,\bar\bfz))$
 and $\bfz$ are  linearly dependent over $\BR$.
\end{Lemma}
\noindent
The key assertion is the following.
\begin{Lemma}\label{ExistenceRadius1} Assume that $f(\bfz,\bar\bfz)$ is a
strongly 
non-degenerate mixed function. Then there exists a positive number
 $r_0$ such that  the two complex vectors
$i\,(\overline{d\log f}(\bfz,\bar\bfz)-{\bar d \log f}(\bfz,\bar\bfz))$
 and $\bfz\in S_r\setminus K_r$ are
linearly independent over $\BR$
for any $r$ with $0< r\le r_0$.
\end{Lemma}
\begin{proof}
We do not assume the convenience of $f(\bfz,\bar\bfz)$ for this lemma.
We proceed as the proof of Lemma 4.3 \cite{Milnor}. Assuming the
 contrary,
we can find an analytic path
$ \bfz(t),\,0\le t\le 1$ such that 

\begin{enumerate}
\item[(a)] $\bfz(0)=O$ and $\bfz(t)\in \BC^n\setminus V$ for $t>0$.
\item[(b)]
 $i\,(\overline{d\log f}-\bar d\log f)(\bfz(t),\bar \bfz(t))=\la(t)
	   \bfz(t)$ 
for some $\la(t)$ such that $\la(t)$ is a real number.
\end{enumerate}
As $\overline{d \log f}-\bar d\log f$ does not vanish
outside of $f\inv(0)$ near the origin by 
Lemma \ref{MilnorFibering2} and Proposition \ref{critical point},
we see that $\la(t)\not \equiv 0$. Consider 
the subset  $I=\{j|z_j(t)\not \equiv 0\}\subset \{1,\dots,n\}$.
For simplicity, we may assume that  $I=\{1,\dots, m\}$.
Consider the Taylor expansions:
\begin{eqnarray*}
&z_j(t)=a_j \, t^{p_j}+\text{(higher terms)},\,\, a_j\ne 0,\,p_j>0,\,j\in I.
\end{eqnarray*}
Put $P={}^t(p_1,\dots, p_m)$, $\bfa=(a_1,\dots, a_m)\in \BC^{*I}$ and
 $d=d(P;f^I)$.
Then we consider the expansions:
\begin{eqnarray*}
\begin{split}
&f(\bfz(t),\bar\bfz(t))=f^I(\bfz(t),\bar\bfz(t))=\al\, t^{q}+\text{(higher terms)},\,\,q\ge d,\,\al\ne 0\\
&\frac{\partial f^I}{\partial z_j}(\bfz(t),\bar\bfz(t))=\frac{\partial
 f_P^I}{\partial z_j}(\bfa,\bar\bfa)\, t^{d-p_j}+\text{(higher
 terms)},\,1\le j\le m\\
&\frac{\partial f^I}{\partial \bar z_j}(\bfz(t),\bar\bfz(t)))=
\frac{\partial f_P^I}{\partial \bar z_j}(\bfa,\bar\bfa)\, t^{d-p_j}+\text{(higher terms)},\,1\le j\le m\\
&\la(t)=\la_0\,t^s+\text{(higher terms)},\,\,\la_0\in \BR^*.
\end{split}
\end{eqnarray*}
The assumption (b) implies that  for $1\le j\le m$,
\begin{eqnarray*}
i\,\left(\overline{\frac{\partial f_P^I}{\partial z_j}}(\bfa,\bar\bfa)/\bar \al -
\frac{\partial f_P^I}{\partial \bar z_j}(\bfa,\bar\bfa)/\al
\right)=\begin{cases}
0&\quad d-p_j-q<s+p_j\\
\la_0 \,a_j,&\quad
d-p_j-q=s+p_j.
\end{cases}
\end{eqnarray*}
Define $J\subset \{1,\dots, m\}$ by  $J:=\{j\,|\,d-p_j-q\,=\,s+p_j\}$.
Assume that  $J=\emptyset$.
Then we have the equality 
$$\overline{df_P^I}(\bfa,\bar\bfa)=\frac{\bar \al}{\al}\times\bar
 df^I(\bfa,\bar\bfa),\,\,j\le m$$
which implies  $f_P^I:\BC^{*I}\to \BC$ has a critical point
at
$\bfz=\bfa$ by Proposition \ref{critical point}. This is a contradiction to the strong non-degeneracy.
Thus we have shown that $J\ne \emptyset$.
We consider the differential:
\begin{eqnarray*}\begin{split}
\frac {d}{dt}f^I(\bfz(t),\bar\bfz(t))&=
\sum_{j=1}^m\frac{\partial f^I}{\partial
 z_j}(\bfz(t),\bar\bfz(t))\frac{dz_j(t)}{dt}
+\sum_{j=1}^m\frac{\partial f^I}{\partial\bar z_j}(\bfz(t),\bar\bfz(t))\frac{d\bar z_j(t)}{dt}\\
&=q\al\, t^{q-1}+\text{(higher terms)}.
\end{split}\end{eqnarray*}
The two terms of the right side of the first row can be written as follows.
\begin{eqnarray*}\begin{split}
&\left(\frac{d\bfz(t)}{dt},\overline{df^I}(\bfz(t),\bar\bfz(t))\right)
=\left(P\bfa,\overline{ df_P^I}(\bfa,\bar\bfa)\right)t^{d-1}
+\text{(higher terms)},\\
& \left(\frac{d\bar\bfz(t)}{dt},\overline{\bar
		  df^I}(\bfz(t),\bar\bfz(t)
\right)
=\left(P\bar\bfa,
 \overline{\bar df_P^I}(\bfa,\bar\bfa)\right)t^{d-1}
+\text{(higher terms)}\\
\end{split}
\end{eqnarray*}
where $P\bfa=(p_1a_1,\dots, p_ma_m)$ and 
$P\bar\bfa=(p_1\bar a_1,\dots, p_m\bar a_m)$.
Thus we get 
\begin{eqnarray*}\begin{split}
 q\al\, t^{q-1}&+\text{(higher terms)}=\\
&\left((P\bfa,\overline{ df_P^I}(\bfa,\bar\bfa))+(P\bar\bfa,
 \overline{\bar df_P^I}(\bfa,\bar\bfa))
\right)\,t^{d-1}+\text{(higher terms)}.
\end{split}
\end{eqnarray*}
Observe that
\begin{eqnarray*}\begin{split}
\Re\left(P\bfa,i\,\frac{ \overline{ df_P^I}(\bfa,\bar\bfa)}{\bar\al}\right)&+\Re
\left(P\bar\bfa,i\,\frac{ \overline{\bar
		  df_P^I}(\bar\bfa,\bar\bfa)}{\bar\al}\right)\\
&=
\Re\left(P\bfa,i\,\frac{ \overline{ df_P^I}(\bfa,\bar\bfa)}{\bar\al}-
i\,\frac{ {\bar df_P^I}(\bar\bfa,\bar\bfa)}{\al}\right )\\
&=\Re\left(\sum_{j\in J} \la_0\,|a_j|^2\,p_j\right)
=\la_0\sum_{j\in J} |a_j|^2p_j
\ne 0
\end{split}
\end{eqnarray*}
as $J\ne \emptyset$.
Thus we see that
 \begin{multline*}
\left(P\bfa,\,\overline{ df_P^I}(\bfa,\bar\bfa)\right)+\left(P\bar\bfa,\,
 \overline{\bar df_P^I}(\bfa,\bar\bfa)\right)
=\\
\al\, i\,\left((P\bfa,\,i\,\frac{\overline{ df_P^I}(\bfa,\bar\bfa)}{\bar \al})+
(P\bar\bfa,\,i\,\frac{ \overline{\bar df_P^I}(P\bar\bfa,\bar\bfa)}{\bar\al})
\right)\ne 0.\\
\end{multline*}
This implies that
$q=d$ (namely $f_P^I(\bfa,\bar\bfa)\ne 0$) and 
\begin{eqnarray*}\begin{split}
 q\,\al&=(P\bfa,\,\overline{ df_P^I}(\bfa,\bar\bfa) )+
(P\bar\bfa,\,\overline{\bar df_P^I}(\bfa,\bar\bfa)),\quad\text{or}\,\,\\
qi&=\left(P\bfa,\,i\,\frac{ \overline{\bar df_P^I}(\bfa,\bar\bfa)}{\bar\al}\right)+
\left(P\bar\bfa,\,i\,\frac{ \overline{\bar df_P^I}(\bar\bfa,\bar\bfa)}{\bar\al}\right).
\end{split}\end{eqnarray*}
Taking the real part of the last equality,
we get  an obvious contradiction:
\[ 0=\Re(\sum_{j\in J} \la_0\,|a_j|^2\,p_j)=\sum_{j\in J} \la_0\,|a_j|^2\,p_j\ne 0.
\]
\end{proof}
\begin{Observation}\label{tangent-space}
Let $\bfw\in f\inv(\eta),\,\eta\ne 0$ be a smooth point. Then the
 tangent space $T_{\bfw}f\inv(\eta)$ is the real subspace of $\BC^n$
 whose vectors are
 orthogonal in $\BR^{2n}$ to the two
 vectors
\[
  i\,(\overline{d\log f}-\bar d\log f)(\bfw,\bar\bfw),\,\,
(\overline{d\log f}+\bar d\log f)(\bfw,\bar\bfw).
\]
\end{Observation}
\begin{proof} 
Assume that $\bfz(t),\,-\eps\le t\le \eps$ is a smooth
 curve in $f\inv(\eta)$
with 
  $\bfz(0)=\bfw$ and 
$\cV(\bfw)=\bfv\in T_{\bfw}f\inv(\eta)$. 
The assertion follows from the next calculation.
\begin{eqnarray*} 
\begin{split}
&\frac d{dt}\log\,f(\bfz(t),\,\bar\bfz(t))=
 \frac{\Re\log f(\bfz(t),\,\bar\bfz(t))}{dt}|_{t=0}
 -{\frac{ d}{dt}}\left( \Re (i\, \log f(\bfz(t),\bar\bfz(t))\right)_{t=0}
\\
& =\Re(\bfv,\,(\overline{d\log f}+\bar d\log f)(\bfw,\bar\bfw))+\Re(\bfv, i\,( \overline{d\log f}-{\bar d\log f})(\bfw,\bar\bfw )).
\end{split}
\end{eqnarray*}
\end{proof}
Now we are ready to prove the existence of the Milnor fibration of the
first description.
\begin{Theorem}\label{MilnorFibering1}{\rm (Milnor fibration, the first description)}
Let $f(\bfz,\bar\bfz)$ be a strongly
 non-degenerate convenient mixed function.
 There exists a positive number $r_0$ such that 
\[
 \varphi=f/|f|: S_r^{2n-1}\setminus K_r\to S^1
\]
is a  locally trivial fibration for any $r$ with $0<r\le r_0$.
\end{Theorem}
\begin{proof}
Taking $r_0,\,r_1,\,\de_0$ sufficiently small so that 
$f\inv(\eta)$ and $S_r^{2n-1}$ intersect transversely for any $\eta\in
 \BC^*$
with $|\eta|\le \de_0$ and $r_1\le r\le r_0$ by Lemma
 \ref{MilnorFibering2}.
Combining with Observation \ref{tangent-space}, the transversality
 implies  that the three vectors
\[
\bfz,\, i\,(\overline{d\log f}-\bar d\log f)(\bfz,\bar\bfz),
\,\,(\overline{d\log f}+\bar d\log f)(\bfz,\bar\bfz)
\]
are linearly independent over $\BR$ on 
$\{\bfz\in S_r\,|\,0< |f(\bfz,\bar\bfz)|\le \de_0\}$.
Therefore  we can construct a horizontal vector field $\mathcal V$ for $\varphi$
on $S_r^{2n-1}\setminus K_r$  so that 
\begin{enumerate}
\item $\Re(\mathcal V(\bfz),\,i\,(\overline{d\log f}-\bar d\log
 f)(\bfz,\bar\bfz))=1$
and $\Re(\mathcal V(\bfz),\bfz)=0$ for any $\bfz\in S_r^{2n-1}-K_r$, and
moreover
\item  $\Re(\mathcal V(\bfz),\,(\overline{d\log f}+\bar d\log
 f)(\bfz,\bar\bfz))=0$ for $\bfz\in S_r$ with $|\,0< |f(\bfz,\bar\bfz)|\le \de_0 $.
\end{enumerate}
We show that the integral curve of $\cV$ does not approach to $K_r$.
In fact, assume that $\bfz(t),\,-\eps\le t\le \eps$ be an integral
 curve
with 
  $\bfz(0)=\bfw$,
$\cV(\bfw)=\bfv$. 
As we have seen in Observation \ref{tangent-space},
\begin{eqnarray}\label{abs-value}
\begin{split}
\frac d{dt}\log|f(\bfz(t),\,\bar\bfz(t))|=
 \Re(\bfv,\,(\overline{d\log f}+\bar d\log f)(\bfw,\bar\bfw)).
\end{split}
\end{eqnarray}
Therefore the condition (2)  guarantees that
 $\mathcal V(\bfz)$ is tangent to the level real hypersurface
of real codimension 1,
 $|f|_{\bfz}:=\{\bfw\in \BC^n\,|\,|f(\bfw)|=|f(\bfz)|\}$.
Thus it is obvious that $\mathcal V$ is integrable for any finite time
 interval
and we get the local triviality by the integration of $\mathcal V$.
\end{proof}

\subsection{Equivalence of two Milnor fibrations}
Take positive numbers $r,\,\de_0$ with $\de_0\ll r$  as in Theorem  \ref{Milnor2}. 
We compare the 
two fibrations
\begin{eqnarray*}
& f: \partial E(r,\de_0)\to S_{\de_0}^1,\qquad
\varphi: S_r^{2n-1}\setminus K_r\to S^1
\end{eqnarray*}
and we will show that they are isomorphic.
However the proof is much more complicated  compared with the case
of holomorphic functions. The reason is that
we have to take care of the  two  vectors
\[
 i\,(\overline{d\log f}-\bar d\log f),\, \overline{d\log f}+\bar d\log f
\]
which are not perpendicular.
(In the holomorphic case, the proof is easy as the two vectors reduce to 
the perpendicular vectors
$i \,\overline{d\log f},\,\overline{d\log f}$.)
Consider a smooth curve
$\bfz(t),\,-1\le t\le 1$, with $\bfz(0)=\bfw\in B_r^{2n}\setminus V$ and 
$\bfv=\frac{d\bfz(t)}{dt}(0)$. Put $\bfv=(v_1,\dots, v_n)$.
First from (\ref{argument-derivative}) and (\ref{abs-value}), we observe that
\begin{eqnarray*}
\begin{split}
& \frac{\log f(\bfz(t),\,\bar\bfz(t))}{dt}|_{t=0}=
\sum_{j=1}^n\left( v_j \,\frac{\partial \log f}{\partial z_j}(\bfw,\bar\bfw)
+\bar v_j \, \frac{\partial \log f}{\partial \bar
z_j}(\bfw,\bar\bfw)
\right)\\
&\qquad =\Re(\bfv,\,(\overline{d\log f}+\bar d\log f)(\bfw,\bar\bfw))+
i\,\Re(\bfv,\,i\,(\overline{d\log f}-\bar d\log f)(\bfw,\bar\bfw)).
\end{split}
\end{eqnarray*}
Define two vectors on $\BC^n-V$:
\begin{eqnarray*}
\begin{split}
\bfv_1(\bfz,\bar\bfz)&= \overline{d\log f}(\bfz,\bar\bfz)+\bar d\log
 f(\bfz,\bar\bfz)\\
 \bfv_2(\bfz,\bar\bfz)&=i\,(\overline{d\log f}(\bfz,\bar\bfz)-\bar d\log
 f(\bfz,\bar\bfz))
 \end{split}\end{eqnarray*}
 The above equality is translated as 
 \begin{eqnarray}\label{log-derivative}
\begin{split}
 \frac{\log f(\bfz(t),\bar\bfz(t))}{dt}|_{t=0}=
\Re(\bfv,\bfv_1(\bfw,\bar\bfw))+
i\,\Re(\bfv,\bfv_2(\bfw,\bar\bfw)).
\end{split}
\end{eqnarray}
 
 The following will play the  key role for the equivalence of two fibrations:
\begin{Lemma}\label{key lemma}
Under the same assumption as in Theorem \ref{MilnorFibering1},
there exists a  positive number $r_0$ so that for any $\bfz$ with
$\|\bfz\|\le r_0$ and $f(\bfz,\bar\bfz)\ne 0$, the three vectors
\[
 \bfz,\quad \bfv_1(\bfz,\bar\bfz),
\quad \bfv_2(\bfz,\bar\bfz)
\]
are either (i) linearly independent over $\BR$
or (ii) they are linearly dependent  over $\BR$ and 
the relation can be written as
\begin{eqnarray}\label{eq3}\begin{split}
& \bfz=a\, \bfv_1(\bfz,\bar\bfz)+b\,
 \bfv_2(\bfz,\bar\bfz), \,a,b\in \BR.
\end{split}
\end{eqnarray}
and the coefficient $a$ is positive.
\end{Lemma}
\begin{proof}
First observe that the pairs 
\begin{eqnarray*}\begin{split}
&\cP_1=\left\{\bfv_1(\bfz,\bar\bfz),  \,
\bfv_2(\bfz,\bar\bfz)\right\},\,\,
\cP_2=\left\{\bfz, \, \bfv_2(\bfz,\bar\bfz)
\right\}
\end{split}
\end{eqnarray*}
are respectively linearly independent over $\BR$ by  Lemma
 \ref{MilnorFibering2}, Lemma \ref{ExistenceRadius1}
and the above equality. Assume that the assertion does not hold.
Consider the real analytic variety
$W$ where the three vectors are linearly dependent over $\BR$.
Let us consider the open set  $U=\BC^n\setminus V$.
Then $W\cap U$ has a finite  number of connected components.
The sign of the coefficient  $a$ in (\ref{eq3}) is constant on each component, as long as
 they are near enough to the origin. 
This is the result of the linear independence of
$\bfz,\,\bfv_2(\bfz,\bar\bfz)$.
We will show that
this sign is positive.
We use  the Curve Selection Lemma (\cite{Milnor,Hamm1}) to find an analytic curve
$\bfz(t),\,0\le t\le 1$, such that
$\bfz(0)=O$ and $\bfz(t)\notin V$ for $t\ne 0$ and there exist real
 valued functions
$\la(t),\mu(t)$ so that 
\[
 \bfz(t)=\la(t)\,\bfv_1(\bfz,\bar\bfz)+\mu(t)\,
\bfv_2(\bfz,\bar\bfz).
\]
Let $I=\{j\,|\,z_j(t)\not \equiv 0\}$. We may assume that $I=\{1,\dots,
 m\}$
and we do the argument in $\BC^I$. 
We consider the Taylor expansions of $\bfz(t)$ and $f(\bfz(t),\bar\bfz(t))$, 
and the Laurent expansions of $\la(t)$ and $\mu(t)$:
\begin{eqnarray*}
\begin{split}
&z_j(t)=a_j\,t^{p_j}+\text{(higher terms)},\quad a_j\in \BC^*,\,
 p_j\in \BN, \,1\le j\le m,\\
&f(\bfz(t),\bar\bfz(t))=\al\,t^{\ell}+\text{(higher terms)},\,\al\in
 \BC^*,\,
\ell\in \BN\\
&\la(t)=\la_0\, t^{\nu_1}+\text{(higher terms)},\,\,\la_0\in \BR^*,\,
\nu_1\in \BZ\\
&\mu(t)=\mu_0\, t^{\nu_2}+\text{(higher terms)},\,\,\mu_0\in \BR^*,\,
\nu_2\in \BZ.
\end{split}
\end{eqnarray*}
First we consider the equality:
\begin{eqnarray}\begin{split}\label{eqeq1}
 z_j(t)=&\la(t)\,(\frac{\overline{\partial f}}{\partial z_j}/\bar f+
\frac{\partial f}{\partial \bar z_j}/f)(\bfz(t),\bar\bfz(t))\\
&+\mu(t)\,
i\,(\frac{\overline{\partial f}}{\partial z_j}/\bar f-
\frac{\partial f}{\partial \bar z_j}/f)(\bfz(t),\bar\bfz(t)) ,\,j=1,\dots, m.
\end{split}\end{eqnarray}
Put $P=(p_1,\dots, p_m)$, $\bfa=(a_1,\dots, a_m)$ and $d=d(P,f^I)$.
Then we observe that
\begin{eqnarray*}
\begin{split}
&\frac{\overline{\partial f}}{\partial
 z_j}(\bfz(t),\bar\bfz(t))/\bar f(\bfz(t),\bar\bfz(t))=\left(\frac{\overline{\partial
 f_P^I}}{\partial z_j}(\bfa,\bar\bfa)/\bar\al\right) t^{d-p_j-\ell}+\text{(higher
 terms)},\\
&\frac{\partial f}{\partial \bar z_j}
(\bfz(t),\bar\bfz(t)))/f(\bfz(t),\bar\bfz(t))=
\left(\frac{\partial
 f_P^I}{\partial \bar z_j}(\bfa,\bar\bfa)/\al \right)t^{d-p_j-\ell}+\text{(higher
 terms)}
\end{split}
\end{eqnarray*}
Thus comparing the equality (\ref{eqeq1}), we see that
\[
 p_j\ge \min\{\nu_1+d-p_j-\ell,\nu_2+d-p_j-\ell\}.
\]
 To avoid the repetition of the similar
 argument and to treat the cases $\nu_2=\nu_1$,  $\nu_2<\nu_1$
and $\nu_2>\nu_1$ simultaneously, we put $\nu_0=\min(\nu_1,\nu_2)$ and
we rewrite the expansions as $\la(t),\mu(t)$ as 
\[
 \begin{split}
&\la(t)=r_0\,t^{\nu_0}+\cdots,\quad r_0\in \BR\\
&\mu(t)=m_0\,t^{\nu_1}+\cdots,\quad m_0\in \BR.
\end{split}
\]
Here  we  have $r_0=0$ or $r_0=\la_0$
(respectively $m_0=0$  or  $m_0=\mu_0$) according to
$\nu_1>\nu_0$ or $\nu_1=\nu_0$ (resp.
$\nu_2>\nu_0$ or  $\nu_2=\nu_0$).
By (\ref{eqeq1}), we get
\begin{eqnarray*}\begin{split}
&\la_0\,\left(\frac{\overline{\partial f_P^I}}{\partial
 z_j}(\bfa,\bar\bfa)/\bar\al+
\frac{\partial f_P^I}{\partial \bar z_j}(\bfa,\bar\bfa)/\al\right)+
i\, m_0\,\left(\frac{\overline{\partial f_P^I}}{\partial
 z_j}(\bfa,\bar\bfa)/\bar\al-
\frac{\partial f_P^I}{\partial \bar z_j}(\bfa,\bar\bfa)/\al\right)\\
&\qquad =\begin{cases}
a_j\quad & p_j=d-p_j+\nu_0-\ell\\
0\quad &p_j>d-p_j+\nu_0-\ell.
\end{cases}
\end{split}
\end{eqnarray*}
More precisely we assert
\begin{Assertion}\label{transversality}
Put $p_{min}=\min\, \{p_i\,|\,i\in I\}$ and $K=\{i\in I\,|\,p_i=p_{min}\}$.
Then  we have
\begin{eqnarray}\label{Assertion}\begin{split}
&\la_0\,\left(\frac{\overline{\partial f_P^I}}{\partial
 z_j}(\bfa,\bar\bfa)/\bar\al+
\frac{\partial f_P^I}{\partial \bar z_j}(\bfa,\bar\bfa)/\al\right)
+\\
&\qquad m_0\, i\,\left(\frac{\overline{\partial f_P^I}}{\partial
 z_j}(\bfa,\bar\bfa)/\bar\al-
\frac{\partial f_P^I}{\partial \bar z_j}(\bfa,\bar\bfa)/\al\right)
 =\begin{cases}\label{eq9}
a_j\quad & j\in K\\
0\quad &j\notin K.
\end{cases}\\
\end{split}
\end{eqnarray}
\end{Assertion}
\begin{proof} 
We examine the equality (\ref{eqeq1}).
The order of $\|\bfz(t)\|$ is $p_{min}$.
On the other hand, the order
of $j$-th component of the right side of (\ref{eqeq1})  is greater than or
equal to $d-p_j+\nu_0-\ell$
and the coefficient of $t^{d-p_j+\nu_0-\ell}$
is given by the left side of (\ref{Assertion}). If there is an
 index $j\notin K$
such that this coefficient is non-zero, then the order of the right side
of  (\ref{eqeq1})
is strictly smaller than $d-p_{min}+\nu_0-\ell$
and the limit of the normalized vector of the right side 
has 0 coefficient on any $j\in K$ and  we have the contradiction 
to $(\ref{eqeq1})$.
\end{proof}
Thus we have proved (\ref{Assertion}).
Now we examine the next equality more carefully:
\begin{eqnarray}\label{eq2}\begin{split}\quad
& \frac{df(\bfz(t),\bar\bfz(t))}{dt}\\
&=\sum_{j=1}^n\left(\frac{\partial f(\bfz(t),\bar\bfz(t))}{\partial z_j}
\frac{dz_j(t)}{dt}+\frac{\partial f(\bfz(t),\bar\bfz(t))}{\partial\bar z_j}
\frac{d \bar z_j(t)}{dt}
\right).\end{split}
\end{eqnarray}
The left hand side is simply 
\[
  \frac{df(\bfz(t),\bar\bfz(t))}{dt}=\al\,\ell\, t^{\ell-1}+\text{(higher terms)}.
\]
We introduce the complex vectors:
\begin{multline*}
  \begin{cases}&\bfv=(v_1,\dots, v_m),\,\,,v_j=\sqrt{p_j}\,f_j,\\
&\bfw=(w_1,\dots, w_m),\,\,w_j=\sqrt{p_j}\,f_{\bar j}
\end{cases}\\
\text{where}\quad
f_j=\frac{\partial f_P^I}{\partial z_j}(\bfa,\bar\bfa),\,\,
f_{\bar j}=\frac{\partial f_P^I}{\partial\bar z_j}(\bfa,\bar\bfa),\,\,
1\le j\le m.
\end{multline*}
The order of the right hand side of (\ref{eq2}) is greater than or equal to
$d-1$. Let  $R$ be  the coefficient of $t^{d-1}$ of the right side. By an
 easy calculation, we have
\begin{eqnarray*}\begin{split}
R&=\sum_{j=1}^m a_jp_j \,f_j +
\sum_{j=1}^m \bar a_jp_j\,f_{\bar j}\\
&=\sum_{j=1}^mp_j f_j
\left\{\la_0\left(\frac{\overline{f_j}}{\bar\al}+
\frac{f_{\bar j}}{\al}\right)+
i\,m_0 \left(\frac{\overline{f_j}}{\bar\al}-
\frac{f_{\bar j}}{\al}\right)\right\}\\
&\qquad\qquad\qquad +\sum_{j=1}^mp_j {f_{\bar j}}
\left\{\la_0\left(\frac{f_j}{\al}+\frac{\overline{f_{\bar j}}}{\bar\al}\right)
-i\, m_0\left(
\frac{f_j}\al-\frac{\overline{f_{\bar j}}}{\bar \al}\right)\right\}\\
&=\al\sum_{j=1}^m\left( p_j|f_j|^2+p_j|f_{\bar j}|^2\right)\left(
\frac{\la_0}{|\al|^2}+\frac{i\,m_0 }{|\al|^2}\right)
+\al\sum_{j=1}^m 2p_jf_jf_{\bar j}\left(
\frac{\la_0}{\al^2}-\frac{i\, m_0 }{\al^2}
\right)\\
&=\al (\|\bfv\|^2+\|\bfw\|^2)\left(
\frac{\la_0}{|\al|^2}+\frac{i\, m_0 }{|\al|^2}\right)+
2\al(\bfw,\bar\bfv)\left(
\frac{\la_0}{\al^2}-\frac{i\,m_0}{\al^2}
\right)
\end{split}
\end{eqnarray*}
Consider two complex numbers:
\[
     \be:=(\|\bar\bfv\|^2+\|\bfw\|^2)\left(
\frac{\la_0}{|\al|^2}+i\frac{m_0 }{|\al|^2}\right),
\,
\ga:=2(\bfw,\bar\bfv)\left(
\frac{\la_0}{\al^2}-i\frac{m_0 }{\al^2}
\right)
    \]
Using the Schwartz inequality,
we see that
\[
 |\ga|=2\frac{\sqrt{\la_0^2+{m_0}^2}}{|\al|^2}|(\bfw,\bar\bfv)|
\le 2\frac{\sqrt{\la_0^2+{m_0}^2}}{|\al|^2}\|\bar\bfv\|\,\|\bfw\|
\]
and by comparing with $|\be|$, we get
\[
 |\be|-|\ga|\ge\frac{\sqrt{\la_0^2+{m_0}^2}}{|\al|^2}|(\|\bar\bfv\|-\|\bfw\|)^2\ge 0.
\]
For the equality $|\be|=|\ga|$,
it is necessary that $\|\bar\bfv\|=\|\bfw\|$ and
 $|(\bfw,\bar\bfv)|=\|\bar\bfv\|\times \|\bfw\|$, or
\[
 \bfw=u\bar\bfv,\quad \exists u\in S^1\subset \BC.
\]
Note that this is equivalent to
$( f_{\bar 1},\dots, f_{\bar m})=u\,(\overline{f_1},\dots,\overline{f_m})$
which implies
$\bar df(\bfa,\bar\bfa)=u\,\overline{df}(\bfa,\bar\bfa)$.
This is a contradiction to the non-degeneracy assumption for 
$f^I(\bfz,\bar\bfz)$. Thus we conclude that $|\be|>|\ga|$ and $R\ne 0$.
\newline
Now the equality  (\ref{eq2}) says,
$\ell-1=d-1$ and 
\[\begin{split}
 &\ell\,\al\,=\,\al\,\be\,+\,\al\,\ga,\,\text{or}\\
&\ell=(\|\bar\bfv\|^2+\|\bfw\|^2)\left(
\frac{\la_0}{|\al|^2}+i\,\frac{m_0 }{|\al|^2}\right)+
2(\bfw,\bar\bfv)\left(
\frac{\la_0}{\al^2}-i\,\frac{m_0 }{\al^2}
\right)
\end{split}
\]
We now assert that  $\la_0>0$. Assume $\la_0\le 0$. Then 
$\Re (\be)\le 0$ and 
to get the
 equality
$\ell=\be+\ga$, we must have $|\ga|>|\be|$.
This is impossible as we have seen that $|\be|>|\ga|$.
See Figure \ref{Schwartz}.
\end{proof}
\begin{figure}[htb,here]
\setlength{\unitlength}{1bp}
\begin{picture}(600,150)(-100,0)
\put(0,130){\special{epsfile=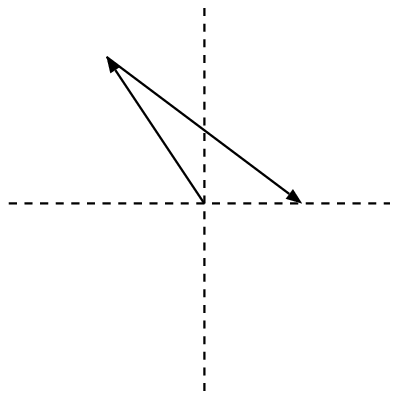 vscale=1 hscale=1}}
\put(25,100){$\be$}
\put (60,100){$\ga$}
\put(80,60){$\ell$}
\put(45,60){$O$}
\end{picture}
\caption{If $\la_0\le 0$,   $|\be|<|\ga|$}
\label{Schwartz}
\end{figure}
Now we are ready to prove the  isomorphism theorem:
\begin{Theorem}\label{equivalence}
Under the same assumption as in Theorem \ref{MilnorFibering1},
the two fibrations
\begin{eqnarray*}
& f: \partial E(r,\de_0)\to S_{\de_0}^1,\qquad
\varphi: S_r^{2n-1}\setminus K_r\to S^1
\end{eqnarray*}
are topologically isomorphic.
\end{Theorem}
\begin{proof}
The proof  is done  as in the case of Milnor fibrations of a holomorphic
 function
(\cite{Okabook}).
We will construct a vector field $\mathcal V$ on 
\[
 E^c(r,\de_0):=B_r\setminus \Int(E(r,\de_0))=
\{\bfz\in B_r\,|\, |f(\bfz,\bar\bfz)|\ge \de_0\}
\]
so that
\begin{eqnarray}\label{eq4}\begin{cases}
& \Re(\mathcal V(\bfz),\bfv_2(\bfz,\bar\bfz))=0,\\
&\,\,\Re(\mathcal V(\bfz),\bfv_1(\bfz,\bar\bfz))>0,\\
&\qquad\,\, \Re(\mathcal V(\bfz),\bfz)>0.
\end{cases}
\end{eqnarray}
Assume for a moment that we have constructed such a vector field.
Along the integral curve $h(t,\bfw)$ of $\mathcal V$ with $h(0,\bfw)=\bfw$, 
the argument of \nl $f(h(t,\bfw),\bar h(t,\bfw))$ is
 constant and the absolute value $|f(h(t,\bfw))|$ and
the norm $\|h(t,\bfw)\|$ are monotone
 increasing. The integral curve is well-defined as long as 
$h(t,\bfw)$ is inside $E^c(r,\de_0)$.
 For each $\bfw\in E^c(r,\de_0)$, there exists a unique
 $\tau(\bfw)$
such that  $\|h(\tau(\bfw))\|=r$.
Thus this gives a topological isomorphism
$\psi: \partial E(r,\de_0)\to S_r^{2n-1}\setminus N_r$
which is defined by $\psi(\bfw)= h(\tau(\bfw),\bfw)$
\[
 \begin{matrix}
 \partial E(r,\de_0)&\mapright{f}& S_{\de_0}^1\\
\mapdown{\psi}&&\mapdown{1/\de_0}\\
 S_r^{2n-1}\setminus \Int\,N_r&\mapright{\varphi}& S^1
\end{matrix}\]
where $N_r=S_r^{2n-1}\cap\{\bfz\,|\, |f(\bfz)|\le \de_0\}$.
 As $N_r\cong D({\de_0})^* \times K_r$
with $D(\de_0)^*=\{\eta\in \BC|0<|\eta|\le \de\}$, the restriction
$\varphi:S_r^{2n-1}\setminus N_r\to S^1$ is isomorphic to the Milnor
 fibration
$\varphi:S_r^{2n-1}\setminus K_r\to S^1$.

For the construction of $\mathcal V$, we use Lemma \ref{key lemma}.
 Take a point $\bfw\in E^c(r,\de_0)$.
If the three vectors
$\bfv_1(\bfw,\bar\bfw),\,\bfv_2(\bfw,\bar\bfw),\,\bfw$
are linearly independent over $\BR$, it is also linearly 
independent over a small open neighborhood $U(\bfw)$.
It is easy to construct locally
$\mathcal V$ on $U(\bfw)$, satisfying the above property (\ref{eq4}).
If the three vectors are linearly dependent over $\BR$, consider the 
expression:
\begin{eqnarray*}\begin{split}
& \bfw=a\, \bfv_1(\bfw,\bar\bfw)+b\,
 \bfv_2(\bfw,\bar\bfw), \,a,b\in \BR, 
\end{split}
\end{eqnarray*}
with $a>0$, we construct $\mathcal V$  on a neighborhood
 $U(\bfw)$
of
$\bfw$  so that 
\begin{eqnarray*}
 \Re(\mathcal V(\bfz),\bfv_2(\bfz,\bar\bfz))=0,
\,\,\Re(\mathcal V(\bfz),\bfv_1(\bfz,\bar\bfz))>0.
\end{eqnarray*}
on $U(\bfw)$.
Note that
\[
 \Re(\bfw,\mathcal V(\bfw))=a\,\Re(\bfv_1(\bfw,\bar\bfw),\mathcal V(\bfw))>0
\]
If $U(\bfw)$ is sufficiently small, this inequality holds on
 $U(\bfw)$.
Consider the open covering
$\cU=\{U(\bfw)\,|\, \bfw\in E^c(r,\de_0)\}$.
Taking a locally finite refinement $\cU'$ of this covering,
 we glue together  vector fields constructed locally  on 
each open set in $\cU'$  using a
 partition of unity as usual.
\end{proof}
\subsection{Polar weighted homogeneous polynomial and its Milnor fibration}
Consider a mixed polynomial $f(\bfz,\bar\bfz)$ which is a radially 
weighted homogeneous polynomial of type $(q_1,\dots, q_n;d_r)$ and a polar
weighted homogeneous polynomial  of type $(p_1,\dots, p_n;d_p)$.
Put $V=f\inv(0)$ as before. 
Then $f:\BC^n\setminus V\to \BC^*$ is a locally trivial fibration
\cite{OkaPolar}. We call it {\em the global fibration}.
On the other hand, the Milnor fibration of the first type:
\[
 \varphi:=f/|f|: S_r\setminus K_r\to S^1,\quad K_r=f\inv(0)\cap S_r
\]
always exists for any $r>0$ and the isomorphism
class does not depend on the choice of $r$.
This can be shown easily, using the polar action.
We simply use the polar action to show the local triviality:
\[\begin{split}
&\psi: \varphi\inv(\theta)\times (\theta-\pi,\theta+\pi)\to
   \varphi\inv((\theta-\pi,\theta+\pi))\\
&\qquad \psi(\bfz,\theta+\eta):=(z_1\exp(i\,p_1\eta /d_p),\dots,
   z_n\exp(i\,p_n\eta /d_p))
\end{split}
\]
Now we have the following assertion
which is  a generalization of  the same 
assertion for  weighted homogeneous polynomials.
\begin{Theorem}\label{Polar Milnor} Let  $f(\bfz,\bar\bfz)$ be a polar weighted polynomial
 as above. We assume that the radial weight vector ${}^t(q_1,\dots,
 q_n)$
is strictly positive. Then
the two fibrations
\[
 f:f\inv(S_\de^1)\to S_\de^1,\quad \varphi=f/|f|: S_r^{2n-1}-K_r\to S^1,
\]
are isomorphic for any $r>0$ and $\de>0$.
\end{Theorem}
\begin{proof}
First,  observe that the isomorphism class of  the global fibration
\[
 f:f\inv(S_\de^1)\to S_\de^1
\]
does not depend on $\de>0$.
This follows from the commutative diagram:
\[
 \begin{matrix}
f\inv(1)&\mapright{\phi_\de}&f\inv(\de)\\
\mapdown{f}&&\mapdown{f}\\
S^1&\mapright{\de}&S_\de^1
\end{matrix},\quad
\phi_\de:(\bfz,\bar\bfz) \mapsto \root{d}\of {\de}\circ(\bfz,\bar\bfz)
\]
where $d=\rdeg\, f$ and $\circ$ denotes the $\BR^*$ action by the radial weights.
Now the global fibration $f:f\inv(S_{\de}^1)\to S_{\de}^1$ is 
isomorphic to  the second fibration
$\varphi:S_r^{2n-1}\setminus K_r\to S^1$ as follows.
For any $\bfz\in f\inv(S_{\de}^1)$, consider the orbit of the radial action
$\tau\mapsto \tau\circ \bfz=(\tau^{q_1}z_1,\dots, \tau^{q_n}z_n)$, $\tau>0$. There
 exists a unique positive real number $\tau=\tau(\bfz)$ so that 
$\|\tau(\bfz)\circ \bfz\|=r$
by the strict positivity assumption of $Q$. Put $\psi: f\inv(S_{\de}^1)\to
 S_r^{2n-1}\setminus K_r$ by 
$\psi(\bfz)=\tau(\bfz)\circ \bfz$ and $\xi:S_{\de}\to S_1$ by
$\xi(\eta)=\eta/{\de}$. Then we have a canonical commutative diagram
which gives an isomorphism of two fibrations.
\[
 \begin{matrix}
f\inv(S_{\de}^1)&\mapright{f}& S_{\de}^1\\
\mapdown{\psi}&&\mapdown{\xi}\\
S_{r}^{2n-1}\setminus K_r&\mapright{\varphi}& S^1
\end{matrix}
\]
\end{proof}

The following is an important criterion for the connectivity of the
Milnor fiber of a polar weighted mixed polynomial.
\begin{Proposition}\label{0-connected}
Let $f(\bfz,\bar\bfz)$ be a  polar weighted mixed
 polynomial of $n$ variables $\bfz=(z_1,\dots,z_n)$.
We assume that $f\inv(0)$ has at least one mixed smooth point.
Then the fiber  $F:=f\inv(1)\subset \BC^n$ is connected.
\end{Proposition}
\begin{proof}
 Put $V=f\inv(0)$.
 Take two points $P,Q\in F$. Connect$P,Q$  by a path $\ell$
in  $\BC^n\setminus V$.
Then $f_{\sharp}(\ell)$ is a closed path in $\BC^*$
based at $1\in \BC^*$
 and let $s$ be the
 rotation number of $f_{\sharp}(\ell)$ around the origin.
Take a smooth point
$R$ in $V$ and take a small lasso $\omega$ around $V$ in the normal
 plane
at $R$. Connect $\omega$ to $P$ in $\BC^n\setminus V$ to get a closed
 path $\omega'$ at $P$. The image of $\omega'$ by $f$ is a closed loop
with the  rotation number 1 around the origin.
Take a new path
$\ell'={\omega'}^{-s}\circ \ell$. Then the image of $\ell'$ has 0
 rotation number around the origin
and thus it is homotopic to the constant loop at $1\in \BC^*$  by a homotopy
$k_t:f\circ\omega'\simeq c_1$, where $c_1$ is the constant path at $1$.
Now lift this homotopy by the radial and the  polar actions
 to get a path $\wtl k_1$ from $P$ to $Q$. Obviously $f\circ\wtl k_1 $
is the constant path $c_1$. Thus $\wtl k_1$ is a path in the fiber $F$
which connects $P$ and $Q$.
(For a holomorphic case, this assertion follows from the 
 Kato-Matsumoto theorem,
\cite{Kato-Matsumoto}).\end{proof}

\section{Curves defined by mixed functions}
In this section, we focus our  study to mixed  plane
 curves  ($n=2$).
\subsection{Holomorphic plane curves}\label{Holomorphic}
Assume that $C$ is a germ of a complex analytic curve defined by a 
convenient non-degenerate holomorphic function $f(z_1,z_2)$ and let
$\De_j,\,j=1,\dots, r$
be the 1-dimensional faces and $M_0,M_1,\dots, M_{r-1}, M_r$
 be the vertices of $\Ga(f)$ such that  $\De_j=\overline{M_{j-1}M_j}$
and 
$M_0,M_r$ are on the coordinate axes. 
Then each face function $f_{\De_j}$ can be
factorized
as 
\[
 f_{\De_j}(z_1,z_2)=c_j\, z_1^{a_j}\,z_2^{b_j}\prod_{i=1}^{\nu_j}
(z_1^{p_j}+\al_{j,i}\, z_2^{q_j}),\quad \gcd(p_j,q_j)=1
\]
where  $\al_{j,1},\dots,\al_{j,\nu_j}$ are mutually distinct.
Then any toric modification with respect to a regular simplicial cone
subdivision
$\Si^*$ of the dual Newton diagram $\Ga^*(f)$ gives a good resolution of 
$f:(\BC^2,O)\to (\BC,0)$. Let $P_j$ be the weight vector of the face
$\De_j$.
 Each vertex $P$ of $\Si^*$ gives an
exceptional
divisor
$\what E(P)$ and the strict transform $\widetilde C$ intersects
with $\what E(P)$ if and only if $P=P_j$ for some $j=1,\dots,
r$.
In the case $P=P_j$, $\what E(P_j)\cap \widetilde C$ is $\nu_j$ point which
corresponds to
 irreducible components
associated with $f_{\De_j}$. 
The vertices $M_1,\dots, M_{r-1}$ do not contribute to the irreducible
components.
The number of irreducible components of $(C,O)$ is 
given by $\sum_{i=1}^{r}\nu_i$.
Note that 
$1+\sum_{i=1}^{r}\nu_i$ is the number of integral points on $\Ga(f)$
(\cite{Okabook}).
The situation for mixed polynomials is more complicated as we will see later.

\subsection{Mixed curves}
Now we consider curves defined by a
 mixed function with the same Newton boundary as in the previous subsection.
Let $f(\bfz,\bar\bfz)$ be a non-degenerate convenient mixed function
with two variables $\bfz=(z_1,z_2)$
and let $C=f\inv(0)$. Let
\[
 \varphi:Y\mapright{\omega} X\mapright{{\what \pi}} \BC^2
\]($Y=\cR X$, $\omega=\omega_{\BR}$ or $\cP X$ and $\omega=\omega_p$)
be the resolution map, described in Theorem \ref{real-resolution}
and Theorem \ref{polar-resolution}.
Let $\widetilde E(P)=\omega\inv(\what E(P))$ for a vertex $P$ of $\Si^*$.
\subsubsection{Simple vertices}
%
A vertex $M=(a,b)\in \Ga(f)$ is called {\em simple} if $f_M$ contains
only a single monomial $z_1^{a_1}\,z_2^{b_1}\,\bar z_1^{a_2}\,\bar z_2^{b_2}$
such that $a=a_1+a_2,\,b=b_1+b_2$. Otherwise we say $M$ is a {\em multiple
vertex}
of $\Ga(f)$.
\begin{Example}{\rm 
Let $f(\bfz,\bar\bfz)=z_1^3+t\, z_1^2\,\bar
z_1+z_2^2$.
Then $\Ga(f)$ has one face with  edge vertices $M_1=(3,0)$ and
$M_2=(0,2)$. $f(\bfz,\bar\bfz)$ is a radially weighted homogeneous
polynomial
of type $(2,3;6)$.
The vertex $M_1$ is a multiple vertex as 
$f_{M_1}(\bfz,\bar \bfz)=z_1^3+t\, z_1^2\bar z_1$.}
\end{Example}
\begin{Lemma} Suppose $M=(n,0)$ and let 
$f_M(z_1,\bar z_1)=\sum_{j=0}^n c_j \,z_1^j\bar z_1^{n-j}$.
Consider the factorization
$f_M(z_1,\bar z_1)=c\,\prod_{j=1}^n (z_1-\al_j\,\bar z_1)$.
Then  
$V^*:=\{z_1\in \BC^*\,|\, f_M(z_1,\bar z_1)=0\}$ is empty  if and only if 
 $|\al_j|\ne 1$ for any $j=1,\dots,n$.
\end{Lemma}
\begin{proof}
Let $V_j^*:=\{z_1\in \BC^*\,|\,z_1=\al_j\bar z_1\}$. 
Then $V^*=\bigcup_{j=1}^n V_j^*$.
It is easy to see that $V_j^*$ is not empty if and only if $|\al_j|=1$.
\end{proof}
Note that $f_M(z_1,\bar z_1)$ is non-degenerate if and only if $V^*=\emptyset$.
For an inside vertex $M_j$ (namely, $M_j$ is not on the axis), the criterion for non-degeneracy
 of the function $f_{M_j}(\bfz,\bar\bfz)$ is not so
simple.
\begin{Example}{\rm 
Consider
 \[
	 C:=\{\bfz\in \BC^2\,|\,f_M(\bfz,\bar\bfz)=t\, z_1z_2\,+\,z_1\bar z_2\,+\,\bar z_1 z_2\}.
 \]
We assert that}
\end{Example}
\begin{Assertion}
 $f_M\inv(0)\subset \BC^{*2}$ is
non-empty if and only if  $|t|\le 2$. 
 $f_M$ is non-degenerate if and
only if $|t|> 2$ or $0<|t|<2$.
\end{Assertion}
\begin{proof}
Put 
\[
 z_1=\rho_1\exp(\theta\, i),\,\,z_2=\rho_2\exp(\eta\, i ),
\,\, t=\xi\exp(\al\,i).
\]
Then  we  see that $C$ is radially homogeneous and it is defined by
\[
 C:\quad 2\,\cos \left( -\theta+\eta \right) +\xi\,{e^{ i\, \left( a+\theta+\eta
 \right) }}=0.
\]
For the existence of non-trivial solutions, we need to have:
\begin{eqnarray}
 &\xi=|t|\le 2,\quad 
\begin{cases}\al+\theta+\eta=m\pi,\,\,\exists m\in \BZ\label{8sol}\\
2\cos(-\theta+\eta)+\xi (-1)^m=0
\end{cases}\\
&\text{or}\quad\xi=0,\quad 
\cos(-\theta+\eta)=0
\end{eqnarray}
Assume that $t\ne 0,\,|t|<2$.
The equation (\ref{8sol}) 
 has 8 solutions  with $0\le \theta,\eta<2\pi$ for $\xi<2$ and $4$
 solutions for $\xi=2$ or $0$.
We can show that $V=f_M\inv(0)$ is non-singular for $0\ne \xi<2$, using
 Proposition \ref{critical point}.
In fact, assume that $\overline {df}(\bfz)=\la \bar df(\bfz)$ with $|\la|=1$. Then we have
\[
 \bar t \bar z_2+z_2=\la z_2,\, \bar t \bar z_1+z_1=\la z_1,\,
tz_1z_2+(z_1\bar z_2+\bar z_1z_2)=0
\]
This implies that $\la^2=\pm 1$ and $|t|=0$ or $|t|=2$.
\end{proof}

\subsubsection{Link components}
Let $f(\bfz,\bar\bfz)$ be a mixed function
with two variables $\bfz=(z_1,z_2)$
and let $C=f\inv(0)$. 
The {\em link components} at the origin  are the components
of $S_{\eps}^{3}\cap C$ for a sufficiently small $\eps$. We are
interested 
in finding out how to compute the number of the
link components  of $C$ at the origin. 
Let us denote this number 
by
$\lkn(C,O)$  and we call $\lkn(C,O)$ {\em the link component number}.
Let us denote the number of components which are not the coordinate
axes
$z_1=0$ or $z_2=0$ by $\lkn^*(C,0)$.
In the case of $f$ being a holomorphic function,
$\lkn(C,O)$ is equal to the number of irreducible components of $(C,O)$,
which is  a combinatorial invariant, provided   $f$ is Newton
non-degenerate, as we have seen in the  previous section \S
\ref{Holomorphic}.
However for a generic mixed function, $\lkn(C,O)$ might be strictly
greater than the number of irreducible components (see 
Example \ref{counter-example} for example).

\begin{Theorem}\label{ln-formula}Assume that $f(\bfz,\bar\bfz)$ is a convenient non-degenerate mixed
polynomial of two variables $\bfz=(z_1,z_2)$
and let $C=f\inv(0)$.  Let $\cF$ be the set of 1-faces of $\Ga(f)$.
Assume that the vertices of $\Ga(f)$ are simple. Then 
the number of the link components $\lkn(C,O)$
is given be the formula:
\[
 \lkn(C,O)=\sum_{\De\in \cF}\lkn^*(f_\De\inv(0),O).
\]
\end{Theorem}

\noindent
{\em Proof.} 
Let $\Phi:\cR X\to \BC^2$ be the resolution of $f$
by the composite of a toric modification
${\what \pi}:X\to \BC^2$ and the normal real blowing-up $\omega: \cR X\to X$.
The simplicity of the vertices implies that 
$\Phi\inv(\what E(P))\cap \wtl C=\emptyset$ for any $P$ for which 
$\De(P)$ a vertex of $\Ga(f)$.
Thus by Theorem \ref{real-resolution}, it is immediate that
there is  one link  component of $(C,O)$
for every
connected component of 
$\widetilde E(P)=\Phi\inv(\what E(P))\cap\widetilde C$
with $\De(P)\in \cF$.
The assertion follows from this observation.
\qed

Now our interest is finding out how we can compute $\lkn^*(f_\De\inv(0),O)$.
In general, it is not so easy to compute this number but there is a
class for which the link number is easily computed.
\subsubsection{Good Newton polar boundary}\label{good polar}
Suppose that  $f(\bfz,\bar\bfz)$ is a mixed function of two variables
and let $\De$ be a face of the Newton boundary. Suppose that
$f_\De(\bfz,\bar\bfz)$ is also a polar weighted homogeneous polynomial. 
Let $Q={}^t(q_1,q_2)$ and $P={}^t(p_1,p_2)$ be the radial and the  polar weight vectors
and $d_r,d_p$ be the respective degree.
In general, the mixed face $\what\De(Q)$ is two-dimensional as
the possible monomial 
$z_1^{\nu_1}\,z_2^{\nu_2}\,\bar z_1^{\mu_1}\,\bar z_2^{\mu_2}$
satisfies two linear equations
\[
 (\nu_1+\mu_1)q_1+(\nu_2+\mu_2)q_2=d_r,\,\,
(\nu_1-\mu_1)p_1+(\nu_2-\mu_2)p_2=d_p.
\]
We say that $f_\De(\bfz,\bar\bfz)$ is {\em a good polar weighted polynomial} if $\dim\,\what \De=1$
and $f_\De(\bfz,\bar\bfz)$  factors as
\begin{eqnarray}\label{eq6}
 f_\De(\bfz,\bar\bfz)=c\,\bfz^{\bfm}\,\bar\bfz^{\bfn}\, \prod_{j=1}^k (z_2^a\,\bar
 z_2^{a'}-\la_j\, {z_1}^b\,\bar z_1^{b'})^{\mu_j}
\end{eqnarray}
with $a\ne a',\,b\ne b'$ and $\gcd(a,a',b,b')=1$.
Note that in this case, $p_1(b-b')=p_2(a-a')$ and non-zero.
We say that 
$f(\bfz,\bar\bfz)$ has {\em a good Newton polar boundary}
if for every face $\De$ of $\Ga(f)$, $f_\De(\bfz,\bar\bfz)$ is a good polar
weighted polynomial.

\begin{Lemma}\label{polar-factor} Assume that $f_\De(\bfz,\bar\bfz)$ is a good polar
 weighted polynomial and assume that a factorization of $f_\De(\bfz,\bar\bfz)$ 
is given as (\ref{eq6}).
Then 
 $f_\De(\bfz,\bar\bfz)$ is  non-degenerate if
 and only if $\mu_1=\dots=\mu_k=1$.
\end{Lemma}
\begin{proof}
Assume that $\mu_j\ge 2$ for some $j$. Then it is easy to see that 
$df(\bfz,\bar\bfz)=\bar df(\bfz,\bar\bfz)=0$ on 
$z_1^{b}\bar z_1^{b'}-\la_j z_2^{a}\bar z_2^{a'}=0$.
Thus it is degenerate. Assume that $\mu_j=1$ for any $j$.
As $f_\De$ is polar weighted, we only need  to show that
 $f_\De\inv(0)\cap \BC^{*2}$ is mixed non-singular. 
Take a point $\bfw\in f_\De\inv(0)\cap \BC^{*2}$ such that 
$w_2^{a}\bar w_2^{a'}-\la_1 w_1^{b}\bar w_1^{b'}=0$ for example.
Then we have
\begin{eqnarray*}\begin{split}
 df(\bfw,\bar\bfw) &=c\,\bfw^{\nu}\bar\bfw^\mu\,\prod_{j=2}^k (w_2^a\bar
 w_2^{a'}-\la_j \,{w_1}^b\bar w_1^{b'})\times\left(
-b\,\la_1\, w_1^{b-1}\bar w_1^{b'},aw_2^{a-1}\bar w_2^{a'}\right)\\
\bar df(\bfw,\bar\bfw) &=c\,\bfw^{\nu}\bar\bfw^\mu\,\prod_{j=2}^k (w_2^a\bar
 w_2^{a'}-\la_j\, {w_1}^b\bar w_1^{b'})\times\left(
-b'\,\la_1\, w_1^{b}\bar w_1^{b'-1},a'w_2^{a}\bar w_2^{a'-1}\right)
\end{split}\end{eqnarray*}
Suppose that $\overline{df}(\bfw,\bar\bfw)=u\,\bar df(\bfw,\bar\bfw)$
for some $u$ with $|u|=1$. This implies that 
$ b=b',\, a=a'$.
This does not happen as we have assumed that $a\ne a',\,b\ne b'$.
\end{proof}
\begin{Example}{\rm 
Let $f(\bfz,\bar\bfz)=z_1^5+\bar z_1 z_2 (z_1^2-\bar z_2^2)+ \bar
z_2^5$.
Then $\Ga(f)$ has three 1-faces and the corresponding face functions are
\[
 z_1^2\bar z_1(z_1^3\bar z_1^{-1}+ z_2),\,\,\bar z_1 z_2 (z_1^2-\bar z_2^2),\,\,
 z_2\bar z_2^2(-\bar z_1+ z_2^{-1}\bar z_2^3).
\]
Thus $f$ has a good Newton polar boundary.}
\end{Example}
\subsubsection{Good binomial polar weighted polynomial}
A polynomial $f(\bfz,\bar\bfz)=z_2^{a}\bar z_2^{a'}-\la z_1^{b}\bar
z_1^{b'}$
with $a\ne a',\,b\ne b'$, $\la\ne 0$ and $\gcd(a,a',b,b')=1$
 is called an  {\em irreducible binomial polar weighted homogeneous polynomial}.
It is irreducible as a mixed polynomial.
By Lemma \ref{polar-factor}, this is a basic polar weighted polynomial
for our purpose.
Put $c_1=b-b'$ and $c_2=a-a'$. Then the associated Laurent polynomial
in the sense of \cite{OkaPolar} is
\[
 g(z_1,z_2)=z_2^{c_2}-\la z_1^{c_1}.
\]
Let $C=\{f=0\}$ and $C'=\{g=0\}$. Note that $c_1,c_2\ne 0$ by the
polar weightedness. 
\begin{Lemma}\label{link-number-binomial} 
We have the equality:
\[
 \lkn^*(C,O)=\gcd(c_1,c_2)=\sharp(C')
\]
where $\sharp(C')$ is the number if irreducible components of $C'$.
\end{Lemma}
\begin{proof} It is easy to see that the number of irreducible
 components
of $C'$ in $\BC^{*2}$ is
$\gcd(c_1,c_2)$. We know that $C\cap \BC^{*2}$ and $C'\cap \BC^{*2}$
are homeomorphic by the same argument as in \cite{OkaPolar}.
We will show that $\lkn^*(C,0)=\gcd(c_1,c_2)$ without using this isomorphism.
We consider  components of $C$ in $\BC^{*2}$.
For this purpose, we use the polar modification.
So we put  $z_1=r_1\exp(\theta_1\,i )$ and $z_2=r_2\exp(\theta_2\,i )$.
Considering the conjugation  diffeomorphism, we may assume that
 $c_1,\,c_2 > 0$. 
For brevity we put $r_1=s_1^{c_2}$ and $\la=\rho^{c_2}\exp(\eta\,i)$
for some $s_1,\, \rho>0$.
Thus 
\begin{eqnarray*}\begin{split}
 f(\bfz,\bar\bfz)&=
r_2^{c_2}\exp(  c_2\,\theta_2\,i)
-\la\,
 r_1^{c_1}\exp( c_1\,\theta_1\, i)\\
&=r_2^{c_2}\exp(  c_2\,\theta_2\, i)-\rho^{c_2}(s_1^{c_1})^{c_2}\exp( ( c_1\,\theta_1+\eta)\,i)
\end{split}\end{eqnarray*}
Thus we have $r_2=\rho s_1^{c_1}$ and 
\[
  \exp(  c_2\,\theta_2\,i)-\exp( ( c_1\theta_1+\eta)\,i)=0.
\]
Put $c_0=\gcd(c_1,c_2)$ and write $c_i=c_0\, c_i'$ for $i=1,2$.
The above  equation is solved as follows.
\begin{eqnarray*}
 r_2=s_1^{c_1}\rho,\,\,
  c_2\,\theta_2\equiv c_1\theta_1+\eta\,\,\modulo \,\,2\pi.
 \end{eqnarray*}
 The last equality can be solved so that the component $C_j$ of $C$ is given as
 \begin{eqnarray*}
 C_j:= \{(s_1^{c_1}\rho\exp(\theta_1\,i ),s_1^{c_2c_1}
\exp(\theta_2\,i ))\,|\,\theta_2=\phi_k(\theta_1),
  \,\, 0\le \theta_1\le 2c_2'\pi\}\notag
\end{eqnarray*} 
where  $ \phi_k(\theta_1):={c_1'\theta_1}/{c_2'}-\eta/c_2+2k\pi/c_2$
 for $k=0,1,\dots, c_2-1$.
 For $k\ge c_0$, write $k=c_0k_1+k_0,\,0\le k_0<c_0$.
 Then $\phi_k(\theta_1)=\phi_{k_0}(\theta_1+2k_1\pi)$ and 
 $C_k=C_{k_0}$ as we have
\begin{eqnarray*}
\begin{split}
 C_k&=\{(r_1\exp(\theta_1\,i ),r_2\exp(\phi_k(\theta_1)\,i )\,|\,
0\le \theta_1\le 2c_2'\pi\}\\
&=\{(r_1\exp(\theta_1\,i ),r_2\exp(\phi_k(\theta_1)\,i ))\,|\,
2k_1\pi\le \theta_1\le 2c_2'\pi+2k_1\pi\}\\
 &=
 \{
 (r_1\exp((\theta_1+2k_0\pi)\,i ),r_2\exp(\phi_{k_0}(\theta_1+2k_0\pi)\,i )\,|\,0\le \theta_2\le 2c_2'\pi\}\\
 &=C_{k_0}.
 \end{split}
 \end{eqnarray*}
Thus 
we get  $\lkn^*(C,O)=c_0$. 
\end{proof}
\begin{Corollary} Let  $f_\De(\bfz,\bar\bfz)$ be  a good polar
 weighted polynomial which is factored as 
\[
 f_\De(\bfz,\bar\bfz)=c\,\bfz^{\nu}\bar\bfz^\mu\prod_{j=1}^k (z_2^a\bar
 z_2^{a'}-\la_j\, {z_1}^b\bar z_1^{b'})
\]
with $\gcd(a,a',b,b')=1,\,a\ne a',\,b\ne b'$ as in Lemma \ref{polar-factor} and let $C=f_\De\inv(0)$.
Then $\lkn^*(C)=k\gcd(a-a',b-b')$.
 \end{Corollary}
 \subsubsection{ Newton pseudo conjugate weighted homogeneous function}
Assume that $f(\bfz,\bar\bfz)$ is a non-degenerate 
Newton pseudo conjugate weighted homogeneous function.
Then for any face $\De$, we can write
\[
 f_\De(\bfz,\bar\bfz)\,=\,M h(\bfz,\bar\bfz)
\]
where $M$ is a mixed monomial and 
 $h$ is a $J$-conjugate weighted homogeneous polynomial
for some $J\subset \{1,2\}$. Thus we can
factorize $h$ as
\[
\iota_J^* h(\bfz,\bar\bfz)=
c\,\prod_{j=1}^k(z_2^{p_1}-\la_j\, z_1^{p_2}),\quad c\ne 0
\]
with $\gcd(p_1,p_2)=1$. In this case, it is easy to see that
\[
 \lkn^*(f_\De\inv(0))=k.
\]
Thus  we obtain a similar formula:
\begin{Proposition}
Assume that $f(\bfz,\bar\bfz)$ is a non-degenerate  convenient
Newton pseudo conjugate weighted homogeneous function. Then
\[
 \lkn(f\inv(0))+1=\text{number of integral points on }\,\,\Ga(f).
\]
\end{Proposition}
\subsubsection{Example of a radially weighted homogeneous polynomial
   with
a non-simple vertex }\label{non-simple}
The link number for a radially weighted homogeneous polynomial
 with a  non-simple vertex is more complicated,
as is seen by the next example.
\noindent
Consider the radially weighted homogeneous polynomial
$$f(\bfz,\bar\bfz)=z_1^3+c\,z_1\bar z_1^2-z_2^3$$
and put $C=f\inv(0)$.
Then $\Ga(f)$ consists of a single face with vertices $(3,0),\, (0,3)$.
It is easy to see that $f$ is non-degenerate if and only if $|c|\ne 1$.
The vertex (3,0) is not simple.
For $|c|< 1$, we have
\[
 z_2=z_1\omega^j \left(1+c \,\exp(-4\,\theta\,i )\right)^{1/3},\,\,j=0,1,2
\] where $\omega=\exp(2\,\pi \,i /3),\,z_1=r\exp(\theta\, i)$
and $\lkn(C,O)=3$. The function
$(1+c \,\exp(-4\,\theta \,i))^{1/3}$ is a well-defined 
single-valued function of $c,\, z_1$ with $|c|<1$ so that 
it takes value 1 for $c=0$.
Considering the family $f(\bfz,\bar\bfz,t)=z_1^3+\,c\,t\,z_1\bar
z_1^2-z_2^3$
for $0\le t\le 1$, we see that this curve is topologically the 
same as $z_1^3+z_2^3=0$.

Assume that $|c|>1$. Then
$(1+c \exp(-4\,\theta\, i))^{1/3}$ is not a single valued function
as a function of $0\le\theta\le 2\pi$.
However  we have a better expression. Put $z_1=r\exp(\theta\, i)$ and
$c=s\exp(\eta\, i)$.
\[
 z_2=s^{1/3}\,r\, \omega^j \exp(i\,\frac{-\theta+\eta}{3})
\left (1+\frac{\exp(4\,\theta\, i)}c\right
 )^{1/3},\,\,j=1,2,3
\]
where $0\le \theta\le 2\pi$. Note that $f\inv(0)\setminus\{O\}$
is a 3-sheeted covering over $\{z_1\ne 0\}$ and 
  three points over $\theta=0$ are
cyclically permuted by the 
monodromy $\theta:0\to 2\pi$. Thus this expression shows that 
$\lkn(C,O)=1$.  
It is also easy to see that this knot is topologically the same with
$z_1|z_1|^2-z_2^3=0$.
Thus we observe that {\em  the topology of a mixed singularities is not
a combinatorial invariant of $\Ga(f)$}.
\section{Milnor fibration for mixed polynomials with
  non-isolated singularities}
We consider a true strongly non-degenerate 
mixed polynomial $f(\bfz,\bar\bfz)$
which is not necessarily convenient.
Take a positive weight vector $P={}^t(p_1,\dots, p_j)\in N^+$ 
which is not strictly positive and we put 
\[
 I(P)=\{j\,|\,p_j=0 \},\qquad
J(P)=\{j\,|\,p_j\ne 0 \}.
\]
We consider the face function $f_P(\bfz,\bar\bfz)$ as a mixed polynomial
in variables $\{z_j|j\in J(P)\}$ and we consider the other variables
$\{z_i|i\in I(P)\}$ are fixed  non-zero complex numbers.
Thus it defines  a family of mixed polynomial functions parameterized by
$\bfz_{I(P)}=(z_i)_{i\in I(P)}$:
\[
 f_P: \BC^{J(P)}(\bfw_{I(P)})\to \BC,\quad
\bfz_{J(P)}\mapsto f_P(\bfz,\bar\bfz).
\]
Here
\[
 \BC^{*J(P)}({\bfw_{I(P)}})=\{\bfz\in \BC^{*n}\,|\,
\bfz_{I(P)}=\bfw_{I(P)}:\text{fixed}\}\cong\BC^{*J(P)}.
\] 
Thus we are considering $f_P$ as a family of mixed polynomials in
$\bfz_{J(P)}$ with coefficients in
$\BC\{\bfz_{I(P)},\bar\bfz_{I(P)}\}$.
If $d(P,f)=0$, then
 $f_P\in \BC\{\bfz_{I(P)},\bar\bfz_{I(P)}\}$.
\begin{Definition}
{\rm We say that $f$ is {\em super strongly non-degenerate}
if the following condition (SSND) is satisfied.

\noindent
(SSND): for any subset $P\in N^+$, either 

\begin{itemize}
\item[(a)] $d(P,f)=0$ i.e., 
$f_P\in \BC\{\bfz_{I(P)},\bar\bfz_{I(P)}\}$  or

\item[(b)] 
$d(P,f)>0$ and 
 $f_P:\BC^{*J(P)}({\bfw_{I(P)}})\to \BC^*$ has no critical points
for any  $\bfw_{I(P)}\in \BC^{*I(P)}$.
\end{itemize}}\end{Definition}
The following is an immediate consequence of the definition.
\begin{Proposition} 
\begin{enumerate}
\item
If $f(\bfz,\bar\bfz)$ is a  convenient strongly non-degenerate
mixed function, then
 $f(\bfz,\bar\bfz)$ is super strongly non-degenerate.
\item Assume that $f(\bfz,\bar\bfz)$ is super strongly non-degenerate
and  $I\in \cN\cV(f)$. Then $f^I$ is also super strongly non-degenerate.
\end{enumerate}
\end{Proposition}
The assertion (2)  can be proved  in the
exact same way as the proof of Proposition \ref{restriction}.
\qed

The following key lemma
 is a mixed polynomial version of Lemma (2.1.4) of Hamm-L\^e \cite{Hamm-Le1}.
\begin{Lemma}\label{non-isolated-Milnor}
Assume that $f(\bfz,\bar\bfz)$ is a true super strongly non-degenerate
mixed function and consider the mixed hypersurface $V$ and its open subset
$V^{\sharp}$. Take a positive number $r_0$
so that $V^{\sharp}\cap B_{r_0}$ is mixed non-singular and any sphere
$S_r$ intersects transversely with $V^{\sharp}$ for any $0<r\le r_0$.
Then for any fixed $0<r\le r_0$, there exists a sufficiently small positive number
$\de$  such that for any $\eta\in \BC$ with $0<|\eta|\le \de$,
the fiber $f\inv(\eta)\cap B_{r_0}$ is smooth and any sphere
$S_s$ intersects transversely with $f\inv(\eta)$ for any $r\le s\le r_0$
and $\eta$ with $0<|\eta|\le \de$.
\end{Lemma}
\begin{proof}
Assume that the assertion is not true.
 Using the  Curve Selection Lemma (\cite{Milnor,Hamm1}), we can find a
real analytic curve
$\bfz(t),\,0\le t\le 1$ such that 
\[
r\le \|\bfz(t)\|\le r_0,\,\, f(\bfz(t),\bar\bfz(t))\not\equiv 0,
\,\,  \bfz(0)\in f\inv(0)
\]
and the fiber $f\inv(\al(t))$ and the sphere of radius 
$\|\bfz(t)\|$ is not transverse at $\bfz(t)$ where
 $\al(t)=f(\bfz(t),\bar\bfz(t))$.
Recall that we have defined two special vectors:
\begin{eqnarray*}
\begin{split}
\bfv_1(\bfz,\bar\bfz)&= \overline{d\log f}(\bfz,\bar\bfz)+\bar d\log
 f(\bfz,\bar\bfz)\\
 \bfv_2(\bfz,\bar\bfz)&=i(\overline{d\log f}(\bfz,\bar\bfz)-\bar d\log
 f(\bfz,\bar\bfz))
 \end{split}\end{eqnarray*}
Recall that  the tangent space of the fiber $T_{\bfz}f\inv(\eta)$
is  spanned by the vectors which  are perpendicular to 
$\bfv_1(\bfz,\bar\bfz)$ and $\bfv_2(\bfz,\bar\bfz)$. Thus  under the assumption
there exist real-valued  analytic functions
$\la(t),\mu(t)$ so that 
\[
 \bfz(t)=\la(t)\,\bfv_1(\bfz,\bar\bfz))+\mu(t)\,
\bfv_2(\bfz,\bar\bfz)),
\]
as in the proof of Lemma \ref{key lemma}. 
Let $I=\{j\,|\,z_j(t)\not \equiv 0\}$.
Then $I\in \cN\cV(f)$. We may assume that $I=\{1,\dots,
 m\}$
and we do the  same argument  in $\BC^I$ as in the proof of Lemma
 \ref{key lemma}. 
We consider the Taylor expansions of $\bfz(t),\,f(\bfz(t),\bar\bfz(t))$ 
and the Laurent expansions of $\la(t)$ and $\mu(t)$:
\begin{eqnarray*}
\begin{split}
&z_j(t)=a_jt^{p_j}+\text{(higher terms)},\quad a_j\in \BC^*,\,
 p_j\ge 0, \,1\le j\le m,\\
&f(\bfz(t),\bar\bfz(t))=\al\,t^{\ell}+\text{(higher terms)},\,\al\in
 \BC^*,\,
\ell\in \BN\\
&\la(t)=\la_0 t^{\nu_1}+\text{(higher terms)},\,\,\la_0\in \BR^*,\,
\nu_1\in \BZ\\
&\mu(t)=\mu_0 t^{\nu_2}+\text{(higher terms)},\,\,\mu_0\in \BR.
\end{split}
\end{eqnarray*}
Here we understand $\nu_1=\infty$ or $\nu_2=\infty$ if $\la(t)\equiv 0$
or
$\mu(t)\equiv 0$ respectively.
We put $\nu_0=\min\{\nu_1,\mu_1\}$
and we write for simplicity as follows.
\begin{eqnarray*}
\begin{split} 
&\la(t)=\what\la_0t^{\nu_0}+\text{(higher terms)},\,\,\la_0\in \BR^*,\,
\nu_1\in \BZ\\
&\mu(t)=\what\mu_0 t^{\nu_0}+\text{(higher terms)},\,\,m_0\in \BR\\
&\text{where}\,\,
\what\la_0=\begin{cases}\la_0\,\,\quad &\text{if}\,\,\nu_1=\nu_0\\
	   0\,\,\quad&\text{if}\,\,\nu_1>\nu_0\end{cases},\\
&\qquad \quad \what\mu_0=\begin{cases}\mu_0\,\,\quad &
\text{if}\,\,\nu_2=\nu_0\\
	   0\,\,\quad &\text{if}\,\,\nu_2>\nu_0\end{cases}\\
\end{split}
\end{eqnarray*}
Note that $\nu_0<\infty$ and  $(\what\la_0,\what\mu_0)\ne (0,0)$ anyway.
Consider the equality:
\begin{eqnarray*}\label{eqeqeq1}\begin{split}
 z_j(t)=&\la(t)\,(\frac{\overline{\partial f}}{\partial z_j}/\bar f+
\frac{\partial f}{\partial \bar z_j}/f)(\bfz(t),\bar\bfz(t))\\
&\quad +\mu(t)\,
i\,(\frac{\overline{\partial f}}{\partial z_j}/\bar f-
\frac{\partial f}{\partial \bar z_j}/f)(\bfz(t),\bar\bfz(t)) ,\,j=1,\dots, m.
\end{split}\end{eqnarray*}
Put $P=(p_1,\dots, p_m)$, 
$I(P)=\{j\,|\,p_j=0\}$,
$J(P)=\{j\,|\,p_j\ne 0\}$, $\bfa=(a_1,\dots, a_m)$ and $d=d(P,f^I)$.
Note that $\bfz(0)\in \BC^{*I(P)}$.
Assume that $ I(P)\in \cN\cV(f)$. Then $\bfz(0)\in V^{\sharp}$
and it is a smooth point. Thus by the assumption, the sphere
$S_{\|\bfz(0)\|}$ intersects transversely with $V^{\sharp}$. Thus the
 same is true for $S_{\|\bfz(t)\|}$ and $f\inv(\al(t))$
for any sufficiently small $t$ which 
is a contradiction to the assumption.
Thus we may assume that  $\bfz(0)\in V\setminus V^{\sharp}$ and therefore
$I(P)\notin \cN\cV(f^I)$ ($\iff$
 $I(P)\cup I\notin \cN\cV(f)$). 
Then we observe that
\begin{eqnarray*}
\begin{split}
&\frac{\overline{\partial f}}{\partial
 z_j}(\bfz(t),\bar\bfz(t))/\bar f(\bfz(t),\bar\bfz(t))=\left(\frac{\overline{\partial
 f_P^I}}{\partial z_j}(\bfa,\bar\bfa)/\bar\al\right) t^{d-p_j-\ell}+\text{(higher
 terms)},\\
&\frac{\partial f}{\partial \bar z_j}
(\bfz(t),\bar\bfz(t)))/f(\bfz(t),\bar\bfz(t))=
\left(\frac{\partial
 f_P^I}{\partial \bar z_j}(\bfa,\bar\bfa)/\al \right)t^{d-p_j-\ell}+\text{(higher
 terms)}
\end{split}
\end{eqnarray*}
By Assertion \ref{transversality}, we have
\begin{eqnarray*}\begin{split}
&\what\la_0\,\left(\frac{\overline{\partial f_P^I}}{\partial
 z_j}(\bfa,\bar\bfa)/\bar\al+
\frac{\partial f_P^I}{\partial \bar z_j}(\bfa,\bar\bfa)/\al\right)+
\what\mu_0\, i\,\left(\frac{\overline{\partial f_P^I}}{\partial
 z_j}(\bfa,\bar\bfa)/\bar\al-
\frac{\partial f_P^I}{\partial \bar z_j}(\bfa,\bar\bfa)/\al\right)\\
&\qquad =
0,\quad j\in J(P) .
\end{split}
\end{eqnarray*}
This implies that $\bfa_{J(P)}$ is a critical point of the 
mixed polynomial
$f_P^I:\BC^{*J(P)}(\bfa_{I(P)})\to \BC$ and $f_P^I(\bfa,\bar\bfa)\ne 0$
with $\bfz_{I(P)}=\bfa_{I(P)}$ fixed. This is a contradiction to
the super strong  non-degeneracy of $f^I$.
\end{proof}
\subsection{Milnor fibration for non-isolated singularities}
Now, by Lemma \ref{non-isolated-Milnor} and Lemma
 \ref{ExistenceRadius1},
 we have the following non-isolated version of the Milnor fibration.
Note that $\varphi=f/|f|: \,S_r^{2n-1}\setminus K_r\to S^1$
is a fibration using a 
$|f|$-level preserving vector field near $K_r$ by
the transversality of $f\inv(\eta)$ and $S_r$ for $\eta,\,|\eta|\ll \de$.
\begin{Theorem}\label{NIM}
Assume that $f(\bfz,\bar\bfz)$ is a super strongly non-degenerate mixed
 function. Then there exists a stable radius $r_0>0$ so that 
for any $r$ with $0<r\le r_0$  and  a sufficiently small number $\de$ 
(compared with $r$), we have two equivalent fibrations:
\[
 \begin{split}
&f:\quad \partial E(r,\de)^*\to S_{\de}^1\\
&\varphi=f/|f|: \,S_r^{2n-1}\setminus K_r\to S^1
\end{split}
\]
where $K_r=f\inv(0)\cap S_r^{2n-1}$.
Moreover, if $f$ is a polar weighted polynomial, the global fibration
$ f:f\inv(S_\de^1)\to S_\de^1$
is also equivalent to the above fibration.
\end{Theorem}

\begin{Example}{\rm 
I.
A monomial $z_1^{\mu_1}\bar z_1^{\nu_1}z_2^{\mu_2}\bar z_2^{\nu_2}$
  is called  {\em an inside monomial} if $\mu_1+\nu_1,\,\mu_2+\nu_2>0$.
An inside  monomial $z_1^{\mu_1}\bar z_1^{\nu_1}z_2^{\mu_2}\bar z_2^{\nu_2}$
is called {\em polar admissible} if $\mu_1\ne \nu_1$ 
and
$\mu_2\ne \nu_2$.
 Let $g(\bfz,\bar\bfz)$ be a  strongly non-degenerate
polar weighted mixed function of two variables $\bfz=(z_1,z_2)$
with two simple end vertices $A,\,B$ of $\Ga(g)$.
We assume that 
\[
 A=(m_1,n_1),\,B=(m_2,n_2),\,\,m_1<m_2,\,n_1>n_2
\]
which  come from the mixed monomials $z_1^{\mu_1}\bar
 z_1^{\nu_1}z_2^{\mu_2}\bar z_2^{\nu_2}$
and $z_1^{\mu_1'}\bar
 z_1^{\nu_1'}z_2^{\mu_2'}\bar z_2^{\nu_2'}$.
Here
\[
 m_1=\mu_1+\nu_1,\,n_1=\mu_2+\nu_2,\,  m_2=\mu_1'+\nu_1',\,n_2=\mu_2'+\nu_2'.
\]
 Consider $P={}^t(1,0)$ for example.
Then $g_P(\bfz,\bar\bfz)=c\bfz^{\mu}\bar\bfz^\nu$ with some non-zero
constant $c$.
Assume that $m_1>0$.
 To check if $g_P:\BC^{*}\to \BC^*$
has a critical point or not as a function of $z_1$  variable, we can use $\log g_P$ instead of $g_P$.
Now we have
\begin{eqnarray*}
\overline{d_{z_1}\log g_P}(\bfz,\bar\bfz)=\frac{\mu_1}{\bar z_1},
\quad
\bar d_{z_1}\log g_P(\bfz,\bar\bfz)=\frac{\nu_1}{\bar z_1}.
\end{eqnarray*}
If $z_1\in \BC^*$ is a critical point of $g_P$ for some fixed $z_2\in \BC^*$,
we must have $u\in S^1$ such that $\frac{\mu_1}{\bar
z_1}=u\frac{\nu_1}{\bar z_1}$.
This is only possible if $\nu_1=\mu_1$.
By a similar discussion for $Q={}^t(0,1)$, we have shown the following.
\begin{Lemma}
 Assume that $g(\bfz,\bar\bfz)$ is a non-degenerate polar weighted mixed
 polynomial whose two end monomials are
$z_1^{\mu_1}\bar
 z_1^{\nu_1}z_2^{\mu_2}\bar z_2^{\nu_2}$
and $z_1^{\mu_1'}\bar
 z_1^{\nu_1'}z_2^{\mu_2'}\bar z_2^{\nu_2'}$ with 
$\mu_1+\nu_1<\mu_2+\nu_2$.
Then $g(\bfz,\bar\bfz)$ is super strongly non-degenerate if and only
 if the following conditions are satisfied.
\begin{enumerate}
\item Either $\mu_1=\nu_1=0$ or $z_1^{\mu_1}\bar
 z_1^{\nu_1}z_2^{\mu_2}\bar z_2^{\nu_2}$ is
 polar admissible.

\item Either $\mu_2'=\nu_2'=0$ or  $z_1^{\mu_1'}\bar
 z_1^{\nu_1'}z_2^{\mu_2'}\bar z_2^{\nu_2'}$
is polar admissible.
\end{enumerate}
\end{Lemma}

\noindent
II.
Let $f(\bfz,\bar\bfz)=z_1^{a_1}\bar z_2^{b_1}+z_2^{a_2}\bar
z_3^{b_2}+\cdots+z_n^{a_n}\bar z_1^{b_n}$
be a simplicial polar weighted homogeneous mixed polynomial.
We assume that $a_j>b_{j-1}\ge 1$  for $j=1,\dots, n$ with
$b_{0}=b_n$. 
We assert that {\em $f$ is super strongly non-degenerate}.
\begin{proof}
Consider $f_P^I$ for some $I\in \cN\cV(f)$ and $P\in N^+$ and let
 $I(P),\,J(P)$  be
as in the proof of Lemma \ref{non-isolated-Milnor}.
We assume that $d(P,f)>0$.
Suppose that $z_1^{a_1}\bar z_2^{b_1}$ is in $f_P^I$.
Then $\{1,2\}\cap J(P)\ne \emptyset$.
Assume that $2\in J(P)$ for example.
Then $\frac{\partial f_P}{\partial \bar z_2}\ne 0$.
If $f_P^I$ has a critical point as a mapping
$f_P^I:\BC^{*J(P)}\to \BC^*$,
we need a non-zero $\frac{\partial f_P^I}{\partial z_2}$ by Proposition 7, which implies
$z_2^{a_2}\bar z_3^{b_2}$ must be in $f_P^I$.
As $a_2>b_1$ by the assumption and
$p_1a_1+p_2b_1=p_2a_2+p_3b_2$, this implies that $1\in J(P)$ i.e.,
 $p_1\ne 0$.
This implies again that $\frac{\partial f_P^I}{\partial \bar z_1}\ne 0$
and therefore $z_n^{a_n}\bar z_1^{b_n}$ must be in $f_P^I$.
By the same reasoning, $a_1>b_n$ implies that $p_n>0$ and $n\in J(P)$.
Then we consider  $\frac{\partial f_P^I}{\partial z_n}$ and we see that
$n-1\in J(P)$ and $z_{n-1}^{a_{n-1}}\bar z_n^{b_{n-1}}$ is in $f_P^I$.
Continuing the same discussion, we conclude
$f_P^I=f$ i.e., $I=\{1,\dots,n\}$. However, $f(\bfz,\bar\bfz)$ is polar weighted and it has no
critical point over $\BC^*$. Thus $f$ is super strongly non-degenerate.
\end{proof}
}\end{Example}
\section{Resolution of a polar type  and the  zeta function}
In this section, we will study the relation between a resolution 
of a polar type and the Milnor fibration of the second type.
We expect a similar formula like the formula of A'Campo
(\cite{AC-zeta}) or the formula of Varchenko
\cite{Varchenko}.
We will restrict ourselves to the case of  mixed curves.
\subsection{Polar weighted case}
Let $f(\bfz,\bar \bfz)$ be a mixed polynomial of $n$ variables 
$z_1,\dots, z_n$ and let $(q_1,\dots, q_n;d_r)$ and
$(p_1,\dots,p_n;d_p)$
be the radial and polar weight types. We assume that $d_p>0$.
Then $f: \BC^{*n}-f\inv (0)\to\BC^*$ is a fibration.
Put $F_s^*=f\inv(s)\cap \BC^{*n}$ for $s\in \BC^*$.
 Then the monodromy map
$h: F_s^*\to F_s^*$ is given by the polar action as
\[
 h(z_1,\dots, z_n)=(z_1\omega^{p_1},\dots, z_n\omega^{p_n}),\,
 \omega=\exp(\frac{2\pi i}{d_p})
\]
Put $F^*=F_1^*$ and let  $\chi(F^*)$ be the Euler characteristic of $F^*$.
Then the monodromy has the period $d_p$ and
the set of the fixed points of 
$h^j:F^*\to F^*$ is  empty if $j\not \equiv 0$ modulo $d_p$,
where $h^j=h\circ\dots\circ h$ (j-times). Thus using
the formula of the zeta function for a periodic mapping (\cite{Milnor}),
we
get
\begin{Lemma}\label{zetaPolar1} Under the above assumption, the
 zeta-function of $h: F^*\to F^*$
is given as
\[
 \zeta(t)= (1-t^{d_p})^{-\chi(F^*)/d_p}.
\]
\end{Lemma}
The zeta function of the global fibration
$f:\BC^n\setminus f\inv(0)\to \BC^*$ can be obtained by patching the
data for each torus stratum.

 Let us do this for curves ($n=2$).
Let $f(\bfz)$ be a non-degenerate
 polar weighted homogeneous polynomial
of type $(p_1,p_2;d_p)$. The signs of $p_1,p_2$ are chosen so that 
$d_p>0$. Suppose that the two edge vertices of $\Ga(f)$
are
simple. Assume that 
 the two end  monomials
 are 
\[
 z_1^{\mu_1}\bar z_1^{\nu_1}z_2^{\mu_2}\bar z_2^{\nu_2},
\quad
z_1^{\mu_1'}\bar z_1^{\nu_1'}z_2^{\mu_2'}\bar z_2^{\nu_2'}
\]
with $\mu_1+\nu_1<\mu_1'+\nu_1'$ and $\mu_2+\nu_2>\mu_2'+\nu_2'$.

Assume that $\mu_1=\nu_1=0$ and $\mu_2'=\nu_2'=0$ i.e.,
$f(\bfz,\bar\bfz)$ is convenient. In this case the  two monomials reduces to
$z_2^{\mu_2-\nu_2}|z_2|^{2\nu_2},\, z_1^{\mu_1'-\nu_1'}|z_1|^{2\nu_1'}$.
Let $F=f\inv(1)\subset \BC^2$, $F_{z_1}=F\cap\{z_2=0\}$ and 
$F_{z_2}=F\cap \{z_1=0\}$.
Note that 
\[
 F_{z_1}=\{(z_1,0)\,|\, z_1^{\mu_1'-\nu_1'}=1\},\quad
 F_{z_2}=\{(0,z_2)\,|\, z_2^{\mu_2-\nu_2}=1\}.
\]
The monodromy map is defined by
\[
 h:F\to F,\quad (z_1,z_2)\mapsto (z_1\omega^{p_1},z_2\omega^{p_2}),
\,\,\omega=\exp(\frac{2\pi i}{d_p})
\]
Note that $p_1(\mu_1'-\nu_1')=p_2(\mu_2-\nu_2)=d_p$.
Therefore  the  fixed points set ${\rm Fix}(h^j)$ of $h^j$ is non-empty
only  for $j=|\mu_1'-\nu_1'|,\,|\mu_2-\nu_2|$, or $ d_p$
and their multiples.
Thus using the calculation through $\exp\zeta(t)$ as in \cite{Milnor},
we get
\begin{Lemma}\label{zetaPolar} Let $f(\bfz,\bar \bfz)$ be a polar
 weighted convenient polynomial as above. Let 
$z_1^{\mu_1'}\bar z_1^{\nu_1'},\, z_2^{\mu_2}\bar z_2^{\nu_2}$ be the 
end monomials and let $d_p$ be the polar degree.
Then the  Euler-Poincar\'e characteristic $\chi(F)$ and the zeta function
of the monodromy $h: F\to F$ are
 given as
\[\begin{split}
&\chi(F)=\chi(F^*)+|\mu_1'-\nu_1'|+|\mu_2-\nu_2|,\,\mu=1-\chi(F)\\
& \zeta(t)=\frac{(1-t^{d_p})^{-\chi(F^*)/d_p}}{(1-t^{|\mu_1'-\nu_1'|})\,
(1-t^{|\mu_2-\nu_2|})}
\end{split}
\]
\end{Lemma}
\begin{Remark}\label{zetaPolar2}{\rm  By a similar consideration, if
 $f(\bfz,\bar\bfz)$
is a polar weighted polynomial which is not convenient,
the assertion is true under the following modification.
Put $\eps_1=1$ or $0$ according to $\mu_2'+\nu_2'=0$ or $\mu_2'+\nu_2'>0$.
Similarly $\eps_2=1$ or  $0$ according to $\mu_1+\nu_1=0$ or
 $\mu_1+\nu_1>0$.
Then 
\[\begin{split}
&\chi(F)=\chi(F^*)+\eps_1|\mu_1'-\nu_1'|+\eps_2|\mu_2-\nu_2|,\,\mu=1-\chi(F)\\
& \zeta(t)=\frac{(1-t^{d_p})^{-\chi(F^*)/d_p}}{(1-t^{|\mu_1'-\nu_1'|})^{\eps_1}
\,(1-t^{|\mu_2-\nu_2|})^{\eps_2}}
\end{split}
\]
}
\end{Remark}
\subsubsection{Simplicial polar weighted polynomial}
Let $f(\bfz,\bar\bfz)=\sum_{j=1}^m c_j\,\bfz^{\mu_j}\bar\bfz^{\nu_j}$.
The associated Laurent polynomial $g(\bfz)$ is defined by 
\[
 g(\bfz)=\sum_{j=1}^m c_j\,\bfz^{\mu_j-\nu_j}.
\]
Recall that $f(\bfz,\bar\bfz)$ is called {\em  simplicial
polar weighted homogeneous} if
$m=n$ and the two matrices have a non-zero determinant {\cite{OkaPolar}}:
\[
 M=\left(
\begin{matrix}
\mu_{11}+\nu_{11}&\dots& \mu_{1n}+\nu_{1n}\\
\vdots&\cdots&\vdots\\
\mu_{n1}+\nu_{n1}&\dots& \mu_{nn}+\nu_{nn}
\end{matrix}\right),\quad
 N=\left(
\begin{matrix}
\mu_{11}-\nu_{11}&\dots& \mu_{1n}-\nu_{1n}\\
\vdots&\cdots&\vdots\\
\mu_{n1}-\nu_{n1}&\dots& \mu_{nn}-\nu_{nn}
\end{matrix}\right)
\]
where $\mu_j=(\mu_{j1},\dots,\mu_{jn})$ and
$\nu_j=(\nu_{j1},\dots,\nu_{jn}),\,j=1,\dots,n$
respectively.
If $f$ is a simplicial polar weighted homogeneous polynomial, we have
shown  that
the  two fibrations defined by
$f(\bfz,\bar\bfz)$ and $g(\bfw)$:
\[
f:\BC^{*n}\setminus f\inv(0)\to \BC^*,\quad
g:\BC^{*n}\setminus g\inv(0)\to \BC^*
\]
are equivalent (\cite{OkaPolar}).
Thus the topology of the Milnor fibration is determined by the 
mixed face $\what \Delta$ where 
$\De$ is the unique face of $\Ga(f)$. In particular, the zeta function of 
$h:F^*\to F^*$ is given as
$\zeta(t)=(1-t^{d_p})^{(-1)^n d/d_p}$ where $d=|\det(N)|$
(\cite{OkaPolar}).
On the other hand, if $f$ is not simplicial,
the topology is not even a combinatorial invariant of $\what \De$ (\S
\ref{non-simple}).
Therefore there does not exist any direct connection with
the topology of the associated Laurent polynomial $g(\bfz)$.
 However here is a useful lemma.
\begin{Lemma}
Suppose that $f_t(\bfz,\bar\bfz),\,0\le t\le 1$ is a family of convenient,
non-degenerate polar weighted 
homogeneous polynomials with the same radial and the  polar weights, and
 assume  that
$\Ga(f_t)$ is constant.
 Then the Milnor fibration
$f_t:\BC^n\setminus f_t\inv(0) \to \BC^*$
and its restriction
$ \BC^{*n}\setminus f_t\inv(0))\to \BC^{*}$  are homotopically
 equivalent
to $f_0:\BC^n\setminus f_0\inv(0) \to \BC^*$ and 
$f_0:\BC^{*n}\setminus f_0\inv(0)\to \BC^{*}$
respectively.
\end{Lemma}
\begin{proof}
Consider the unit  sphere $S^{2n-1}=S_1^{2n-1}$.
For each $I\subset \{1,\dots, n\}$,
$|I|\ne\emptyset$,
the intersection $(f_t^I)\inv(0)\cap S^I$ is transverse and smooth
for any $t$
where $S^I=\{\bfz^I\in \BC^I\,|\,\|\bfz\|=1\}$.
Thus by the compactness argument, there exists a common positive
number $\de$ such that 
the intersection $(f_t^I)\inv(\eta)\cap S^I,$ is transverse and smooth
for any $t,\,0\le t\le 1$ and $\eta$ with $|\eta|\le \de$.
This implies by the Ehresmann fibration theorem (\cite{Wolf1})
that
the fibrations
\[
 f_t^I:   E_t^I(1,\de)^*\to D(\de)^*
\]
are equivalent for each $t$, where
\[
 E_t^I(1,\de)=(f_t^I)\inv(D(\de)^*)\cap B^I,\,
B^I=\{\bfz^I\in \BC^I\,|\,\|\bfz^I\|\le 1\}.
\]
Thus we can construct characteristic diffeomorphisms
 \[
  h_{\theta}: f_t\inv (\de)\cap B^{2n}\to f_t\inv(\de\exp(\theta\,i))\cap B^{2n}
\]
 for 
$0\le \theta\le 2\pi$
which preserve the stratification
$f\inv(\de)\cap B^{I}$,
$I\subset \{1,\dots, n\}$.
Now the assertion follows from Theorem \ref{Polar Milnor}.
\end{proof}
\begin{Example}{\rm 
Consider the  family of  polar weighted mixed polynomials in  two variables:
\[
 f_t(\bfz,\bar\bfz)=-2z_1^2\,\bar z_1+z_2^2\,\bar z_2+t\,z_1^2\,\bar z_2,\,t\in
 \BC
\]
and let $C_t=f_t\inv(0)$.
The radial and polar weight types are $(1,1;3)$ and $(1,1;1)$
respectively.
Thus the critical points of $f_t:\BC^2\to \BC$ are the solutions of
\begin{eqnarray}|\al|=1,\quad
 \begin{cases}
&-4 z_1\,\bar z_1+2\bar t\,\bar z_1\,z_2=-2\al\, z_1^2 \\
&2z_2\,\bar z_2=\al\,(z_2^2+t\,z_1^2)\\
&-2z_1^2\,\bar z_1+z_2^2\,\bar z_2+t\,z_1^2\,\bar z_2=0.\\
\end{cases} \label{eq123}
\end{eqnarray}
 \begin{figure}[htb]   
\setlength{\unitlength}{1bp}   
\begin{picture}(600,150)(-20,0)  
\put(70,130){\special{epsfile=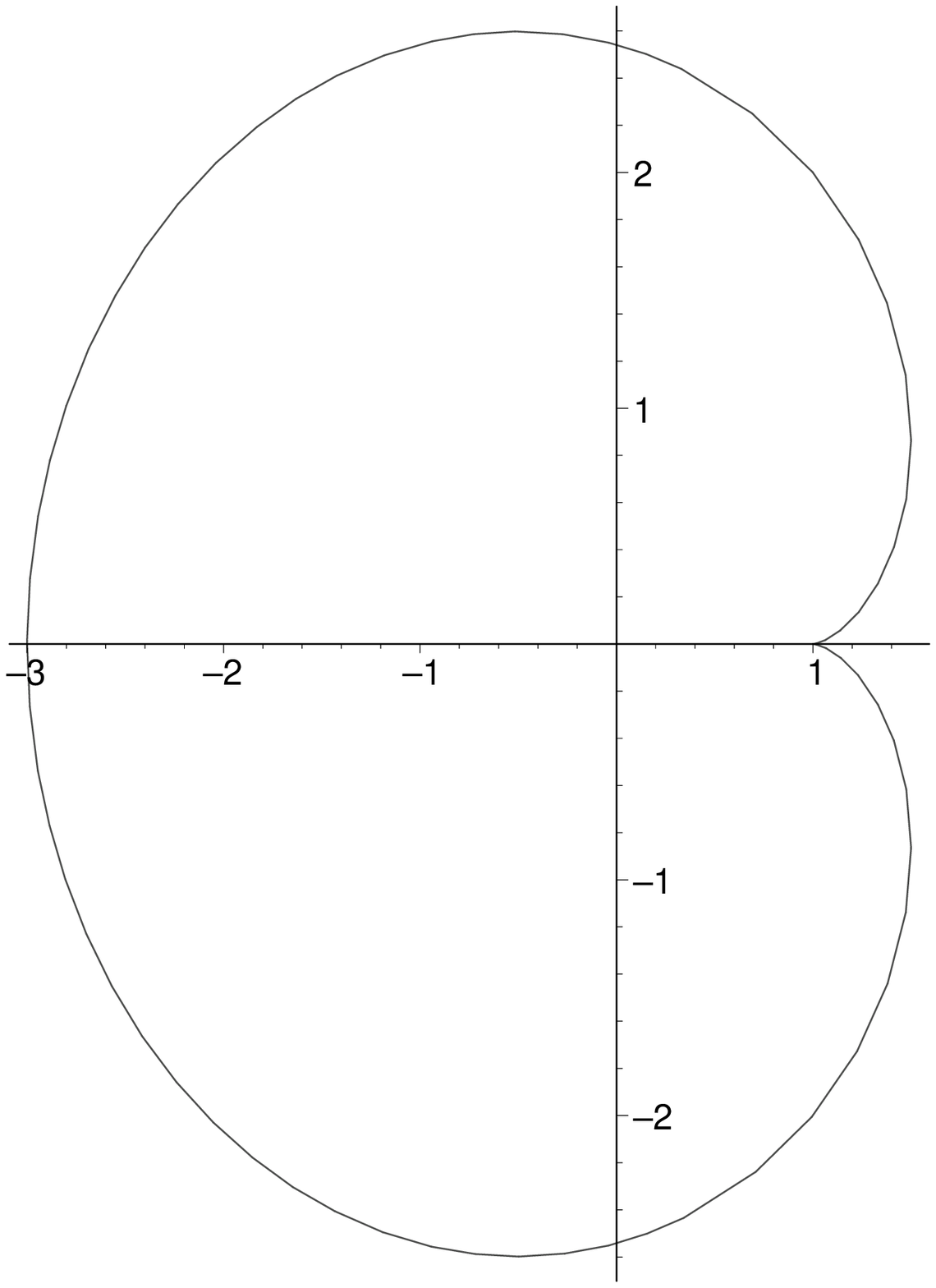 vscale=0.3 hscale=0.4}}   
\put(160,45){$O$}
\put(140,80){$U_1$}
\put(240,80){$U_2$}
\put(200,115){$\Xi$}
\end{picture}
\vspace{.5cm}
\caption{Degeneration locus $\Xi$}\label{Fig1}
\end{figure}   

First it is easy to see that
for a solution $(\bfz,\al)$ of $(16)$, either
$\bfz=(0,0)$ or  $\bfz\in \BC^{*2}$.
Secondly the equations are homogeneous in $z_1,z_2$.
Thus we may assume that $|z_2|=1$.
By (\ref{eq123}), we get 
$2z_2\,\bar z_2^2=2\,\al\, z_1^2\,\bar z_1$. Thus $|z_1|=1$.
Put $z_1/z_2=\exp(\theta\,i)$.
Then we can solve as 
\[
 t=-\exp(-2\theta i)+2\exp(-\theta i),\,
z_1=z_2\exp(\theta\,i),\,\,\al=\frac 2{z_2^2+tz_1^2}.
\]
Put $\Xi:=\{-\exp(-2\theta i)+2\exp(-\theta i)|0\le \theta\le 2\pi\}$.
$\Xi$ is the locus where $f_t$ is degenerate.
The complement $\BC\setminus \Xi$ has two components, $U_1,\,U_2$
where $U_1$ is the  bounded region with boundary $\Xi$. See Figure 3.
By  further calculation, we can see that 
$\lkn(C_t)=1$, $\chi(F)=1$, $\chi(F^*)=-1$ for  $t\in U_1$  and $\lkn(C_t)=3$,
$\chi(F)=-1$, $\chi(F^*)=-3$ for $t\in U_2$.
(See Appendix for the calculation.)
The associated Laurent polynomial is
$g_t(\bfz)= -2z_1+z_2+t\,z_1^2 z_2^{-1}$ which is non-degenerate for
$t\ne 1,0$.
Thus we see that $\chi(G_t^*)=-2$ for $t\ne 0,1$
where $G_t^*=g_t\inv(1)\cap \BC^{*2}$ (see \cite{Okabook}). 
This example shows that Theorem 10 of \cite{OkaPolar} does not hold
for non-simplicial polar weighted polynomials.
}\end{Example}
\subsection{Zeta function of  non-degenerate mixed curves}
Let $f(\bfz,\bar\bfz)$ be a convenient non-degenerate mixed polynomial
 and   let $\De_1,\dots, \De_s$ be the
faces of $\Ga(f)$. Let $Q_j={}^t(q_{j1},q_{j2})$ be the weight vector of
$\De_j$
for $j=1,\dots, s$.
Assume that each face function $f_{\De_j}$ is also polar weighted and
the inside monomials corresponding to the vertices
$M_j=\De_j\cap \De_{j+1},\,j=1,\dots, s-1$ are polar admissible.
Let $(a_1+2b_1,0),\,(0,a_2+2b_2)$ be the vertices of $\Ga(f)$ on the coordinate
axes which come from the monomials
$z_1^{a_1}|z_1|^{{2b_1}}$ and $z_2^{a_2}|z_2|^{2b_2}$ respectively.
We call $a_1,\,a_2$ {\em the polar sections of $\Ga(f)$} on the respective
coordinate axes
 $z_2=0$ and $z_1=0$.
Let $f_{\De_i}(\bfz,\bar\bfz)$ be the face function of 
$\De_i$
and assume that $(p_{i1},p_{i2};m_i)$ is the polar weight type of
$f_{\De_i}(\bfz,\bar\bfz)$.
Let 
$F_i^*=\{\bfz\in \BC^{*2}\,|\, f_{\De_i}(\bfz,\bar\bfz)=1\}$.
Then we have the following.
\begin{Theorem}\label{MixAC}
Assume that $f(\bfz,\bar\bfz)$ is a non-degenerate convenient mixed
 polynomial
such that its face functions $f_{\De_j}(\bfz,\bar\bfz),\,j=1,\dots, s$
 are
polar weighted polynomials. Then 
the Euler-Poincar\'e characteristic of
the Milnor fiber $F$ of $f$  and the zeta function
of the monodromy $h:F\to F$ are  given as follows.
\[
 \begin{split}
&\chi(F)=\sum_{i=1}^s \chi(F_i^*)\,+\, |a_1|+|a_2|\\
&\zeta(t)=\frac{\prod_{i=1}^s
  (1-t^{m_i})^{-\chi(F_i^*)/m_i}}{(1-t^{|a_1|})\,(1-t^{|a_2|})}
\end{split}
\]
where $a_1,a_2$ are the respective polar sections and $m_j$
is the polar degree of the face function $f_{\De_i}(\bfz,\bar\bfz)$,
$j=1,\dots,s$ as above ($m_j>0$).
\end{Theorem}
\begin{Remark} The assertion is true for 
non-degenerate mixed polynomials with polar weighted face functions
in two variables
which may not be convenient.
For example, if $\Ga(f)\cap \{z_2=0\}=\emptyset$,
we eliminate $|a_1|$ and $(1-t^{|a_1|})$ from the formula.\end{Remark}
The proof occupies the rest of the section.
For the proof, we use the
following multiplicative property of the zeta function.
Consider an excision pair $\{A,B\}$ in the Milnor fiber $F$.
We say $\{A,B\}$ is {\em stable} for the monodromy map $h$ if
$h(A)\subset A$ and $h(B)\subset B$.
\begin{Proposition} {\rm (Proposition 2.8, \cite{Okabook})}
Suppose that $F$ decomposes into $h$ stable  excision couple $A,\,B$
so that $F=A\cup B$.
Put $C=A\cap B$.
Then let $\zeta(t),\,\zeta_A(t)$,
$\zeta_B(t)$ and $\zeta_C(t)$  be the zeta functions of $h:F\to F$ and 
$h_A:=h|_A: A\to A$, $h_B:=h|_B:B\to B$
and $h_{C}:=h|_{C}:C\to C$ respectively. Then
\[
 \zeta(t)=\frac{\zeta_A(t)\,\zeta_B(t)}{\zeta_C(t)}.
\]
\end{Proposition}
\subsubsection{Resolution of a polar type  and the Milnor fibration}
Let us consider an admissible toric modification
$\what \pi: X\to \BC^2$ with respect to the regular fan $\Si^*$ with vertices
$\{P_0,P_1,\dots, P_{\ell+1}\}$ and we assume that 
$Q_j=P_{\nu_j},j=1,\dots, s$ and $P_0=E_1={}^t(1,0)$
and $P_{\ell+1}=E_2={}^t(0,1)$.
Then we take the polar modification 
$\omega_p:\cP X\to X$
along $\widehat E(P_1),\dots, \widehat E(P_{\ell})$.
Put $\Phi_p:\cP X\to \BC^2$ be the composite with $\what \pi: X\to \BC^2$.
Consider the second Milnor fibration
\[
 f\circ \Phi_p:\Phi_p\inv( E(r,\delta)^*)\to D(\de)^*
\]
on the resolution space $\cP X$.
Take $P_j$ for $1\le j\le \ell$. There are two toric coordinate charts
of $X$
which contain the  vertex $P_j$:
\[\begin{split}
& \si_{j-1}=\Cone(P_{j-1},P_j)\quad
\text{gives the coordinate chart}\,\,(U_{j-1},(u_{j-1},v_{j-1}))\\
& \si_{j}=\Cone(P_{j},P_{j+1})\quad
\text{gives the coordinate chart}\,\,(U_j,(u_{j},v_{j})).\\
\end{split}
\]
Put $M=(P_j,P_{j+1})\inv (P_{j-1},P_j)$. It takes the form:
\[
   M=\left(\begin{matrix}
\ga_j&1\\
-1&0
\end{matrix}
\right).
 \]
Then the two coordinate systems are connected by the relation
\begin{eqnarray}\label{adjacent}
 u_j=u_{j-1}^{\ga_j}v_{j-1},\quad v_j=u_{j-1}^{-1}.
\end{eqnarray}
Put $P_j={}^t(c_j,d_j),\,j=1,\dots, \ell$.
The inverse image $\wtl U_j:=\omega_p\inv(U_j)$ has
the polar coordinates
$(r_j,\theta_j,s_j,\eta_{j})$
which corresponds to $(u_j,v_j)$ with $u_j=r_j\exp(i\,\theta_j )$
and $v_j=s_j\exp(i\,\eta_j )$.
The relation (\ref{adjacent}) says that 
\begin{eqnarray}
 s_j=r_{j-1}\inv ,\qquad \eta_j=-\theta_{j-1}.
\end{eqnarray}
We do not take  a normal
polar modification along the two non-compact divisors $u_0=0$ and $v_\ell=0$.
Thus the coordinates of $\wtl U_0$ and $\wtl U_\ell$ are
$(u_0,s_0,\eta_0)$ and $(r_\ell,\theta_\ell,v_\ell)$ respectively.
Recall that the exceptional divisor
$\wtl E(P_j)$ is defined by $r_j=0$ in $\wtl U_j$ and 
by $s_{j-1}=0$ in $\wtl U_{j-1}$ for $1\le j\le \ell$.
Note that $u_0=0$ in $U_0$ corresponds bijectively to 
the axis $z_1=0$ in the base space $\BC^2$
and 
\[
(P_0,P_1)=\left(\begin{matrix}1&c_1\\
0&1\end{matrix}\right),\,d_1=1,\, z_1=u_0v_0^{c_1},\,z_2=v_0.
    \]
Similarly on $\wtl U_\ell$, $v_\ell=0$ corresponds to $z_2=0$
and \[
     z_1=u_\ell,\,z_2=u_\ell^{d_\ell}v_\ell,\,\,c_\ell=1.
    \]
   \begin{figure}[htb]   
\setlength{\unitlength}{1bp}   
\begin{picture}(600,150)(-20,0)  
\put(70,130){\special{epsfile=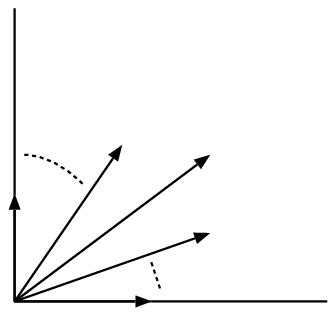 vscale=1 hscale=1}}   
\put(60,45){$O$}
\put(60,80){$E_2$}
\put(100,30){$E_1$}
\put(130,60){$P_{j-1}$}
\put(135,90){$P_j$}
\put(100,100){$P_{j+1}$}
\put(120,72){$\si_{j-1}$}
\put(105,82){$\si_j$}
\end{picture}
\vspace{-.8cm}
\caption{Regular fan $\Si^*$}\label{Fig2}
\end{figure}   

\subsection{Decomposition of the fiber}
Recall that
\[\begin{split}
& E(r,\de)^*=\{(z_1,z_2)\,|\, 0<|f(z_1,z_2,\bar z_1,\bar z_2)|\le \de,\,
\|(z_1,z_2)\|\le r\}\\
&\phi(\bfz):=\sqrt{|z_1|^2+|z_2|^2},\quad \wtl B_r=\phi\inv(B_r)\\
&F_{\de}=\{(z_1,z_2)\,|\, f(z_1,z_2,\bar z_1,\bar z_2)= \de,\,
(z_1,z_2)\in  B_r\}:\,\,\text{Milnor fiber}.
\end{split}
\]
We denote the pull-back  of a function $h$ on $\BC^2$ to
$\cP X$ by
$\tilde h$ for simplicity.
On $\cP X$, we consider the subsets
\[\begin{split}
& W_j(r,\rho)=\{\wtl \bfx=(r_j,\theta_j,s_j,\eta_j)\in\wtl U_j\,|\,
\,1/\rho\ge s_j\ge \rho\}\\
&T_{j-1}(\rho)=\{(r_{j-1},\theta_{j-1},s_{j-1},\eta_{j-1})\in
   \wtl U_{j-1}\,|\,
r_{j-1}\le \rho,\,s_{j-1}\le \rho\}\\
&WT_j(\rho)=\{(r_j,\theta_j,s_j,\eta_j)\in \wtl U_j\,|\,
s_j= \rho,\,r_j\le \rho\}\\
&TW_j(\rho)=\{(r_{j-1},\theta_{j-1},s_{j-1},\eta_{j-1})\in \wtl U_{j-1}\,|\,
r_{j-1} =\rho,\,s_{j-1}\le \rho\}\\
\end{split}
\]
and 
\[
 \begin{split}
& T_0(\rho):=\{(u_0,s_0,\eta_0)\in \wtl U_0\,|\,
  |u_0|\le \rho,\,s_0\le \rho\}\\
& W_0(r,\rho):=\{(u_0,s_0,\eta_0)\in \wtl U_0\,|\,\tl
  \phi(u_0,s_0,\eta_0)\le r,\,|u_0|\ge \rho,\,s_0\ge \rho\}\\
& T_\ell(\rho):=
\{(r_{\ell},\theta_{\ell},v_\ell)\in \wtl U_\ell\,|\,r_\ell \le \rho,\,
|v_\ell|\le \rho\}\\
& W_{\ell}(r,\rho):=
\{(r_{\ell},\theta_{\ell},v_\ell)\in \wtl U_\ell,|\,r_\ell\ge \rho,\,
|v_{\ell}|\ge \rho,\,\wtl\phi(r_{\ell},\theta_{\ell},v_\ell)\le r\}\\
\end{split}
\]

 \begin{figure}[htb]   
\setlength{\unitlength}{1bp}   
\begin{picture}(600,150)(-20,0)  
\put(70,130){\special{epsfile=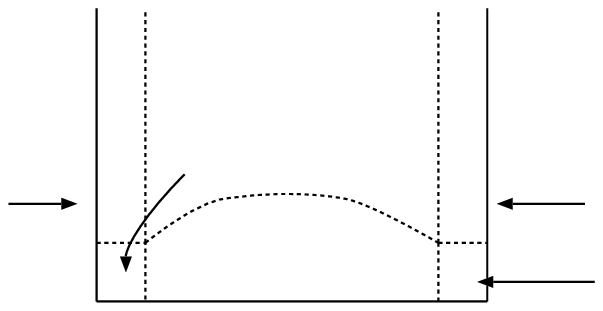 vscale=1 hscale=1}}   
\put(35,80){$\{r_{j-1}=0\}$}
\put(90,35){$\{s_{j-1}=0\}\cup\{r_j=0\}$}
\put(130,50){$W_j(r,\rho)$}
\put(120,85){$T_{j-1}(\rho)$}
\put (230,80){$\{s_j=0\}$}
\put (245,50) {$T_j(\rho)$}
\end{picture}
\vspace{-.8cm}
\caption{Decomposition of $\cP X$}\label{Fig2}
\end{figure}   
Note that 
\[
 \begin{split}
&\wtl \phi(u_0,s_0,\eta_0)=s_0\sqrt{1+|u_0|^2s_0^{2c_1-2}}= s_0+o(s_0)\\
&\wtl\phi(r_{\ell},\theta_{\ell},v_\ell)=r_\ell\sqrt{1+{|v_\ell|}
r_\ell^{2d_\ell-2}}= r_\ell+o(r_\ell)
\end{split}\]
Here $o(s_0)$ implies $o(s_0)/s_0\to 0$ when $s_0\to 0$.
Put 
\[
 A(r,\rho)=\bigcup_{j=0}^{\ell+1}W_j(r,\rho)\cup
\bigcup_{j=0}^{\ell}T_{j}(\rho).
\]
Put $\wtl{ E}(r,\de)^*=\Phi_p\inv( E(r,\de)^*)$ with $\de\ll r$
and  $ A(r,\rho,\de)^*= A(r,\rho)\cap \wtl f\inv(D_\de^*)$
with $\de\ll r,\rho$.
It is easy to see that $A(r,\rho,\de)^*=\wtl{ E}(r,\de)^*$
as long as  $\rho\ll r$ and  $\de\ll \rho,\,r$.
We see that the choice of $\rho$ does not give any effect on $A(r,\rho,\de)^*$,
as long as $\de\ll \rho\ll r$.
 Thus we can use $A(r,\rho,\de)^*$
as the total space of the Milnor fibration:
$\wtl f: A(r,\rho,\de)^*\to D_\de^*$.
We decompose $A(r,\rho,\de)^*$ into monodromy invariant subspaces
as follows.
\[
\begin{split}
& A(r,\rho,\de)^*\cap W_j(r,\rho),\,\,
A(r,\rho,\de)^*\cap
 T_{j}(\rho)\\
&A(r,\rho,\de)^*\cap
 TW_j(\rho),\,\,
A(r,\rho,\de)^*\cap WT_j(\rho),\quad j=0,\dots,\ell.
\end{split}
\]

\subsubsection{Transversality}
Assume that
 $\De(P_j)=\De_t\cap \De_{t+1}=\{M_t\}$
 and that $M_t$ comes from the monomial
$z_1^{\al_{t1}}\,|z_1|^{2\be_{t1}}\,z_2^{\al_{t2}}\,|z_2|^{2\be_{t2}}$.
 By the definition we can
write
\[\begin{split}
& \wtl f(r_j,\theta_j,s_j,\eta_j)\equiv r_j^{d(P_j)}\,s_j^{d(P_{j+1})}
\exp((\al_{t1}\,c_j+\al_{t2}\,d_j)\theta_j \,i)\\
&\qquad \times \exp((\al_{t1}\,c_{j+1}+\al_{t2}\,d_{j+1})\eta_j\, i)
+\, O (r_j^{d(P_j)+1}).
\end{split}
\]
Thus it is easy to see that $\wtl f\inv(\xi),\,|\xi|=\de$
 intersects transversely
with
$WT_j(\rho)$ if $\de$  is  sufficiently small and
$ \de\ll r,\rho$. Similarly  $\wtl f\inv(\xi)$ intersects transversely
with
$TW_j(\rho)$ under the same assumptions.

\noindent
Fix such $r,\de,\rho$.
Under the above decomposition of
$A(r,\rho,\de)^*$,
the Milnor fiber $\wtl F_\de:=\wtl f\inv(\de)\cap \wtl B$ decomposes
into
the following strata:
\[
 \begin{split}
& \wtl F_\de\cap W_j(r,\rho),\,\,\wtl F_\de\cap T_j(\rho),\,
\,\wtl F_\de\cap WT_j(\rho),\,\,
\wtl F_\de\cap TW_{j}(\rho),\,j=0,\dots, \ell.
\end{split}
\]
 By the above transversality,
we see that (after choosing a suitable vector field to define the
characteristic diffeomorphisms) $\wtl F_\de\cap W_j(r,\rho)$,
$\wtl F_\de\cap T_{j}(\rho)$,
$\wtl F_\de\cap TW_j(\rho)$ and $\wtl F_\de\cap WT_j(\rho)$
are invariant by the monodromy $h:\wtl F_\de\to \wtl F_\de$.
Now the proof of Theorem \ref{MixAC} follows from the following
observations.
\begin{enumerate}
\item  The zeta functions of $h$ restricted  on $\wtl F_\de\cap
      T_{j}(\rho)$
are trivial for $1\le j\le \ell-1$.
\item The zeta functions of $h$ restricted on
$\wtl F_\de\cap W_j(r,\rho)$ with $j\ne \nu_1,\dots, \nu_s$ are
      trivial.
\item  The zeta functions of $h$ restricted on
$\wtl F_\de\cap WT_j(\rho)$ and
$\wtl F_\de\cap TW_j(\rho)$   are trivial.
\item The zeta functions of $h$ on $\wtl F_\de\cap
      T_0(\rho)$ and $\wtl F_\de\cap
      T_{\ell}(\rho)$  are respectively given by
\[
 \frac 1{(1-t^{|a_2|})},\quad \frac 1{(1-t^{|a_1|})}.
\]
\item (Face contribution) The zeta function of $h: \wtl F_\de\cap
      W_{\nu_j}(\rho)$ is 
\nl
$(1-t^{m_j})^{-\chi(F_j^*)/m_j}$
where $F_j^*=f_{\De_j}\inv(1)\cap \BC^{*2}$ and $m_j$ is  the polar degree
       of $f_{\De_j}$.

\end{enumerate}
\subsection{Outline of the proof of the assertions (1) to (5)}
~~~

(1) Consider $\wtl F_\de\cap T_j(\rho)$.
Assume that
 $\De(P_j)=\De_t\cap \De_{t+1}=\{M_j\}$
 and that $M_j$ comes from the monomial
$z_1^{\al_{t1}}|z_1|^{2\be_{t1}}z_2^{\al_{t2}}|z_2|^{2\be_{t2}}$
as above.
Then 
\[
 \wtl F_\de\cap T_j(\rho)=
\{(r_j,\theta_j,s_j,\eta_j)\,|\, r_j,s_j\le \rho,\,
 \wtl f(r_j,\theta_j,s_j,\eta_j)=\de\}.
\]
$\wtl f(r_j,\theta_j,s_j,\eta_j)$ takes the form
\[\begin{split}
& \wtl f(r_j,\theta_j,s_j,\eta_j)\equiv
 c_{M_t}\, r_j^{d(P_j)}s_j^{d(P_{j+1})}
\exp((\al_{t1}\,c_j+\al_{t2}\,d_j)\theta_j\, i)\\
&\qquad \times \exp((\al_{t1}\,c_{j+1}+\al_{t2}\,d_{j+1})\eta_j\,  i) \,+\,
O(r_j^{d(P_j)+1}s_j^{d(P_{j+1})+1})
\end{split}
\]
($c_{M_t}$ is a non-zero constant) and the homotopy type of
 this part of the  Milnor fiber is given by
\[
 \{(\theta_j,\eta_j)\in S^1\times S^1\,|\,c_{M_t}
\exp(((\al_{t1}c_j+\al_{t2}d_j)\theta_j +(\al_{t1}
c_{j+1}\al_{t2}d_{j+1})\eta_j) i)=1\}
\]
which is a finite union of copies of  $S^1$ by the following.
\begin{Observation}
Let $a,b$ be  integers with $(a,b)\ne (0,0)$ and
let
\[
 F=\{(\exp(\theta\, i),\exp(\eta \,i))\in S^1\times S^1\,|\,
\exp((a\theta+b\eta)\,i)=1\}.
\] Then $F$ is a disjoint union of  copies of $S^1$
 and the number of $S^1$ is $\gcd(a,b)$.
\end{Observation}
The monodromy acts as the permutation
of the components and we see that the characteristic polynomials
on the  $0$-th homology and the  $1$-th homology is the same. Thus 
the zeta function
is trivial.
The assertion (3) can be shown  in the same way.

\noindent
Let us see the assertion  (2). By the same argument,
\[
\wtl F_\de\cap W_j(\rho)=\{\wtl
f(r_j,\theta_j,s_j,\eta_j)=\de,
\, 1/\rho\ge s_j \ge\rho\}
\]
and by throwing away higher terms, we may consider that $\wtl f$
is again homotopically defined by 
\[
c_{M_t} \,\exp\left((\al_{t1}c_j+\al_{t2}d_j)\theta_j \,i+(\al_{t1}
c_{j+1}+\al_{t2}d_{j+1})\eta_j \,i\right)
\]
Again we see that the Milnor fiber is fibered over the interval
$\{\rho\le s_j\le 1/\rho\}=[\rho,1/\rho]$
 with fiber being a finite union of $S^1$'s. Thus the zeta
function is again trivial. (Recall that $r_{j-1}=1/s_j$.)

\vspace{.2cm}\noindent
(4) Let us consider the fibration restricted on $T_0(\rho)$.
The situation is different from that of (3).
Let $M_0=\Ga(f)\cap \{z_1=0\}$ and assume it comes from the monomial
$z_2^{a_2}|z_2|^{2b_2}$.
The pull back function takes the form:
\[
 \wtl f(u_0,s_0,\eta_0)=c_{M_0}\,s_0^{a_2+2b_2}\exp(a_2\eta_0\, i)+O(s_0^{a_2+2b_2+1})
\]
and throwing away the higher term and putting $c_{M_0}=\tau_0\exp(\xi\, i)$, we see that 
$\wtl F_\de\cap T_0(\rho) $ consists of $a_2$-contractible components:
\[
\wtl F_\de\cap T_0(\rho)= \{(u_0,s_0,\eta_0)\,|\,\tau_0\,s_0^{a_2+2b_2}=\de,
\, \xi+a_2\eta_0\equiv 0\,\, \modulo\, 2\pi\}.
\]

More precisely, 'throwing away' implies  the following standard  discussion.
Consider the  family of functions
\[
  \wtl f_\tau(u_0,s_0,\eta_0):=c_{M_0}\,s_0^{a_2+2b_2}\exp(a_2\eta_0\,
  i)+\tau\,O(s_0^{a_2+2b_2+1}),\,\,0\le \tau\le 1.
\]
In the level of the  original function $f$, this corresponds to the family
$f_\tau=c_{M_0}\bfz^{M_0}+\tau\,(f(\bfz,\bar \bfz)-c_{M_0}\bfz^{M_0})$.
Consider the strata of the respective  Milnor fibers restricted in this 
neighborhood $T_0(\rho)$ and their union:
\[
\begin{split}
\wtl  F_{\de,\tau}&=\{(u_0,s_0,\eta_0)|\wtl f_\tau(u_0,s_0,\eta_0)=\de,\,
(u_0,s_0,\eta_0)\in T_0(\rho) \}\\
\wtl {\mathcal F}_{\de}&=\{(u_0,s_0,\eta_0,\tau)|\wtl f_\tau(u_0,s_0,\eta_0)=\de,\,
(u_0,s_0,\eta_0,\tau)\in T_0(\rho)\times [0,1] \}.
\end{split}
\]
Taking $\de$ sufficiently small, we may assume that 
$\wtl  F_{\de,\tau}$ is smooth and intersects transversely with the
boundary of $T_0(\rho)$ for any $0\le \tau\le 1$. Now we apply the
Ehresmann fibering theorem   (\cite{Wolf1}) to the
projection
$\pi:\wtl {\mathcal F}_{\de}\to [0,1]$ and we conclude that the Milnor fibers
$\wtl  F_{\de,\tau},\,0\le \tau \le 1$ are diffeomorphic to
$\wtl  F_{\de,0}$. (We apply this
argument to each case (1) to (5).)

Thus using the Milnor fiber $\wtl  F_{\de,0}$, we see that
each component is homeomorphic to a disk
$\{u_0\,|\,|u_0|\le \rho\}$,
as the above equation has $a_2$ solutions for $\eta_0$.
The monodromy is acting cyclically among these components.
Thus the zeta function of this restriction is 
$1/(1-t^{|a_2|})$. 

 We see also that 
$\wtl F_\de\cap W_0(\rho)=\emptyset$ if $\de\ll \rho$.

The other edge $T_\ell(\rho)$  gives the term
$1/(1-t^{|a_1|})$.

(5) Now we consider the restriction of $W_{\nu_j}(\rho)$.
Then the principal part takes the form
\[\begin{split}
& \wtl f(r_j,\theta_j,s_j,\eta_j)=\wtl
   f_{\De_j}(r_j,\theta_j,s_j,\eta_j)+O(r^{d(P_{\nu_j})+1})
\end{split}
\]
where $ \ord_{r_j}\, f_{\De_j}(r_j,\theta_j,s_j,\eta_j)=d(P_{\nu_j})$
and the Milnor fibering restricted on this stratum  $W_{\nu_j}(\rho)$
is determined by the principal part
$\wtl f_{\De_j}(r_j,\theta_j,s_j,\eta_j)$.
The last work for us is to determine this contribution.

Consider the curve $C_j=\{f_{\De_j}(\bfz,\bar\bfz)=0\}$
and its polar type resolution by the same mapping $\Phi_p:\cP X\to
\BC^2$.
By the   polar admissibility assumption of the inside vertices,
the Milnor fibration of the second description exists and it is equivalent to the
Milnor fibration
of the first description by Theorem \ref{NIM}.
Then combining the assertions (1) to (4)
applied for $C_j$, we see that 
the above contribution is nothing but the
zeta function of
the monodromy of 
$f_{\De_j}:\BC^{*2}\setminus f_{\De_j}\inv(0)\to \BC^*$,
which is given by 
$(1-t^{m_j})^{-\chi(F_j^*)/m_j}$ as we have seen 
in Lemma \ref{zetaPolar1} and Theorem \ref{NIM}.
Note that
$\wtl F_\de\cap W_0(\rho)=\emptyset$ and 
$\wtl F_\de\cap W_{\ell+1}(\rho)=\emptyset$.
This completes the proof of Theorem \ref{MixAC}.
\subsection{Topology of a  polar weighted polynomial and Kouchnirenko
  type formula}
We consider  a non-degenerate polar weighted mixed polynomial
$f_\De(z_1,z_2,\bar z_1,\bar z_2)$ with
$\De=\overline{AB}$ where $A,B$ are polar admissible simple vertices.
Let $(p_1,p_2;m_\De)$ be the polar weight type. Let $F_\De=f_\De\inv(1)$
be the fiber of the global fibration,
$F_\De^*=F_\De\cap \BC^{*2}$ and let $K_\De=f_\De\inv(0)\cap S^3$.
Note that $F_\De$ is diffeomorphic to the fiber of  the Milnor fibration
$f_\De/|f_\De|: S^3\setminus K_\De\to S^1$ or 
$f_\De:\partial E(r,\de)^*\to S_\de^1$, as $f_\De$ is 
super strongly non-degenerate by Theorem \ref{NIM}.
The Milnor fiber is connected by
 Proposition \ref{0-connected}.
Let $P_i(t)$ be the characteristic polynomial of the monodromy
at the $i$-th homology for $i=0,1$.
Then $P_1(t)=\zeta(t)(1-t)$ as $P_0(t)=(1-t)$.
 We consider the Wang sequence of the Milnor fibration:
\[
0 \to H_2(S^3-K_\De)\to H_1(F_\De)\mapright{h_*-\id}
H_1(F_\De)\to H_1(S^3-K_\De)\to \BZ\to 0.
\]
Put $r_\De^*=\lkn^*(f_\De\inv(0))$. 
Thus $H_0(K_\De)=\BZ^{r_\De^*+\eps(\De)}$
where $\eps(\De)$ is the number of coordinate axes which are a subset of
$f_\De\inv(0)$.
Thus $\eps(\De)=0,1,2$  according to the   two vertices $A,\,B$ are 
either on the axis
or not.
Let $\mu_\De$ and $\mu_\De'$ be the multiplicities  of the factor
 $(t-1)$ in
$P_1(t)$ and  $\zeta(t)$ respectively. Then by the equality 
$P_1(t)=\zeta(t)(1-t)$ and Lemma \ref{zetaPolar} and Remark \ref{zetaPolar2},
\[
 \mu_\De=\mu_\De'+1,\quad \mu_\De'= -\chi(F_\De^*)/m_\De-2+\eps(\De).
\]
On the other hand  by the Alexander duality, we have the isomorphism:
\[
 H_2(S^3-K_\De)\cong H^1(S^3,K_\De)\cong \wtl H^0(K_\De).
\]
As  the monodromy map $h_*$ is periodic,  we have
\[
r_\De^*+\eps(\De)-1= \dim\,{\Ker}\,\{h_*-\id:H_1(F_\De)\to H_1(F_\De)\}
=\mu_\De.
\]
Thus we obtain
\begin{Lemma}\label{lknEuler}
The Euler-Poincar\'e characteristic and the link component number
satisfy the following equality:
\[
 r_\De^*=-\chi(F_\De^*)/m_\De.
\]
\end{Lemma}
Usually it is easier to compute $r_\De^*$ and we can compute
$\chi(F_\De^*)$ by Lemma \ref{lknEuler}.
Now we can state our Kouchnirenko type formula:
\begin{Theorem}\label{MixKouchnirenko}
Let $f(\bfz,\bar\bfz)$ be a non-degenerate convenient mixed
polynomial as in Theorem \ref{MixAC}.
Let $\De_1,\dots, \De_s$ be  faces of $\Ga(f)$ and  we assume that 
$f_{\De_j}(\bfz,\bar\bfz)$ is a polar weighted homogeneous polynomial
with polar degree $m_j$.
Let $r_j=\lkn^*(f_{\De_j}\inv(0))$ for $j=1,\dots, s$.
Then the Milnor number $\mu(F)=b_1(F)$ is given by the formula:
\[
 \mu(F)=\sum_{j=1}^s r_j \,m_j\,-\, |a_1|\,-\,|a_2|\,+\,1.
\]
Here $m_j$ is the polar degree of $f_{\De_j}$ and we assume that
 $m_j>0$.
 $a_1,\,a_2$ are the
 polar sections of $\Ga(f)$ on the respective coordinate axes.
\end{Theorem}
As a special case, the following is a formula for a good polar weighted
mixed polynomial (see \S \ref{good polar} for the definition)
which corresponds to the  Orlik-Milnor formula 
\cite{MilnorOrlik} for a
weighted homogeneous isolated singularity.
\begin{Corollary}
Assume that 
  $f(\bfz,\bar\bfz)$ is  a good polar weighted
 polynomial
which is factored as
\begin{eqnarray}
 &f(\bfz,\bar\bfz)=c
  \prod_{j=1}^k (z_2^a\,
|z_2|^{2a'}-\la_j \,{z_1}^b\,|z_1|^{2b'}),\quad c\ne 0
\end{eqnarray}
with $a\ne 0,\,b\ne 0$.
Let $r=\gcd(|a|,|b|)$.
The polar weight is given by
$P={}^t(p_1\eps_1,p_2\eps_2)$
where
$p_1=|a|/r,\,p_2=|b|/r$,
$\eps_1= b/|b|,\eps_2= a/|a|$ and  the polar degree  $d_p$
is given as $d_p=|a|\,|b|\,k/r$,  $\lkn(f\inv(0))=r\,k$ and 
\[
 \begin{split}
&\mu=|a|\,|b|\,k^2 -k\,(|a|+|b|)+1=(k\,|a|-1)\,(k\,|b|-1)\,\,\text{and}\\
&\zeta(t)=\frac{(1-t^{d_p})^{rk}}{(1-t^{|a|})(1-t^{|b|})}.
\end{split}\]
\end{Corollary}

\subsection{Appendix: Calculation of Example 8.1.2}
We give the detail of the calculation for  Example 8.1.2.
Let 
\begin{eqnarray*}\begin{split}
& f_t(\bfz,\bar\bfz)=-2z_1^2\bar z_1+z_2^2\bar z_2+t\,z_1^2\bar z_2,\,t\in
 \BC\\
&V_t^*:=\{(z_1,z_2)\in \BC^{*2}\,|\, f_t(\bfz,\bar\bfz)=0\}\\
&F_t^*:=\{(z_1,z_2)\in \BC^{*2}\,|\, f_t(\bfz,\bar\bfz)=1\}.
\end{split}
\end{eqnarray*}
and we compute link components.
As $f_t$ is radially weighted, we may assume that 
$|z_2|=1$.
Thus we compute the section with $|z_2|=1$.
We put
\[
z_1=x_1+y_1\,i,\,z_2=x_2+y_2\,i,\,\, x_2=\cos(a),\quad y_2=\sin(\theta),\,
\]
Then $f_t(\bfz,\bar\bfz)=0$ can be rewritten as $f_1=f_2=0$ where
\begin{eqnarray*}
\begin{split}
&f_1=-2\,{x_1}^{3}-2\,x_1\,{y_1}^{2}+ ( \cos
 ( a )  ) ^{3}+\cos ( a )  ( \sin
 ( a )  ) ^{2}+t{x_1}^{2}\cos ( a )\\
&\qquad +
2\,tx_1\,y_1\,\sin ( a ) -t{y_1}^{2}\cos
 ( a ) \\
&f_2=-2\,{x_1}^{2}y_1-2\,{y_1}^{3}+ ( \cos
 ( a )  ) ^{2}\sin ( a ) + ( \sin
 ( a )  ) ^{3}-t{x_1}^{2}\sin ( a )\\
&\qquad  +
2\,tx_1\,y_1\,\cos ( a ) +t{y_1}^{2}\sin
 ( a ) 
\end{split}
\end{eqnarray*}
The resultant $R$ of $f_1$ and $f_2$ in $y_1$ takes the form $R=g_1g_2$
where
\[
 \begin{split}
&{ g_1}=2\,{x_1}^{3}- ( \cos ( a )  ) ^{3
}-t{x_1}^{2}\cos ( a )\\
&{ g_2}={t}^{4}{x_1}^{2}-{t}^{3} ( \sin ( a ) 
 ) ^{2}-2\,\cos ( a ) x_1 \,{t}^{2}+ ( \cos
 ( a )  ) ^{2}+ ( \sin ( a ) 
 ) ^{2}
\end{split}
\]
{\bf $U_1$:} Assume that $t=0$. Then $g_2\equiv 1$. The equation $g_1=0,\,f_1=f_2=0$
has a unique solution 
\[
 \begin{cases}
&x_1=\frac 12 \,{2}^{2/3}\cos ( a ) \\
&y_1=\frac 12 \,{2}^{2/3}\sin ( a ) \\
&x_2=\cos(a),\,\quad y_2=\sin(a),
\end{cases}
\qquad 0\le a\le 2\pi.
\]
This can be also observed by \cite{R-S-V}.

\vspace{.3cm}\noindent
{\bf $U_2$: } Consider the case $t=3$ as a model of $V_t,\,t\in U_3$.
 First, $g_1,g_2$ takes the following form.
\[
 \begin{split}
&g_1=2\,{x_1}^{3}- ( \cos ( a )  ) ^{3}-3\,{{
\it x_1}}^{2}\cos ( a ) \\
&g_2=81\,{x_1}^{2}-26\, ( \sin ( a )  ) ^{2}-18
\,\cos ( a ) x_1 + ( \cos ( a ) 
 ) ^{2}
\end{split}
\]
Over $g_1=0$, we have one component parametrized as
\[
 \begin{split}
&x_1= ( \frac 12 \,\sqrt [3]{3+2\,\sqrt {2}}+\frac 12 \,{\frac {1}{\sqrt [3]{3+2\,
\sqrt {2}}}}+\frac 12  ) \cos ( a ) \\
&y_1=-{\frac {\sin ( a ) }{ ( 3+2\,\sqrt {2} ) ^{2/3}
- ( 3+2\,\sqrt {2} ) ^{2/3}\sqrt {2}-\sqrt [3]{3+2\,\sqrt {
2}}+\sqrt [3]{3+2\,\sqrt {2}}\sqrt {2}}}\\
& 0\le a\le 2\pi.
\end{split}
\]
Over $g_2=0$, we have two components parametrized as
\[
 \begin{split}
&x_1=\frac 19 \,\cos ( a ) \pm \frac 19 \,\sqrt {26}\sin ( a  )\\
&y_1=
\frac {1}{234}\, ( \sqrt {26}\sin ( a ) \pm 26\,\cos
 ( a )  ) \sqrt {26},\,\,0\le a\le 2\pi.\\
\end{split}
\]
Thus we have shown that $\lkn (V_0)=1$ and $\lkn(V_3)=3$.

It is my pleasure to thank to the referee for the careful checking of
the first draft and a nice suggestion to make our paper more understandable.
\def\cprime{$'$} \def\cprime{$'$} \def\cprime{$'$} \def\cprime{$'$}
  \def\cprime{$'$} \def\cprime{$'$} \def\cprime{$'$} \def\cprime{$'$}

\end{document}